\documentclass[11pt]{amsart}

\usepackage[T1]{fontenc} 
\usepackage[utf8]{inputenc}           
\usepackage[english]{babel}          
\usepackage{amsmath}      
\usepackage{amsfonts}                
\usepackage{amssymb}                
\usepackage{amsthm}
\usepackage{enumitem}
\usepackage{mymacros}
\usepackage{hyperref}

\newcommand{\graded}{standard graded}

\theoremstyle{definition}
\newtheorem*{ack}{Acknowledgements}

\DeclareMathOperator{\Ad}{Ad}
\DeclareMathOperator{\Aut}{Aut}

\renewcommand{\MR}[1]{}

\title[Isomorphism problem for Nevanlinna-Pick spaces]{On the isomorphism problem for multiplier algebras of Nevanlinna-Pick spaces}
\author{Michael Hartz}
\address{Department of Pure Mathematics, University of Waterloo, Waterloo, ON N2L 3G1, Canada}
\email{mphartz@uwaterloo.ca}
\thanks{The author is partially supported by an Ontario Trillium Scholarship.}

\begin{document}

\begin{abstract}
  We continue the investigation of the isomorphism problem for multiplier algebras of reproducing kernel
  Hilbert spaces with the complete Nevanlinna-Pick property. In contrast to previous work in this area,
  we do not study these spaces by identifying them with restrictions of a universal space, namely the Drury-Arveson space.
  Instead, we work directly with the Hilbert spaces and their reproducing kernels. In particular,
  we show that two multiplier algebras of Nevanlinna-Pick spaces on the same set are equal if and only if the Hilbert
  spaces are equal. Most of the article is devoted to the study of a special class of
  complete Nevanlinna-Pick spaces on homogeneous varieties. We provide a complete
  answer to the question of when two multiplier algebras of spaces of this type
  are algebraically or isometrically isomorphic. This generalizes results of Davidson, Ramsey, Shalit,
  and the author.
\end{abstract}

\maketitle

\section{Introduction}
\label{sec:intro}

We continue the study of the isomorphism problem for multiplier algebras
of complete Nevanlinna-Pick spaces.
This problem was studied in \cite{DHS15,DRS11,DRS15,Hartz12,KMS13} by making use of
a theorem of Agler and McCarthy \cite{AM00} to
identify a given complete Nevanlinna-Pick space with a restriction
of a special universal space, namely the Drury-Arveson space, to an analytic variety.
Roughly speaking, the results then typically state that
two algebras are isomorphic if and only if
the underlying varieties are geometrically equivalent in a suitable sense.
For an up-to-date account on these results, the reader is referred to the recent survey article \cite{SS14}.

While this approach has been successful in dealing with the (completely) isometric isomorphism
problem (see \cite{DRS15} and \cite{SS14}), the algebraic (or even completely bounded)
isomorphism problem seems to be more difficult. Essentially the only instance for which
the algebraic isomorphism problem has been completely resolved is the case of restrictions
of Drury-Arveson space on a finite dimensional ball to homogeneous varieties \cite{DRS11,Hartz12}. The existence
of algebraic isomorphisms is also quite well understood for multiplier algebras associated to certain one-dimensional varieties under the
assumption of sufficient regularity on the boundary \cite{APV03,ARS08,KMS13}. For more general varieties,
however, the situation is far less clear.
Moreover, several results in \cite{DRS15} only apply to varieties which are contained
in a finite dimensional ball. From the point of view of the study
of multiplier algebras of complete Nevanlinna-Pick spaces, this condition is rather restrictive.
There are many natural examples of complete Nevanlinna-Pick spaces on the unit disc or, more generally,
on a finite dimensional unit ball, which
cannot be realized as the restriction of Drury-Arveson space on a finite dimensional ball. Indeed,
the classical Dirichlet space, which consists of analytic functions on the unit disc, is such an example (see
also Proposition \ref{prop:finite_embedding}).

In this note, we take a different point of view and study the complete Nevanlinna-Pick spaces and their
reproducing kernels directly.
In particular, we consider a class of
spaces on homogeneous varieties in a ball in $\bC^d$.
This more direct approach has the disadvantage that we can no longer make use of the well developed theory
of the Drury-Arveson space. In particular, the tools coming from the non-commutative
theory of free semigroup algebras \cite{DP98a,DP98,DP99} are not available any more.

Nevertheless, the direct approach has certain benefits. Firstly, by studying the spaces
directly, we are able to stay within the realm of reproducing kernel Hilbert spaces on subsets
of $\bC^d$ for finite $d$. We thus avoid the issues surrounding the Drury-Arveson
space $H^2_\infty$ on an infinite dimensional ball, such as the extremely complicated nature
of the maximal ideal space of $\Mult(H^2_\infty)$ (cf. \cite{DHS15a}). Secondly, many spaces of interest
are graded in a natural way. Indeed, we consider a class of complete Nevanlinna-Pick spaces of
analytic functions on the open unit ball $\bB_d$ in $\bC^d$ which contain the polynomials as a dense subspace
and in which homogeneous polynomials of different degree are orthogonal.
When identifying such a space with a restriction of the Drury-Arveson space, the grading becomes
less visible, since it is usually not compatible with the natural grading on the Drury-Arveson space.
By working with the spaces directly, we are able to exploit their graded nature.
Finally, when working with two spaces on the same set,
one can also ask if their multiplier algebras are equal, rather than just isomorphic.

In addition to this introduction, this article has ten sections. In Section \ref{sec:prelim},
we gather some preliminaries regarding complete Nevanlinna-Pick spaces.

In Section \ref{sec:mult_to_kernel}, we observe that it is possible to recover the reproducing kernel
of a complete Nevanlinna-Pick space from its multiplier algebra. As a consequence,
we obtain that two complete Nevanlinna-Pick spaces whose multiplier algebras are equal
have the same reproducing kernels,
up to normalization.

In Section \ref{sec:composition_operators}, we apply the results of Section 2 to composition operators on multiplier algebras.
In particular, we characterize those complete Nevanlinna-Pick spaces of analytic functions
on $\bB_d$ whose multiplier algebras are isometrically invariant under conformal automorphisms
of $\bB_d$.

In Section \ref{sec:alg_cons}, we study the notion of algebraic consistency, which, roughly speaking, assures
that the functions in a complete Nevanlinna-Pick space are defined on the largest possible domain of definition. It turns
out that this notion is closely related to the notion of a variety from \cite{DRS15}.

In Section \ref{sec:graded_spaces}, we consider a general notion of grading on a complete Nevanlinna-Pick space.
The main result in this section asserts
that multiplier norm and Hilbert space norm coincide for homogeneous elements.

In Section \ref{sec:unit_inv}, we set the stage for the remainder of this article by introducing
a family of unitarily invariant complete Nevanlinna-Pick spaces
on $\bB_d$.
The aim is then to investigate the isomorphism problem for multiplier algebras of restrictions
of such spaces to homogeneous varieties. This is done by following the route of
\cite{DRS11}.

In Section \ref{sec:maximal_ideal_space}, we study the maximal ideal spaces of the multiplier algebras of spaces
introduced in Section \ref{sec:unit_inv}. In particular, we introduce a regularity condition on the maximal
ideal space, which we call \emph{tameness}. It is shown that a large collection of spaces, which includes
the spaces $\cH_s$ from \cite{DHS15} and their counterparts on $\bB_d$, is indeed tame.

In Section \ref{sec:hol_maps_on_varieties},
we recall several results from \cite{DRS11} about holomorphic maps on homogeneous varieties,
thereby providing simpler proofs in some instances. We also point out that a crucial argument
from \cite{DRS11} can be used to show that the group of unitaries is a maximal subgroup
of the group of conformal automorphisms of $\bB_d$.

In Section \ref{sec:existence_graded}, we show that the arguments from \cite{DRS11} can be adapted to our setting to
show that if two of our multiplier algebras are isomorphic, then they are isomorphic
via an isomorphism which preserves the grading.

Finally, Section \ref{sec:iso_results} contains the main results about isometric and algebraic isomorphism
of the multiplier algebras. We finish by reformulating some of the results in terms of
restrictions of Drury-Arveson space, thereby providing a connection to examples in \cite{DHS15}.

\section{Preliminaries}
\label{sec:prelim}

We begin by recalling several notions from the theory of reproducing kernel Hilbert spaces and Nevanlinna-Pick
interpolation. Most of the material outlined in this section can be found in the book \cite{AM02}.
Let $\cH$ be a reproducing kernel Hilbert space on a set $X$ with reproducing kernel $K$. Although it is not
essential, we will always assume for convenience that our reproducing kernel Hilbert spaces are separable.
Let $\Mult(\cH)$ denote the multiplier algebra of $\cH$. Suppose that $n$ is a positive natural number and that
we are given points $z_1,\ldots,z_n \in X$ and $\lambda_1,\ldots,\lambda_n \in \bC$.
If there exists a multiplier $\varphi \in \Mult(\cH)$ of norm at most $1$ with
\begin{equation*}
  \varphi(z_i) = \lambda_i \quad (i= 1,\ldots,n),
\end{equation*}
then the matrix
\begin{equation*}
  \Big((1 - \lambda_i \ol{\lambda_j}) K(z_i,z_j) \Big)_{i,j=1}^n
\end{equation*}
is positive. If the converse of this statement holds,
we say that $\cH$ satisfies the \emph{$n$-point Nevanlinna-Pick property}.
We say that $\cH$ is a \emph{Nevanlinna-Pick space} if it
satisfies the $n$-point Nevanlinna-Pick property for every $n \in \bN$.
Finally, if the analogous result for matrix-valued
interpolation holds, $\cH$ is said to be a \emph{complete Nevanlinna-Pick space}. While the title of this note simply refers to
Nevanlinna-Pick spaces for the sake of brevity, we will in fact mostly be concerned with complete Nevanlinna-Pick spaces.
The prototypical
example of a complete Nevanlinna-Pick space is the Hardy space $H^2(\bD)$ on the unit disc. For a detailed account on this topic,
we refer the reader to Chapter 5 of the book \cite{AM02}.

Most of the reproducing kernel Hilbert spaces we will consider will be \emph{irreducible}
in the following sense: For any two points $x,y \in X$, the reproducing kernel $K$ satisfies
$K(x,y) \neq 0$, and if $x \neq y$, then
$K(\cdot,x)$ and $K(\cdot,y)$ are linearly independent in $\cH$ (see, for example, Section 7.1 in \cite{AM02}).
We say that $K$ is \emph{normalized} at the point $x_0 \in X$ if $K(x,x_0) = 1$ for all $x \in X$.
If $K$ is normalized at some point in $X$, we say that $\cH$ is \emph{normalized}.

Irreducible complete Nevanlinna-Pick spaces are characterized by a theorem of McCullough and Quiggin (see, for
example, Section 7.1 in \cite{AM02}).
We require the following version of Agler and McCarthy, which is \cite[Theorem 7.31]{AM02}.
Recall that a function $F: X \times X \to \bC$ is said to be \emph{positive definite} if for any
finite collection $\{x_1,\ldots,x_n\}$ of points in $X$, the matrix
\begin{equation*}
  \Big(F(x_i,x_j)\Big)_{i,j=1}^n
\end{equation*}
is positive semidefinite.

\begin{thm}[McCullough-Quiggin, Agler-McCarthy]
  \label{thm:CNP_char}
  Let $\cH$ be an irreducible reproducing kernel Hilbert space on a set $X$ with reproducing kernel $K$
  which is normalized at point in $X$.
  Then $\cH$ is a complete Nevanlinna-Pick space if and only if the
  Hermitian kernel $F = 1 - 1/K$ is positive definite.
\end{thm}

If $\cH$ is a reproducing kernel Hilbert space on a set $X$ with kernel $K$, and if $Y \subset X$,
then $\cH \big|_Y$ denotes the reproducing kernel Hilbert space on $Y$ with kernel $K \big|_{Y \times Y}$.
Equivalently,
\begin{equation*}
  \cH \big|_Y = \{ f \big|_Y: f \in \cH\},
\end{equation*}
and the norm on $\cH \big|_Y$ is defined such that the map
\begin{equation*}
  \cH \to \cH \big|_Y, \quad f \mapsto f \big|_Y,
\end{equation*}
is a co-isometry (see \cite[Section 5]{Aronszajn50}). The definitions immediately imply that
restriction map
\begin{equation*}
  \Mult(\cH) \to \Mult(\cH \big|_Y), \quad \varphi \mapsto \varphi \big|_Y,
\end{equation*}
is a contraction. It is not hard to see that if $\cH$ is a Nevanlinna-Pick space, then this map is
a quotient map, and in particular surjective.

Our main source of examples is a class of spaces on the open unit ball $\bB_d$ in $\bC^d$. Occasionally, we will
allow $d = \infty$, in which case $\bC^d$ is understood to be $\ell_2$. To be precise,
by a \emph{unitarily invariant space on $\bB_d$}, we mean a reproducing kernel
Hilbert space $\cH$ on $\bB_d$ with reproducing kernel $K$ which is normalized at $0$, analytic in the first component,
and satisfies
\begin{equation*}
  K(Uz, Uw) = K(z,w)
\end{equation*}
for all $z,w \in \bB_d$ and all unitary maps $U$ on $\bC^d$.
Spaces of this type appear throughout the literature, see for example \cite[Section 4]{GHX04} or \cite[Section 4]{GRS02}.
The following characterization of unitarily invariant spaces is well known.
Since we do not have a convenient reference
for the proof, it is provided below.

\begin{lem}
  \label{lem:unitarily_invariant_power_series}
  Let $d \in \bN \cup \{\infty\}$ and let $K: \bB_d \times \bB_d \to \bC$ be a function. The following are equivalent:
    \begin{enumerate}[label=\normalfont{(\roman*)}]
      \item $K$ is a positive definite kernel which is normalized at $0$, analytic in the first component,
        and satisfies $K(z,w) = K(U z, U w)$ for all $z,w \in \bB_d$ and all unitary maps $U$ on $\bC^d$.
      \item There is a sequence $(a_n)_n$ of non-negative real numbers with $a_0 = 1$ such that
        \begin{equation*}
          K(z,w) = \sum_{n=0}^\infty a_n \langle z,w \rangle^n
        \end{equation*}
        for all $z,w \in \bB_d$.
    \end{enumerate}
\end{lem}

\begin{proof}

  (ii) $\Rightarrow$ (i) By the Schur product theorem, the map $(z,w) \mapsto \langle z,w \rangle^n$ is positive
  definite for all $n \in \bN$, hence $K$ is positive definite. Clearly, $K$ is normalized at $0$ and
  invariant under unitary maps of $\bC^d$. Moreover, for fixed $w \in \bB_d$, the series in (ii) converges
  uniformly in $z$ on $\bB_d$, hence $K$ is analytic in the first variable.

  (i) $\Rightarrow$ (ii)
  Let $z_1,w_1,z_2,w_2 \in \bB_d$ satisfy $\langle z_1,w_1 \rangle = \langle z_2,w_2 \rangle$. We will show
  that $K(z_1,w_1) = K(z_2,w_2)$. This will complete the proof, since then, there
  exists a function $f: \bD \to \bC$ such that $K(z,w) = f(\langle z,w \rangle)$
  for all $z,w \in \bB_d$. Since $K$ is analytic in the first component and is normalized at the origin, $f$ is
  necessarily analytic and satisfies $f(0) = 1$. Positive definiteness of $K$ finally implies that
  the Taylor coefficients of $f$ at $0$ are non-negative, see the proof of \cite[Theorem 7.33]{AM02} and also Corollary \ref{cor:powerseries_kernel} below.

  In order to show that $K(z_1,w_1) = K(z_2,w_2)$, first note that for $z,w \in \bB_d$, the identity
  \begin{equation*}
    K(\lambda z, w) = K(z, \ol{\lambda} w)
  \end{equation*}
  holds for all $\lambda \in \bT$, as multiplication by a complex scalar of modulus $1$ is a unitary map on $\bC^d$.
  Since $K(z,\ol{\lambda} w) = \ol{K(\ol{\lambda} w,z)}$, we see that both sides of the above equation define
  analytic maps in $\lambda$ in an open neighbourhood of $\ol{\bD}$, hence the above identity holds for all $\lambda \in \ol{\bD}$.
  In particular, we see that
  \begin{equation*}
    K(r z, w) = K(z, r w)
  \end{equation*}
  for $z,w \in \bB_d$ and $r \in [0,1]$.
  Consequently, we may without loss of generality
  assume that $||w_1|| = ||w_2||$. Then there exists a unitary map on $\bC^d$ which maps $w_1$ onto $w_2$. Since
  $K$ is invariant under unitary maps by assumption, and so is the scalar product $\langle \cdot,\cdot \rangle$, we 
  may in fact suppose that $w_1 = w_2$. Let $w$ denote this vector. Since $K$ is normalized at $0$,
  the claim is obvious if $w = 0$, so assume that $w \neq 0$.

  From the assumption $\langle z_1,w \rangle = \langle z_2 ,w \rangle$, we deduce that there exist vectors $v,r_1,r_2 \in \bC^d$
  such that $v \in \bC w$ and $r_1,r_2 \in (\bC w)^\bot$ and such that
  \begin{equation*}
    z_i = v + r_i \quad (i=1,2).
  \end{equation*}
  For $\lambda \in \bT$, let $U_\lambda$ denote the unitary map on $\bC^d$ which fixes $\bC w$ and acts
  as multiplication by $\lambda$ on $(\bC w)^\bot$. Then for $i=1,2$ and $\lambda \in \bT$, we have
  \begin{equation*}
    K(z_i,w)= K(U_\lambda z_i, U_\lambda w) = K(v + \lambda r_i, w).
  \end{equation*}
  Observe that the right-hand side defines an analytic function in $\lambda$ in an open neighbourhood of $\ol{\bD}$, which is
  therefore constant. In particular,
  \begin{equation*}
    K(z_1,w) = K(v,w) = K(z_2,w),
  \end{equation*}
  which completes the proof.
\end{proof}

If $\cH$ is a unitarily invariant space on $\bB_d$, then it easily follows from the representation
of the kernel in part (ii) of the preceding lemma that convergence in $\cH$ implies
uniform convergence on $r \bB_d$ for $0 < r < 1$. Since the kernel functions $K(\cdot,w)$ for $w \in \bB_d$
are analytic by assumption, and since finite linear combinations of kernel functions are dense in $\cH$,
we therefore see that every function in $\cH$ is analytic on $\bB_d$.

We also require the following straightforward generalization of \cite[Theorem 7.33]{AM02}.
\begin{lem}
  \label{lem:NP_a_b}
  Let $d \in \bN \cup \{\infty\}$ and let $\cH$ be a unitarily invariant space on $\bB_d$ with reproducing kernel
  \begin{equation*}
    K(z,w) = \sum_{n=0}^\infty a_n \langle z,w \rangle^n,
  \end{equation*}
  where $a_0 = 1$. Assume that $a_1 > 0.$ Then the following are equivalent:
  \begin{enumerate}[label=\normalfont{(\roman*)}]
    \item $\cH$ is an irreducible complete Nevanlinna-Pick space.
    \item The sequence $(b_n)_{n=1}^\infty$ defined by
      \begin{equation*}
        \sum_{n=1}^\infty b_n t^n = 1 - \frac{1}{\sum_{n=0}^\infty a_n t^n}
      \end{equation*}
      for $t$ in a neighbourhood of $0$ is a sequence of non-negative real numbers.
  \end{enumerate}
  In particular, if (ii) holds, then $\cH$ is automatically irreducible.
\end{lem}

\begin{proof}
  Observe that
  \begin{equation*}
    1- \frac{1}{K(z,w)} = \sum_{n=1}^\infty b_n \langle z,w \rangle^n.
  \end{equation*}
  It is known that this kernel is positive if and only if $b_n \ge 0$ for all $n \ge 1$
  (see the proof of \cite[Theorem 7.33]{AM02}, and also Corollary \ref{cor:powerseries_kernel} below).
  Consequently, the implication (i) $\Rightarrow$ (ii) follows from Theorem \ref{thm:CNP_char},
  and
  (ii) $\Rightarrow$ (i) will follow from the same result, once we observe that $\cH$ is irreducible
  in the setting of (ii).

  Since $a_0 = 1$ and $a_1 > 0$, the space $\cH$ contains the constant
  function $1$ and the coordinate functions (see \cite[Proposition 4.1]{GHX04} or \cite[Section 4]{GRS02}),
  from which it readily follows that $K(\cdot,x)$ and $K(\cdot,y)$ are linearly independent
  if $x \neq y$. We finish the proof by showing that $\sum_{n=0}^\infty a_n t^n$ never vanishes on $\bD$.
  Assume toward a contradiction that $t_0 \in \bD$ is a zero of $\sum_{n=0}^\infty a_n t^n$
  of minimal modulus. Then the equality in (ii) holds for all $t \in \bD$ with $|t| < |t_0|$,
  and $t_0$ is a pole of $\sum_{n=1}^\infty b_n t^n$. Since $b_n \ge 0$ for $n \ge 1$, this implies
  that $|t_0|$ is a pole of $\sum_{n=1}^\infty b_n t^n$, and consequently $|t_0|$ is a zero
  of $\sum_{n=0}^\infty a_n t^n$. This is a contradiction, since $a_0 = 1$ and $a_n \ge 0$ for
  $n \ge 0$, and the proof is complete.
\end{proof}

Perhaps the most important example of a unitarily invariant complete Nevanlinna-Pick space is the Drury-Arveson space
$H^2_m$ on $\bB_m$, where $m \in \bN \cup \{\infty\}$.
This space corresponds to the choice $a_n = 1$ for all $n \in \bN$ above, hence its reproducing
kernel is given by
\begin{equation*}
  k_m(z,w) = \frac{1}{1 - \langle z,w \rangle }.
\end{equation*}
The following theorem of Agler and McCarthy \cite{Am00a} (see also \cite[Theorem 8.2]{AM02}) asserts
that $H^2_m$ is a universal complete Nevanlinna-Pick space.

\begin{thm}[Agler-McCarthy]
  \label{thm:DA_univ}
  If $\cH$ is a normalized irreducible complete Nevanlinna-Pick space on a set $X$ with kernel $K$, then
  there exists $m \in \bN \cup \{\infty\}$ and an embedding $j: X \to \bB_m$ such that
  \begin{equation*}
    K(z,w) = k_m(j(z),j(w)) \quad (z,w \in X).
  \end{equation*}
  In this case, $f \mapsto f \circ j$ defines a unitary operator from $H^2_m \big|_{j(X)}$
  onto $\cH$.
\end{thm}
In this setting, we say that $j$ is \emph{an embedding for $\cH$}.

\section{From multiplier algebras to kernels}
\label{sec:mult_to_kernel}

We begin by observing that the kernel of a Nevanlinna-Pick space can be recovered from the
isometric structure of its multiplier algebra. Results similar to the next proposition
are well known, see for example \cite{GRS02} and \cite{AM02}, especially Exercise 8.35. Since we do not
have a reference for the exact statement, a complete proof is provided.

\begin{prop}
  \label{prop:kernel_recover_iso}
  Let $\cH$ be an irreducible reproducing kernel Hilbert space on a set $X$ with kernel $K$.
  Suppose that $K$ is normalized at $x_0 \in X$ and satisfies the two-point Nevanlinna-Pick property.
  Then
  \begin{equation*}
    \sup \{ \Re \varphi(w) : ||\varphi||_{\Mult(\cH)} \le 1 \text { and } \varphi(x_0) = 0 \} = \Big( 1 - \frac{1}{K(w,w)} \Big)^{1/2}
  \end{equation*}
  for every $w \in X$, and this number is strictly positive if $w \neq x_0$.
  Moreover, there is a unique multiplier $\varphi_w$ which achieves the supremum if $w \neq x_0$, namely
  \begin{equation*}
    \varphi_w(z) = \frac{1 - \frac{1}{K(z,w)}}{\sqrt{1 - \frac{1}{K(w,w)}}}.
  \end{equation*}
  Equivalently,
  \begin{equation*}
    K(z,w) = \frac{1}{1 - \varphi_w(z) \varphi_w(w)}.
  \end{equation*}
\end{prop}

\begin{proof}
  By the two-point Nevanlinna-Pick property, there exists a contractive multiplier $\varphi$ with
  $\varphi(x_0) = 0$ and $\varphi(w) = \lambda$ if and only if
  the Pick matrix at points $(x_0, w)$,
  \begin{equation*}
    \begin{pmatrix}
      1 & 1 \\ 1 & K(w,w) (1 - |\lambda|^2)
    \end{pmatrix},
  \end{equation*}
  is positive, which, in turn, happens if and only if
  \begin{equation*}
    K(w,w) \ge \frac{1}{1 - |\lambda|^2}.
  \end{equation*}
  This proves the formula for the supremum. Moreover, we see that the supremum
  is actually attained.

  Irreducibility of $\cH$ implies that $K(w,w) > 1$ if $w \neq x_0$. Indeed, since $K$
  is normalized at $x_0$, we have
  \begin{equation*}
    1 = K(x_0,w) = |\langle K(\cdot,w), K(\cdot,x_0) \rangle| \le K(w,w)^{1/2}
  \end{equation*}
  by Cauchy-Schwarz,
  with equality occurring only if $K(\cdot,w)$ and $K(\cdot,x_0)$ are linearly dependent. Since $\cH$
  is irreducible, this only happens if $w = x_0$.

  Let $\varphi = \varphi_w$ be any multiplier which achieves the supremum.
  If $z \in X$ is arbitrary, then the Pick matrix at points $(x_0, w ,z)$,
  \begin{equation*}
    \begin{pmatrix}
      1 & 1 & 1 \\
      1 & 1 & K(w,z) (1 - \varphi(w) \overline{\varphi(z)}) \\
      1 & K(z,w) (1 - \varphi(z) \overline{\varphi(w)}) & K(z,z) (1 - |\varphi(z)|^2)
    \end{pmatrix},
  \end{equation*}
  is positive, since $||\varphi||_{\Mult(\cH)} \le 1$ (observe that the three-point Nevanlinna-Pick property
  is not needed for this implication).
  The determinant of this matrix is
  \begin{equation*}
    -|1 - K(z,w) (1 - \varphi(z) \ol{\varphi(w)})|^2,
  \end{equation*}
  hence
  \begin{equation*}
    K(z,w) (1 - \varphi_w(z) \overline{\varphi_w(w)}) = 1.
  \end{equation*}
  Since
  \begin{equation*}
    \varphi_w(w) = \Big(1 - \frac{1}{K(w,w)} \Big)^{1/2},
  \end{equation*}
  the formula for $\varphi_w$ follows. In particular, $\varphi_w$ is unique if $w \neq x_0$.
\end{proof}

The following consequence, which generalizes Section 5.4 in \cite{AM02}, is immediate.

\begin{cor}
  \label{cor:multiplier_equal}
  Let $\cH_1$ and $\cH_2$ be two irreducible Nevanlinna-Pick spaces
  on the same set $X$, with kernels $K_1$ and $K_2$, respectively, which are
  normalized at a point $x_0 \in X$.
  Then the following are equivalent:
  \begin{enumerate}[label=\normalfont{(\roman*)}]
    \item $\Mult(\cH_1) = \Mult(\cH_2)$ isometrically.
    \item $\Mult(\cH_1) = \Mult(\cH_2)$ completely isometrically.
    \item $\cH_1 = \cH_2$ isometrically.
    \item $K_1 = K_2$.
  \end{enumerate}
\end{cor}

\begin{proof}
  The implications (iv) $\Rightarrow$ (iii) $\Rightarrow$ (ii) $\Rightarrow$ (i) are clear.
  The implication (i) $\Rightarrow$ (iv) follows from the preceding proposition.
\end{proof}

Observe that the last result is generally false without the assumption that both spaces
are Nevanlinna-Pick spaces. Indeed, the Hardy space and the Bergman space on the unit disc
both have $H^\infty(\bD)$ as their multiplier algebra.

We can also use Proposition \ref{prop:kernel_recover_iso} to show that certain
algebras of functions are not multiplier algebras
of complete Nevanlinna-Pick spaces. For $H^\infty(\bB_d)$, this is done in Proposition 8.83 in \cite{AM02}.

\begin{cor}
  There is no irreducible reproducing kernel Hilbert space on $\bD^d$ for $d \ge 2$ which satisfies the two-point Nevanlinna-Pick property
  and whose multiplier algebra is $H^\infty(\bD^d)$.
\end{cor}

\begin{proof}
  Let $a= 1/2$ and let $w=(a,a,0,\ldots) \in \bD^d$. Let $\varphi \in H^\infty(\bD^d)$
  be non-constant with $||\varphi||_\infty \le 1$ and $\varphi(0) = 0$.
  Then $z \mapsto \varphi(z,z,0,\ldots)$ defines an analytic map from $\bD$ into $\bD$
  which fixes the origin, hence $|\varphi(a)| \le 1/2$ by the Schwarz lemma. In particular,
  \begin{equation*}
    \sup \{ \Re (\varphi(a)): ||\varphi||_{H^\infty(\bD^d)} \le 1 \text{ and }
  \varphi(0) = 0 \} \le \frac{1}{2}.
  \end{equation*}
  But there are several functions which realize the value $1/2$, for example the coordinate
  projections $z_1$ and $z_2$. In particular, the extremal problem with normalization point $0$
  in Proposition \ref{prop:kernel_recover_iso} does not have a unique solution. Since every
  non-vanishing kernel can be normalized at an arbitrary point without changing the multiplier algebra
  (see \cite[Section 2.6]{AM02}), it follows that
  $H^\infty(\bD^d)$ is not the multiplier algebra of an irreducible reproducing kernel Hilbert space
  which satisfies the two-point Nevanlinna-Pick property.
\end{proof}

We now consider a second way of recovering the reproducing kernel of a complete
Nevanlinna-Pick space from its multiplier algebra.
In contrast to Proposition \ref{prop:kernel_recover_iso},
this approach uses the operator space structure of the multiplier algebra.

If $\varphi: X \to \cB(\cE,\bC)$ is a function with $||\varphi(x)|| < 1$ for all $x \in X$,
where $\cE$ is an auxiliary Hilbert space, we define a kernel $K_\varphi$ on $X$ by
\begin{equation*}
  K_\varphi (z,w) = \frac{1}{1-\varphi(z)\varphi(w)^*}.
\end{equation*}
Expressing the last identity as a geometric series, we see that $K_\varphi$ is positive definite.

\begin{prop}
  \label{prop:kernel_recover_ci}
  Let $\cH$ be an irreducible reproducing kernel Hilbert space on a set $X$ with kernel $K$, normalized at $x_0 \in X$.
  Then $K$ is an upper bound for the set
  \begin{equation*}
    \Big\{ K_\varphi:  \varphi \in \Mult(\cH \otimes \cE, \cH) \text{ with } 
    ||M_\varphi|| \le 1 \text{ and } \varphi(x_0) = 0 \}
  \end{equation*}
  with respect to the partial order given by positivity. Moreover, $K$ is the maximum of this set if and only if
  $\cH$ is a complete Nevanlinna-Pick space.
\end{prop}

\begin{proof}
  We first observe that every $\varphi$ as in the proposition maps $X$ into the open unit ball of $\cB(\cE,\bC)$. To this end,
  let $x \in X$, and consider the Pick matrix associated to $\{x_0, x\}$. Since $K$ is normalized at $x_0$, and since
  $\varphi(x_0)=0$, we obtain
  \begin{equation*}
    \begin{pmatrix}
      1 & 1 \\ 1 & K(x,x) (1-\varphi(x) \varphi(x)^*)
    \end{pmatrix},
  \end{equation*}
  so this matrix is positive. In particular, the (2,2)-entry is necessarily bounded above by $1$, so that
  \begin{equation*}
    ||\varphi(x)||^2 \le 1 - \frac{1}{K(x,x)} < 1.
  \end{equation*}

  Now, if $\varphi: X \to \cB(\cE,\bC)$ is a multiplier of norm at most $1$, then $K / K_\varphi$ is positive definite
  by a well-known characterization of contractive multipliers.
  If in addition $\varphi(x_0) = 0$, then $K_\varphi$ is normalized at $0$, hence so is $K / K_\varphi$.
  This implies that $K/K_\varphi - 1$ is positive definite (see, for example, the proof of Corollary 4.2 in \cite{GRS05}),
  and thus also
  \begin{equation*}
    K - K_\varphi = K_\varphi \Big( \frac{K}{K_\varphi} - 1 \Big)
  \end{equation*}
  is positive definite
  by the Schur product theorem. Consequently, $K_\varphi \le K$.

  If $K$ belongs to the set in the statement of the proposition, then $1 - 1/K$ is positive definite, so
  $\cH$ is a complete Nevanlinna-Pick space by Theorem \ref{thm:CNP_char}.

  Assume now that $\cH$ is a complete Nevanlinna-Pick space, so that we can write
  \begin{equation*}
    K(z,w) = \frac{1}{1- \langle b(z),b(w) \rangle}
  \end{equation*}
  for some function $b: X \to \bB_\infty$ by Theorem \ref{thm:DA_univ}.
  Consider for $z \in X$ the row operator
  \begin{equation*}
    \varphi(z) = (b_1(z), b_2(z), \ldots) \in \cB(\ell^2, \bC),
  \end{equation*}
  where the $b_i$ are the coordinate functions of $b$.
  Since
  \begin{equation*}
    K(z,w)  (1 - \varphi(z) \varphi(w)^*) = 1,
  \end{equation*}
  we have $\varphi \in \Mult(\cH \otimes \ell^2, \cH)$ with $||\varphi|| \le 1$, and $K = K_\varphi$.
  Also, $\varphi(x_0) = 0$ since $K$ is normalized at $x_0$.
\end{proof}

One advantage of this second approach is that we also obtain information about inclusions
of multiplier algebras.
\begin{cor}
  Let $\cH_1$ and $\cH_2$ be reproducing kernel Hilbert spaces on the same set $X$ with kernels
  $K_1$ and $K_2$, respectively. Assume that $\cH_1$ is an irreducible complete Nevanlinna-Pick space,
  and suppose that $K_1$ and $K_2$ are both normalized at $x_0 \in X$. Then the following are equivalent:
  \begin{enumerate}[label=\normalfont{(\roman*)}]
    \item $\Mult(\cH_1) \subset \Mult(\cH_2)$, and the inclusion map is a complete contraction.
    \item $K_2 / K_1$ is positive definite.
  \end{enumerate}
  In this case, $\cH_1 \subset \cH_2$, and the inclusion map is a contraction.
\end{cor}

\begin{proof}
  (i) $\Rightarrow$ (ii) Proposition \ref{prop:kernel_recover_ci} yields a multiplier
  $\varphi \in \Mult(\cH_1 \otimes \cE, \cH_1)$ with
  $||\varphi||_{\Mult(\cH_1 \otimes \cE, \cH_1)} \le 1$ such
  that $K_1 = K_\varphi$. By assumption, $\varphi$ is a multiplier on $\cH_2$ of norm at most $1$, hence
  $K_2 / K_1 = K_2 / K_\varphi$ is positive. Moreover, another application of the proposition shows that
  \begin{equation*}
    K_1 = K_\varphi \le K_2,
  \end{equation*}
  so that $\cH_1 \subset \cH_2$ and the inclusion map is a contraction.

  (ii) $\Rightarrow$ (i) always holds for reproducing kernel Hilbert spaces. Indeed,
  $\varphi \in \Mult(\cH_1 \otimes \ell^2(n), \cH_1 \otimes \ell^2(n))$ with $||\varphi|| \le 1$, if
  and only if
  \begin{equation*}
    K_1(z,w) (I - \varphi(z) \varphi(w)^*)
  \end{equation*}
  is a positive definite operator valued kernel.
  By assumption and the Schur product theorem, it follows that
  \begin{equation*}
    K_2(z,w) (I - \varphi(z) \varphi(w)^*)
  \end{equation*}
  is positive definite, hence
  $\varphi \in \Mult(\cH_2 \otimes \ell^2(n), \cH_2 \otimes \ell^2(n))$ with $||\varphi|| \le 1$.
\end{proof}

We finish this section by observing that the completely bounded version of Corollary \ref{cor:multiplier_equal} is not true, that is,
if the identity map from $\Mult(\cH_1)$ to $\Mult(\cH_2)$ is merely assumed to be a completely bounded isomorphism,
then it does not follow that $\cH_1 = \cH_2$ as vector spaces.

\begin{exa}
  \label{exa:mult_equal}
  Let $\cD$ be the Dirichlet space on $\bD$, whose reproducing kernel is given by
  \begin{equation*}
    K_{\cD}(z,w) = - \frac{\log(1-\overline{w}z)}{\ol{w}z}
  \end{equation*}
  and let $H^2 = H^2(\bD)$ be the Hardy space on $\bD$ with reproducing kernel
  \begin{equation*}
    K_{H^2}(z,w) = \frac{1}{1 - z \ol{w}}.
  \end{equation*}
  Then $H^2$ and $\cD$ are complete Nevanlinna-Pick spaces (see, for example, \cite[Corollary 7.41]{AM02}).

  Let $(z_n)_{n=0}^\infty$ be a sequence in $(0,1)$ with $z_0 = 0$ and $\lim_{n \to \infty} z_n = 1$ which is interpolating
  for the multiplier algebra of the Dirichlet space $\cD$. Then $(z_n)$ is also interpolating for $H^\infty = \Mult(H^2)$,
  so if $V = \{z_n: n \in \bN \}$, then
  $\Mult(H^2 \big|_V)$ and $\Mult(\mathcal D \big|_V)$ are equal as algebras, since they are
  both equal to $\ell^\infty$.

  In fact, the normalized kernels in $\mathcal D \big|_V$ and $H^2 \big|_V$ form a Riesz system (see Section 9.3 in
  \cite{AM00}), so there is a bounded invertible map
  \begin{equation*}
    A: H^2 \big|_V \to \cD \big|_V \quad \text{ such that } \quad A \Big( \frac{K_{H^2} (\cdot,w)}{||K_{H^2} (\cdot,w)||} \Big)
    = \frac{K_{\cD} (\cdot,w)}{||K_{\cD} (\cdot,w)||}
  \end{equation*}
  for all $w \in V$. A straightforward computation shows that if $\varphi \in \Mult(H^2 \big|_V)$, then
  \begin{equation*}
    ((A^*)^{-1} M_\varphi A^* f) (w) = \varphi(w) f(w)
  \end{equation*}
  for $f \in \cD \big|_V$ and $w \in V$, so
  \begin{equation*}
    (A^*)^{-1} M_\varphi A^* = M_\varphi.
  \end{equation*}
  It follows that the identity map between $\Mult(H^2 \big|_V)$ and $\Mult(\cD \big|_V)$ is given by
  a similarity.

  However, the spaces $H^2 \big|_V$ and $\cD \big|_V$ are not equal. Indeed, if $f \in \cD$, then
  \begin{equation*}
    |f(z)| = | \langle f, K_{\cD}(\cdot,z) \rangle| \le ||f|| \, \sqrt{K_{\cD}(z,z)} \approx ||f|| \sqrt{- \log(1-z^2)}
  \end{equation*}
  as $z \to 1$, but there are functions in $H^2$ which grow faster, such as
  \begin{equation*}
    f(z) = \sum_{n=0}^\infty (n+1)^{-3/4} z^n,
  \end{equation*}
  for which
  \begin{equation*}
    |f(z)| \approx \Gamma \Big( \frac{1}{4} \Big) (1 - z)^{-1/4}
  \end{equation*}
  as $z \to 1$ from below (see \cite[Chap. XIII,p.280,ex. 7]{WW52}).
\end{exa}

\section{Composition Operators}
\label{sec:composition_operators}

The methods of the last section also apply to composition operators on multiplier algebras.
If $K_1$ and $K_2$ are two kernels on a set $X$, we say that $K_1$ is a rescaling of $K_2$
if there exists a nowhere vanishing function $\delta: X \to \bC$ such that
\begin{equation*}
  K_1(z,w) = \delta(z) \overline{\delta(w)} K_2(z,w) \quad (z,w \in X).
\end{equation*}
Rescaling is an equivalence relation on kernels, and two kernels which are equivalent in this sense
give rise to the same multiplier algebra (see Section 2.6 in \cite{AM02}).

\begin{prop}
  \label{prop:isometric_comp_operator}
  Let $\cH_1$ and $\cH_2$ be irreducible complete Nevanlinna-Pick
  spaces on sets $X_1$ and $X_2$ with kernels $K_1$ and $K_2$, respectively.
  Suppose that $F: X_2 \to X_1$ is a bijection. Then the following are equivalent:
  \begin{enumerate}[label=\normalfont{(\roman*)}]
    \item $C_F: \Mult(\cH_1) \to \Mult(\cH_2), \varphi \mapsto \varphi \circ F$, is an isometric
      isomorphism.
    \item $K_2$ is a rescaling of $(K_1)_F$, where $(K_1)_F = K_1(F(z),F(w))$ for $z,w \in X_2$.
  \end{enumerate}
  In fact, if
  \begin{equation*}
    K_2(z,w) = \delta(z) \overline{\delta(w)} K_1(F(z),F(w)) \quad (z,w \in X_2)
  \end{equation*}
  for some nowhere vanishing function $\delta$ on $X_2$, then
  \begin{equation*}
    U: \cH_1 \to \cH_2, \quad f \mapsto \delta (f \circ F),
  \end{equation*}
  is unitary, and $C_F = \operatorname{Ad}(U)$.
\end{prop}

\begin{proof}
  (i) $\Rightarrow$ (ii). We may assume that $K_2$ is normalized at a point $x_0 \in X_2$. Define
  a kernel $K$ on $X_2$ by
  \begin{equation*}
    K(z,w) = \frac{K_1(F(z),F(w)) K_1(F(x_0),F(x_0))}{K_1(F(z),F(x_0)) K_1(F(x_0),F(w))}
  \end{equation*}
  and let $\cH$ be the reproducing kernel Hilbert space on $X_2$ with kernel $K$. Since $K$
  is a rescaling of $(K_1)_F$, the assumption implies that $\Mult(\cH) = \Mult(\cH_2)$,
  isometrically. Moreover, $K$ is normalized at $x_0$, hence $K_2 = K$ by Corollary \ref{cor:multiplier_equal}.

  (ii) $\Rightarrow$ (i). This implication holds in general, without the assumption that the kernels are
  complete Nevanlinna-Pick kernels.  To see this, it suffices to show the additional assertion.
  It is a standard fact from the theory
  of reproducing kernels that $U$ is unitary. Indeed, the adjoint of $U$ satisfies
  \begin{equation*}
    U^* K_2(\cdot,w) = \ol{\delta(w)} K_1(\cdot,F(w))
  \end{equation*}
  for all $w \in X_2$, thus the assumption easily implies that $U^*$ is unitary.
  Moreover, for $f \in \cH_2$ and $\varphi \in \Mult(\cH_1)$, we have
  \begin{equation*}
    U M_\varphi U^* f = U (\varphi \frac{1}{\delta} (f \circ F^{-1})) = (\varphi \circ F) f,
  \end{equation*}
  hence $C_F = \operatorname{Ad}(U)$ is a well-defined completely isometric isomorphism.
\end{proof}

The last result applies in particular to automorphisms of multiplier algebras.

\begin{cor}
  \label{cor:K_auto}
  Let $\cH$ be an irreducible complete Nevanlinna-Pick space on a set $X$ with kernel $K$,
  normalized at $x_0$, and let $F: X \to X$ be a bijection and $a = F^{-1}(x_0)$. Then $C_F$
  is an isometric automorphism of $\Mult(\cH)$ if and only if
  \begin{equation*}
    K(F(z),F(w)) = \frac{K(z,w) K(a,a)}{K(z,a) K(a,w)}
  \end{equation*}
  for all $z,w \in X$.
\end{cor}

\begin{proof}
  This follows from the preceding proposition as $K_F$ is normalized at $a$. 
\end{proof}

We wish to apply the preceding result to spaces of analytic functions on $\bB_d$. The group
of conformal automorphisms of $\bB_d$ is denoted by $\Aut(\bB_d)$. We also allow the case $d = \infty$,
see \cite{HS71} and the references therein.

\begin{prop}
  \label{prop:K_auto_ball}
  Let $d \in \bN \cup \{\infty\}$ and let
  $\cH$ be a reproducing kernel Hilbert space of analytic functions on $\bB_d$ with
  kernel $K$. Assume that $K$ is normalized at $0$ and does not vanish anywhere on $\bB_d$.
  Then the identity
  \begin{equation}
    \label{eq:K_auto}
    K(\varphi(z),\varphi(w)) = \frac{K(z,w) K(a,a)}{K(z,a)K(a,w)} \quad (z,w \in \bB_d),
  \end{equation}
  where $a = \varphi^{-1}(0)$,
  holds for every $\varphi \in \operatorname{Aut}(\bB_d)$ if and only if
  \begin{equation*}
    K(z,w) = \frac{1}{(1- \langle z,w \rangle)^\alpha}
  \end{equation*}
  for some $\alpha \in [0,\infty)$.
\end{prop}

\begin{proof}
  It is well-known that \eqref{eq:K_auto} holds if $K(z,w) = (1 - \langle z,w \rangle)^{-1}$, see \cite[Theorem 2.2.5]{Rudin08}.
  When raising this identity to the power of $\alpha$, care must be taken if $\alpha$ is not an integer.
  However, \eqref{eq:K_auto} holds for arbitrary $\alpha$, and $z,w \in \bB_d$ with $||z||$ small,
  as $K(z,w)$ is close to $1$ in this case. Since both sides of \eqref{eq:K_auto} are analytic in $z$,
  it holds for all $z \in \bB_d$.

  Conversely, suppose that \eqref{eq:K_auto} holds for all automorphisms $\varphi$. Choosing $\varphi$
  to be unitary, it follows that $K(U z, U w) = K(z,w)$ for all unitary operators $U$ on $\bC^d$.
  By Lemma \ref{lem:unitarily_invariant_power_series},
  there exists an analytic function $f: \bD \to \bC$
  with $f(0) = 1$ and non-negative derivatives at $0$ such that
  \begin{equation*}
    K(z,w) = f(\langle z,w \rangle).
  \end{equation*}
  We wish to show that $f(z) = (1 - z)^{-\alpha}$ for some $\alpha \in [0,\infty)$. Since every
  conformal automorphism of $\bD$ extends to a conformal automorphism of $\bB_d$ (see \cite[Section 2.2.8]{Rudin08}),
  it suffices to prove this for the case $d = 1$.
  
  For $r \in (-1,1)$, consider the conformal
  automorphism $\varphi_r$ of $\bD$ given by
  \begin{equation*}
    \varphi_r(z) = \frac{r -z}{1 - rz}.
  \end{equation*}
  Then for $z \in \bD$ and $w \in (0,1)$, we have
  \begin{equation*}
    f(\varphi_r(z) \varphi_r(w)) = \frac{f(zw) f(r^2)}{f(r z) f(r w)},
  \end{equation*}
  hence
  \begin{equation*}
    f(\varphi_r(z) \varphi_r(w)) f(r z) f(r w) = f(z w) f(r^2).
  \end{equation*}
  Taking the derivative with respect to $r$ at $r = 0$, and simplifying, we obtain
  \begin{equation*}
    (z + w) (f'(z w) (z w - 1) + f(zw) f'(0))) = 0
  \end{equation*}
  for all $z \in \bD$ and $w \in (0,1)$, hence
  \begin{equation*}
    f'(z) (1-z) - f(z) f'(0) = 0.
  \end{equation*}
  for all $z \in \bD$, and $f(0)=1$. This is a first order linear ODE, whose solutions are given by
  \begin{equation*}
    f(z) = (1 - z)^{ -\alpha},
  \end{equation*}
  where $\alpha = f'(0)$. Since $f'(0) \ge 0$, the result follows.
\end{proof}

The desired result about complete Nevanlinna-Pick spaces on $\bB_d$ whose multiplier algebras are isometrically
automorphism invariant is the following corollary.

\begin{cor}
  \label{cor:multiplier_auto_inv}
  Let $d \in \bN \cup \{ \infty \}$ and let $\cH$ be a reproducing kernel Hilbert space of analytic functions on $\bB_d$
  with kernel $K$, normalized at $0$.
  The following are equivalent:
  \begin{enumerate}[label=\normalfont{(\roman*)}]
    \item $\cH$ is an irreducible complete Nevanlinna-Pick space and every
      $\varphi \in \operatorname{Aut}(\bB_d)$ induces an isometric composition operator on $\Mult(\cH)$.
    \item There exists $\alpha \in (0,1]$ such that
      \begin{equation*}
        K(z,w) = \frac{1}{(1-\langle z,w \rangle)^\alpha} \quad (z,w \in \bB_d).
      \end{equation*}
  \end{enumerate}
\end{cor}

\begin{proof}
  In light of Corollary \ref{cor:K_auto} and Proposition \ref{prop:K_auto_ball}, it suffices to show
  that
  \begin{equation*}
    K(z,w) = \frac{1}{(1 - \langle z,w \rangle)^\alpha}
  \end{equation*}
  is an irreducible complete Nevanlinna-Pick kernel if and only if $\alpha \in (0,1]$.

  If $\alpha = 0$, then $K$ is identically $1$, and thus not irreducible. If $\alpha > 0$, then
  Lemma \ref{lem:NP_a_b} applies to show that $K$ is an irreducible complete Nevanlinna-Pick kernel
  if and only if the function
  $1 - (1 - x)^\alpha$ has non-negative Taylor coefficients at $0$.
  Observe that
  \begin{equation*}
    1 - (1 - x)^\alpha = \sum_{k=1}^\infty (-1)^{k+1} \binom{\alpha}{k} x^k,
  \end{equation*}
  where
  \begin{equation*}
    \binom{\alpha}{k} = \frac{\alpha (\alpha - 1) (\alpha - 2) \ldots (\alpha - k +1)}{k!}.
  \end{equation*}
  The coefficient of $x^2$ in this formula equals
  \begin{equation*}
    - \frac{\alpha (\alpha - 1)}{2},
  \end{equation*}
  which is negative if $\alpha > 1$. Conversely, if $\alpha \le 1$, then
  all Taylor coefficients are non-negative.
\end{proof}

\section{Algebraic consistency and varieties}
\label{sec:alg_cons}

When studying isomorphisms of multiplier algebras, we
will usually make an additional assumption,
which, roughly speaking, guarantees that the functions in the reproducing kernel Hilbert space
are defined on their natural domain of definition.

More precisely, let $\cH$ be a Hilbert function space
on a set $X$ with $1 \in \cH$. A non-zero bounded
linear functional $\rho$ on $\cH$ is called \emph{partially multiplicative} if $\rho(\varphi f) = \rho(\varphi) \rho(f)$ whenever $\varphi \in \Mult(\cH)$ and $f \in \cH$. We say that
$\cH$ is \emph{algebraically consistent} if for every partially multiplicative
functional $\rho$ on $\cH$, there exists $x \in X$ such that $\rho(f) = f(x)$
for all $f \in \cH$.

\begin{exa}
  The reproducing kernel Hilbert space $\cH$ on $\bD$ with kernel
  \begin{equation*}
    K(z,w) = \sum_{n=0}^\infty 2^{-n} (z \overline w)^n = \frac{1}{1 - \frac{1}{2} z \overline{w}}
  \end{equation*}
  is not algebraically consistent. Indeed, every function in $\cH$ extends uniquely to
  an analytic function on the open disc of radius $\sqrt{2}$ around the origin.
\end{exa}

\begin{rem}
  Our definitions of a partially multiplicative functional and of algebraic consistency are
  inspired by \cite[Definition 1.5]{CM95} of Cowen-MacCluer, but are slightly different.
  A non-zero bounded linear functional $\rho$ on $\cH$ is partially multiplicative
  in the sense of Cowen-MacCluer
  if $\rho(f g) = \rho(f) \rho(g)$ whenever $f,g \in \cH$ such that
  the pointwise product $f g$ belongs to $\cH$ as well. The Hilbert function space $\cH$ is algebraically consistent
  in the sense of Cowen-MacCluer if every such functional is given by evaluation at a point in $X$. We also
  refer the reader to \cite[Section 2]{MS15}; their \emph{generalized kernel functions} are precisely
  the elements of $\cH$ which give rise to partially multiplicative functionals in the sense of Cowen-MacCluer.

  Clearly, every functional which is partially multiplicative in the sense of Cowen-MacCluer is partially
  multiplicative in our sense. Therefore, every Hilbert function space which is algebraically consistent in our
  sense is algebraically consistent in the sense of Cowen-MacCluer.

  Our definition of algebraic consistency requires that $1 \in \cH$ and is only meaningful
  if $\cH$ has ``enough'' multipliers. This limits its applicability for general Hilbert function spaces.
  However, it seems to be well-suited for normalized irreducible complete Nevanlinna-Pick spaces.
  In particular, we will see that in this setting,
  algebraic consistency in our sense is closely related to the notion of
  a variety from \cite{DRS15} (see Proposition \ref{prop:alg_consistent_equiv} below) and behaves
  well with respect to restrictions of complete Nevanlinna-Pick spaces to subsets
  (see Lemma \ref{lem:alg_cons_restr} below). Moreover, if $\cH$ is a normalized irreducible
  complete Nevanlinna-Pick space, then $1 \in \cH$ and the multiplier algebra contains at least
  all kernel functions (this known fact can be deduced, for example, from Proposition \ref{prop:kernel_recover_iso},
  as $\psi_w = 1 - 1 / K(\cdot,w)$ is a strictly contractive multiplier, so
  $K(\cdot,w) = \sum_{n=0}^\infty \psi_w^n$ converges absolutely in the Banach algebra $\Mult(\cH)$).
  In particular, $\Mult(\cH)$ is dense in $\cH$. It remains open if the two definitions of algebraic
  consistency agree for normalized irreducible complete Nevanlinna-Pick spaces (see also Remark \ref{rem:issue_part_mult} below).
\end{rem}

The following lemma provides examples of algebraically consistent spaces (compare with \cite[Theorem 2.15]{CM95}).
The proof in fact shows that for unitarily invariant spaces, our notion of algebraic consistency and
the one of Cowen-MacCluer coincide.

\begin{lem}
  \label{lem:alg_cons_examples}
  Let $d \in \bN \cup \{\infty\}$ and let $\cH$ be a complete Nevanlinna-Pick space on $\bB_d$
  with kernel of the form
  \begin{equation*}
    K(z,w) = \sum_{n=0}^\infty a_n \langle z,w \rangle^n
  \end{equation*}
  such that $a_0 = 1$ and $a_1 \neq 0$.
  \begin{enumerate}[label=\normalfont{(\alph*)}]
    \item If $\sum_{n=0}^\infty a_n = \infty$, then $\cH$ is algebraically consistent on $\bB_d$.
    \item If $\sum_{n=0}^\infty a_n < \infty$, but the series $\sum_{n=0}^\infty a_n x^n$
      has radius of convergence $1$, then the functions in $\cH$ extend to (norm) 
      continuous functions on $\overline{\bB_d}$, and $\cH$ is an algebraically
      consistent space of functions on $\overline{\bB_d}$.
    \item If $\sum_{n=0}^\infty a_n x^n$ has radius of convergence greater than $1$, then $\cH$ is not
      algebraically consistent on $\bB_d$ or on $\ol{\bB_d}$.
  \end{enumerate}
\end{lem}

\begin{proof}
  We begin with some considerations that apply to both (a) and (b). It is known that the condition $a_1 \neq 0$ implies
  that the function $\langle \cdot,w \rangle$ is a multiplier for $w \in \bB_d$ (see, for example,
  \cite[Section 4]{GRS02}). Incidentally, this can also
  be deduced from Proposition \ref{prop:norm_hom} below. For each $i$, let
  \begin{equation*}
    \lambda_i = \rho(z_i).
  \end{equation*}
  We claim that $(\lambda_i) \in \overline{\bB_d}$. To this end, let $w \in \bB_d$ be
  finitely supported, say $w_i = 0$ if $i > N$. Then
  \begin{equation*}
    \langle \lambda,w \rangle = \sum_{i=1}^N \lambda_i \overline{w_i}
    = \rho( \langle \cdot,w \rangle).
  \end{equation*}
  Since $\rho$ is partially multiplicative and non-zero, $\rho(1) = 1$. Thus, we get
  \begin{equation}
    \label{eqn:alg_cons_examples}
    \rho(K(\cdot,w)) = \sum_{n=0}^\infty a_n \rho(\langle \cdot,w \rangle^n)
    = \sum_{n=0}^\infty a_n \rho( \langle \cdot,w \rangle)^n
    = \sum_{n=0}^\infty a_n \langle \lambda,w \rangle^n.
  \end{equation}
  In either case, the series $\sum_{n=0}^\infty a_n x^n$ has radius of convergence $1$, hence
  $|\langle \lambda,w \rangle| \le 1$. Since $w \in \bB_d$ was an arbitrary finitely
  supported sequence, we conclude that $\lambda \in \overline{\bB_d}$.

  Assume now that $\sum_{n=0}^\infty a_n = \infty$. We wish to show that $\lambda \in \bB_d$.
  Suppose for a contradiction that $||\lambda||=1$. Observe that \eqref{eqn:alg_cons_examples}
  holds for all $w \in \bB_d$, so choosing $w = r \lambda$ for $0 < r < 1$, we see that
  \begin{equation*}
    \sum_{n=0}^\infty a_n r^n = \rho(K(\cdot,r \lambda))
    \le ||\rho|| \, \Big( \sum_{n=0}^\infty a_n r^{2 n} \Big)^{1/2}
    \le ||\rho|| \, \Big( \sum_{n=0}^\infty a_n r^{n} \Big)^{1/2},
  \end{equation*}
  which is not possible as $\sum_{n=0}^\infty a_n = \infty$. Consequently,
  $\lambda \in \bB_d$, and it follows from \eqref{eqn:alg_cons_examples} that
  $\rho$ equals point evaluation at $\lambda$. This proves (a).

  For the proof of (b), we observe that $K$ extends to a jointly norm continuous
  function on $\overline{\bB_d} \times \overline{\bB_d}$, hence all functions
  in $\cH$ extend to norm continuous functions on $\overline{\bB_d}$, and
  $\cH$ becomes a reproducing kernel Hilbert space on $\overline{\bB_d}$ in this way.
  Equation \eqref{eqn:alg_cons_examples} shows that every partially multiplicative
  functional is given by point evaluation at a point $\lambda \in \overline{\bB_d}$, so
  that $\cH$ is algebraically consistent.

  Finally, to show (c), we observe that if $\sum_{n=0}^\infty a_n x^n$ has radius of convergence $R > 1$, then
  the functions in $\cH$ extend uniquely to analytic functions on the ball of radius $\sqrt{R}$. In particular,
  $\cH$ is not algebraically consistent on $\bB_d$ or on $\ol{\bB_d}$.
\end{proof}

To show that algebraic consistency is closely related to
the notion of a variety from \cite{DRS15}, we first need a simple lemma.

\begin{lem}
  \label{lem:alg_cons_restr}
  Let $\cH$ be a normalized irreducible complete Nevanlinna-Pick space on a set $X$ which is
  algebraically consistent.
  If $Y \subset X$, then $\cH \big|_Y$ is an algebraically consistent space of functions
  on $Y$ if and only if there is a set of functions $S \subset \cH$ such that
  \begin{equation*}
    Y = \{ x \in X: f(x) = 0 \text{ for all } f \in S \}. 
  \end{equation*}
\end{lem}

\begin{proof}
  Suppose that $Y$ is the common vanishing locus of a set $S \subset \cH$, and let
  $\rho$ be a partially multiplicative functional on $\cH \big|_Y$. Then
  $\widetilde \rho(f) = \rho(f \big|_Y)$ defines a partially multiplicative
  functional on $\cH$. Since $\cH$ is assumed to be algebraically consistent,
  $\widetilde \rho$ is given by point evaluation at a point $y \in X$. We claim that
  $y \in Y$. To this end, observe that for $f \in S$, we have
  \begin{equation*}
    f(y) = \widetilde \rho (f) = \rho( f \big|_Y) = \rho(0) = 0,
  \end{equation*}
  from which we deduce that $y \in Y$. Since every function in $\cH \big|_Y$ is the
  restriction of a function in $\cH$, it follows that $\rho$ is given by evaluation at $y$.
  Hence, $\cH \big|_Y$ is algebraically consistent.

  Conversely, assume that $\cH \big|_Y$ is algebraically consistent. Let
  $S$ be the kernel of the restriction map $\cH \to \cH \big|_Y$ and let $\widehat Y$ denote the
  vanishing locus of $S$. Clearly, $Y \subset \widehat Y$, and we wish to show that $Y = \widehat Y$.
  To this end, observe that every function
  $f \in \cH \big|_Y$ extends uniquely to a function $\widehat f \in \cH \big|_{\widehat Y}$ of the same norm.
  Assume for a contradiction that there exists $x \in \widehat Y \setminus Y$.
  Then we obtain a bounded functional $\rho$ on $\cH \big|_Y$ which is defined
  by $\rho(f) = \widehat f(x)$. To see that $\rho$ is partially multiplicative,
  note that if $\varphi \in \Mult(\cH \big|_Y)$, then by the Nevanlinna-Pick property,
  $\varphi$ extends to a multiplier on $\cH \big|_{\widehat Y}$, which necessarily equals $\widehat \varphi$. Thus, $\widehat{\varphi f} = \widehat \varphi \widehat f$.
  Since $\cH$ is irreducible, it separates the points of $X$, so $\rho$ is not equal to point evaluation at a point in $Y$,
  a contradiction. Therefore, $Y = \widehat Y$.
\end{proof}

\begin{rem}
  \label{rem:issue_part_mult}
  It is the second part of the above proof where the difference between our definition of partially multiplicative functional
  and the one of Cowen-MacCluer is important. Whereas the functional $\rho$ constructed above is partially
  multiplicative in our sense, it does not seem to be clear if $\rho$ is partially multiplicative
  in the sense of Cowen-MacCluer. Using notation as in the proof, the crucial question is the following:
  If $f,g \in \cH \big|_Y$ such
  that $f g \in \cH \big|_Y$, is $\widehat{f g} = \widehat f \widehat g$?

  It is not hard to see that the following properties are equivalent for a normalized irreducible complete Nevanlinna-Pick
  space $\cH$ on a set $X$ and a subset $Y \subset X$:
  \begin{enumerate}[label=\normalfont{(\roman*)}]
    \item Whenever $f,g \in \cH \big|_Y$ such that $f g \in \cH \big|_Y$, then $\widehat{f g} = \widehat{f} \widehat{g}$.
    \item Whenever $f,g \in \cH \big|_Y$ such that $f g \in \cH \big|_Y$, then $\widehat{f} \widehat{g} \in \cH \big|_{\widehat Y}$.
    \item Whenever $h_1,h_2,h_3 \in \cH$ such that $h_1 = h_2 h_3$ on $Y$, then $h_1 = h_2 h_3$ on $\widehat Y$.
  \end{enumerate}
  Here, as in the proof, $\widehat Y$ denotes the vanishing locus of the kernel of the restriction map $\cH \to \cH \big|_Y$,
  which is the smallest common zero set of a family of functions in $\cH$ which contains $Y$. Moreover,
  for $f \in \cH \big|_Y$, the unique extension of $f$ to a function in $\cH \big|_{\widehat Y}$ is denoted by $\widehat f$.

  Property (iii) and hence all properties are satisfied if
  $\cH = H^2(\bD)$, the Hardy space on the unit disc, and $Y \subset \bD$ is any subset, since the product of two functions
  in $H^2(\bD)$ belongs to $H^1(\bD)$, and the zero sets of families of functions in $H^2(\bD)$ and $H^1(\bD)$ coincide (they
  are precisely the Blaschke sequences in $\bD$, see \cite[Section II.2]{Garnett07}).
  
  It does not seem to be known
  if these properties hold if $\cH = H^2_d$ for $d \ge 2$ and $Y \subset \bB_d$ is an arbitrary subset. If they always
  hold in this case, then the arguments of this section show that our notion of algebraic consistency and the one
  of Cowen-MacCluer agree for normalized irreducible complete Nevanlinna-Pick spaces.
  We also refer the reader to \cite[Section 5]{MS15}, where
  it is shown these properties hold for $\cH = H^2_\infty$ and certain special subsets $Y$ of $\bB_\infty$.
\end{rem}

Let $\cH$ be a normalized irreducible complete Nevanlinna-Pick space on $X$ with kernel $K$. Recall from
Section \ref{sec:prelim} that
an embedding for $\cH$ is an injective function
$j: X \to \bB_m$ such that
\begin{equation*}
  K(z,w) = k_m(j(z),j(w)) \quad (z,w \in X),
\end{equation*}
where $k_m$ denotes the kernel of the Drury-Arveson space on $\bB_m$. A \emph{variety} in $\bB_m$
(see \cite[Section 2]{DRS15}) is the common zero set of a family of functions in $H^2_m$.

\begin{prop}
  \label{prop:alg_consistent_equiv}
  Let $\cH$ be a normalized
  irreducible complete Nevanlinna-Pick space on a set $X$ with kernel $K$.
  The following assertions are equivalent:
  \begin{enumerate}[label=\normalfont{(\roman*)}]
    \item $\cH$ is algebraically consistent.
    \item There exists an embedding $j: X \to \bB_m$ for $\cH$ such that $j(X)$ is a variety.
    \item For every embedding $j: X \to \bB_m$ for $\cH$, the set $j(X)$ is a variety.
    \item Every weak* continuous character on $\Mult(\cH)$ is given by evaluation at a point in $X$.
  \end{enumerate}
\end{prop}

\begin{proof}
  Let $j: X \to \bB_m$ be an embedding for $\cH$, and let $V = j(X)$.
  Then
  \begin{equation*}
    U: H^2_m \big|_V \to \cH, \quad f \mapsto f \circ j,
  \end{equation*}
  is a unitary operator,
  and consideration of the map $T \mapsto U T U^*$ shows that $U$ maps
  $\Mult(H^2_m \big|_V)$ onto $\Mult(\cH)$. Thus, $\cH$ is algebraically consistent
  if and only if $H^2_m \big|_V$ is. Observe that $H^2_m$ is algebraically consistent
  by Lemma \ref{lem:alg_cons_examples}. Thus, the equivalence of (i), (ii) and (iii) follows
  from Lemma \ref{lem:alg_cons_restr}.

  To see that (iii) implies (iv), we note that the identification of $\Mult(H^2_m \big|_V)$ with $\Mult(\cH)$ from
  the first part is a weak*-weak* homeomorphism, since it is implemented by conjugation with a unitary operator.
  Thus, the result follows from the fact that every weak* continuous character on $\Mult(H^2_m \big|_V)$
  is given by evaluation at a point in $V$, provided that $V$ is a variety (see \cite[Proposition 3.2]{DRS15}).

  Conversely, suppose that (iv) holds, and let $\rho$ be a partially multiplicative functional on $\cH$.
  Then the restriction of $\rho$ to $\Mult(\cH)$ is a character. Since
  \begin{equation*}
    \rho(\varphi) = \rho(M_\varphi 1) \quad \text{ for all } \varphi \in \Mult(\cH),
  \end{equation*}
  it is weak* continuous. By assumption, there is a point $x \in X$ such that $\rho(\varphi) = \varphi(x)$ for
  all $\varphi \in \Mult(\cH)$. Since $\Mult(\cH)$ is dense in $\cH$, it follows that $\rho$ is given by evaluation
  at $x$. Consequently, $\cH$ is algebraically consistent.
\end{proof}

In the setting of the last proposition, we identify $X$ with a subset of the maximal ideal space
of $\Mult(\cH)$ via point evaluations.

\begin{lem}
  \label{lem:hom_weak_star}
  Let $\cH_1$ and $\cH_2$ be normalized algebraically consistent irreducible complete Nevanlinna-Pick spaces on sets $X_1$ and $X_2$, respectively. Let $\Phi: \Mult(\cH_1) \to \Mult(\cH_2)$ be a unital
  homomorphism.
  Then the following assertions
  are equivalent:
  \begin{enumerate}[label=\normalfont{(\roman*)}]
    \item $\Phi$ is weak*-weak* continuous.
    \item $\Phi^*(X_2) \subset X_1$.
    \item There is a map $F: X_2 \to X_1$ such that
      \begin{equation*}
        \Phi(\varphi) = \varphi \circ F
      \end{equation*}
      for all $\varphi \in \Mult(\cH_1)$.
  \end{enumerate}
  In this case, the map $F$ in (iii) is the restriction of $\Phi^*$ to $X_2$.
\end{lem}

\begin{proof}
  The implication (i) $\Rightarrow$ (ii) follows immediately from the description
  of the weak* continuous characters in Proposition \ref{prop:alg_consistent_equiv}.
  Assume that (ii) holds, and let $F$ denote the restriction of $\Phi^*$ to $X_2$.
  Then
  \begin{equation*}
    \Phi(\varphi)(\lambda) = \Phi^*(\delta_\lambda) (\varphi) = (\delta_{F(\lambda)}) (\varphi)
    = (\varphi \circ F)(\lambda)
  \end{equation*}
  for all $\lambda \in X_2$. Hence, $\Phi$ is given by composition with $F$, that is, (iii) holds.

  To show that (iii) implies (i), it suffices to show that $\Phi$ is weak*-weak* continuous
  on bounded sets by the Krein-Smulian theorem.
  This in turn follows from the general fact that for a bounded net of multipliers,
  convergence in the weak* topology is equivalent to pointwise convergence.

  Finally, if $F$ is as in (iii), then
  \begin{equation*}
    \varphi(F(x)) = \Phi(\varphi) (x) = \varphi(\Phi^*(x))
  \end{equation*}
  for all $x \in X_2$ and all $\varphi \in \Mult(\cH_1)$, so the assertion follows from the 
  fact that $\Mult(\cH_1)$ separates the points of $X_1$ as $\cH_1$ is an irreducible complete
  Nevanlinna-Pick space (this can be deduced, for example, from Proposition \ref{prop:kernel_recover_iso}).
\end{proof}

As a consequence, we see that weak*-weak* homeomorphic isometric isomorphisms between multiplier algebras are always unitarily implemented.
In \cite{DRS15}, this was shown for spaces which admit an embedding into a
finite dimensional ball using different
methods. This additional assumption was recently removed in \cite{SS14} by refining these methods.
\begin{prop}
  Let $\cH_1$ and $\cH_2$ be normalized algebraically consistent irreducible complete Nevanlinna-Pick spaces on sets $X_1$ and $X_2$, respectively.
  Let $\Phi: \Mult(\cH_1) \to \Mult(\cH_2)$ be a unital isometric isomorphism. If $\Phi$ is a weak*-weak* homeomorphism,
  then $\Phi$ is given by composition with a bijection $F: Y \to X$ and it is unitarily implemented.
\end{prop}

\begin{proof}
  Lemma \ref{lem:hom_weak_star}, applied to $\Phi$ and $\Phi^{-1}$, shows that $\Phi$ is given by composition. Thus,
  Proposition \ref{prop:isometric_comp_operator} implies that $\Phi$ is unitarily implemented.
\end{proof}

\section{Graded complete Nevanlinna-Pick spaces}
\label{sec:graded_spaces}

In this section, we consider reproducing kernel Hilbert spaces which admit a natural grading.
Let $\cH$ be a reproducing kernel Hilbert space on a set $X$ with reproducing kernel $K$, and let
$X$ be equipped with an action of the circle group $\bT$. We say that $K$ is $\bT$-invariant if
\begin{equation*}
  \bT \to \bC, \quad \lambda \mapsto K(\lambda z,w),
\end{equation*}
is continuous for all $z,w \in X$, and
\begin{equation*}
  K(\lambda z, \lambda w) = K(z,w)
\end{equation*}
for all $\lambda \in \bT$ and $z,w \in X$.
Then the $\bT$-action on $X$ induces a strongly continuous unitary representation
\begin{equation*}
  \Gamma: \bT \to \cB(\cH), \quad \Gamma(\lambda) (f) (z) = f(\lambda z).
\end{equation*}
Indeed, $\Gamma(\lambda)$ is unitary for $\lambda \in \bT$, and for $v,w \in X$, we have
\begin{equation*}
  \langle \Gamma(\lambda) K(\cdot,w), K(\cdot,v) \rangle =  K(\lambda v,w),
\end{equation*}
which is continuous in $\lambda$.
For $n \in \bZ$, let
\begin{equation*}
  \cH_n = \{ f \in \cH: \Gamma(\lambda) f = \lambda^n f \text{ for all } \lambda \in \bT \}.
\end{equation*}
Then the closed subspaces $\cH_n$ are pairwise orthogonal, and
it follows from a standard application of the Fej\'er kernel that
\begin{equation*}
  \cH = \bigoplus_{n \in \bZ} \cH_n.
\end{equation*}
Elements of $\cH_n$ are called homogeneous of degree $n$.

\begin{exa}
  \phantomsection
  \label{eg:graded_spaces}
  \begin{enumerate}[label=\normalfont{(\alph*)},wide]
    \item
  Let $d < \infty$ and $\Omega \subset \bC^d$ be open and connected with $0 \in \Omega$ and $\bT \Omega \subset \Omega$.
  Then $\bT$ acts on $\Omega$ by scalar multiplication. Let $\cH$ be a reproducing kernel Hilbert space
  of analytic functions on $\Omega$ with a $\bT$-invariant kernel $K$.
  It is not hard to see that
  \begin{equation*}
    \cH_n = \{ f \in \cH: f \text{ is a homogeneous polynomial of degree } n \}
  \end{equation*}
  for $n \ge 0$, and $\cH_n = \{0\}$ for $n < 0$.
  Concrete examples of this type include many classical spaces on $\bB_d$ or $\bD^d$, such as the Hardy space and the Dirichlet space.
  \item
  Let $d \in \bN \cup \{\infty\}$, and let $X \subset \bC^d$ satisfy $\overline{\bD} X \subset X$.
  Let $\cH$ be a reproducing kernel Hilbert space on $X$ with a $\bT$-invariant kernel $K$, and assume
  that for $f \in \cH$ and $x \in X$, the function
  \begin{equation*}
    f_x: \overline{\bD} \to \bC, \quad z \mapsto f(z x),
  \end{equation*}
  is contained in the disc algebra. Then $\cH_n = \{0\}$ for $n < 0$ and for $n \ge 0$, the
  space $\cH_n$ consists of all functions $f$ in $\cH$ such that $f_x$ is a multiple of $z^n$
  for every $x \in X$.
  \end{enumerate}
\end{exa}

We require a homogeneous decomposition not only for functions, but also for kernels.
\begin{lem}
  \label{lem:kernel_decomposition}
  Let $K$ be a $\bT$-invariant positive definite kernel on $X$, possibly with zeroes on the diagonal.
  Then there are uniquely determined Hermitian kernels $K_n$ on $X$ such that
  for $z,w \in X$, we have
  \begin{enumerate}[label=\normalfont{(\arabic*)}]
    \item $K_n(\lambda z, w) = \lambda^n K_n(z,w)$ for $\lambda \in \bT$ and
    \item $K(z,w) = \sum_{n \in \mathbb Z} K_n(z,w)$, where the series converges absolutely.
  \end{enumerate}
  In this case, $K_n$ is the reproducing kernel of the space of homogeneous
  elements of degree $n$ in $\cH$. In particular, each $K_n$ is positive definite.
\end{lem}

\begin{proof}
  Let $\cH$ be the reproducing kernel Hilbert space on $X$ with kernel $K$. For $w \in X$, let
  \begin{equation*}
    K(\cdot, w) = \sum_{n \in \bZ} K_n(\cdot,w)
  \end{equation*}
  be the homogeneous expansion of $K(\cdot,w)$ in $\cH = \bigoplus_{n \in \bZ} \cH_n$.
  Observe that for $f \in \cH_n$ and $w \in X$, we have
  \begin{equation*}
    \langle f, K_n(\cdot,w) \rangle = \langle f, K(\cdot,w) \rangle = f(w),
  \end{equation*}
  hence $K_n$ is the reproducing kernel of $\cH_n$, and in particular positive definite.
  The first property is clear. Since convergence in $\cH$ implies pointwise convergence on $X$, it follows that
  \begin{equation*}
    K(z,w) = \sum_{n \in \bZ} K_n(z,w).
  \end{equation*}
  Positive definiteness of $K$ implies that $|K_n(z,w)|^2 \le K_n(z,z) K_n(w,w)$, thus
  $|K_n(z,w)| \le \max \{ K_n(z,z),K_n(w,w) \}$, so the series converges absolutely.

  The uniqueness statement follows from the uniqueness of the Fourier expansion of the continuous function
  \begin{equation*}
    \lambda \mapsto K(\lambda z ,w) = \sum_{n \in \bZ} K_n(\lambda z,w) = \sum_{n \in \bZ} \lambda^n K_n(z,w)
  \end{equation*}
  on $\bT$.
\end{proof}

Incidentally, the last lemma provides a simple proof of the following known fact (cf. the proof
of Theorem 7.33 in \cite{AM02}).

\begin{cor}
  \label{cor:powerseries_kernel}
  Let $(a_n)_n$ be a sequence of complex numbers such that the power series $\sum_{n=0}^\infty a_n t^n$
  has a positive radius of convergence $R$. Let $\cE$ be a Hilbert space and let $B_R(0)$ denote
  the open ball of radius $R$ around $0$ in $\cE$. Then the function $K$ defined by
  \begin{equation*}
    K(z,w) = \sum_{n=0}^\infty a_n \langle z,w \rangle^n \quad (z,w \in B_R(0))
  \end{equation*}
  is a positive definite kernel if and only if $a_n \ge 0$ for all $n \in \bN$.
\end{cor}

\begin{proof}
  By the Schur product theorem, $(z,w) \mapsto \langle z,w \rangle^n$ is a positive definite kernel for
  all $n \in \bN$. Thus, the backward direction is clear. Conversely, if $K$ is a positive definite kernel,
  then an application of Lemma \ref{lem:kernel_decomposition} shows that
  \begin{equation*}
    K_n(z,w) = a_n \langle z,w \rangle^n
  \end{equation*}
  defines a positive definite kernel for all $n \in \bN$. In particular, each $K_n$ is Hermitian, hence
  $a_n \in \bR$. Moreover, if $a_n \le 0$, then $-K_n$ is positive definite as well, hence $K_n = 0$
  and thus $a_n = 0$. This observation finishes the proof.
\end{proof}

Let $\cH$ be a reproducing kernel Hilbert space on a set $X$ with a $\bT$-invariant kernel $K$. Assume
that $K$ is normalized at a point in $X$, so that the constant function $1$ is contained in $\cH$
and has norm $1$. Recall that $\cH$ admits an orthogonal decomposition $\cH = \bigoplus_{n \in \bZ} \cH_n$.
We say that $\cH$ is \emph{\graded}
if $\cH_0 = \bC 1$ and $\cH_n = \{0\}$ for $n < 0$. All spaces in Example \ref{eg:graded_spaces}
are \graded, provided their kernel is normalized at a point. In particular, unitarily invariant spaces
on $\bB_d$ are \graded.

In Drury-Arveson space, the multiplier norm of a homogeneous polynomial is equal to its Drury-Arveson norm.
This can be shown by embedding $H^2_d$ into the full Fock space (see also \cite[Lemma 9.5]{SS09}).
For a special class of complete Nevanlinna-Pick spaces $\cH$ on $\bD$,
it was shown that $||z^n||_{\cH} = ||z^n||_{\Mult(\cH)}$ \cite[Lemma 7.2]{DHS15} for all $n \in \bN$.
The next proposition generalizes these results.

\begin{prop}
  \label{prop:norm_hom}
  Let $\cH$ be an irreducible complete Nevanlinna-Pick space which is \graded. If $f \in \cH$ is homogeneous,
  then $f \in \Mult(\cH)$ and $||f||_{\Mult(\cH)} = ||f||_{\cH}$.
\end{prop}

\begin{proof}
  The proof is an abstract version of the proof of Lemma 7.2 in \cite{DHS15}.
  Let $K = \sum_{n=0}^\infty K_n$ be the homogeneous decomposition of $K$ from Lemma \ref{lem:kernel_decomposition}.
  In a first step, we will show that for every pair of natural numbers $n$ and $k$, the kernel
  \begin{equation*}
    K_{n+k} - K_n K_k
  \end{equation*}
  is positive definite. We proceed by induction on $n$. The assumption $\cH_0 = \bC 1$ implies that $K_0 = 1$,
  so this is trivial for $n=0$. Assume that $n \ge 1$ and that the assertion is true for $0,1,\ldots, n-1$. Since $K$ is a normalized irreducible complete
  Nevanlinna-Pick kernel, $F = 1 - \frac{1}{K}$ is a positive definite kernel on $X$
  by Theorem \ref{thm:CNP_char}, and it is clearly $\bT$-invariant.
  Let $F = \sum_{j=0}^\infty F_j$ be the homogeneous decomposition of $F$.
  Since $K = K F + 1$, we have
  \begin{equation*}
    \sum_{i=0}^\infty K_i = \sum_{i=0}^\infty \sum_{j=0}^i K_{i-j} F_{j} + 1,
  \end{equation*}
  where we have used that all series converge absolutely.
  Since the homogeneous expansion is unique, we may compare
  homogeneous components in this equation.
  For $i=0$, we use that $K_0 = 1$ to obtain $F_0 = 0$. For $i \ge 1$, we therefore
  get the identity
  \begin{equation*}
    K_i = \sum_{j=1}^{i} K_{i-j} F_{j}.
  \end{equation*}
  Using this identity with $i=n+k$ and $i=n$, we deduce that
  \begin{align*}
    K_{n+k} - K_n K_k &= \sum_{j=1}^{n+k} K_{n+k-j} F_j - \sum_{j=1}^n K_{n-j} K_k F_j  \\
    &\ge \sum_{j=1}^n (K_{n+k-j} - K_{n-j} K_{k}) F_j \ge 0
  \end{align*}
  by induction hypothesis and the Schur product theorem. This finishes the inductive proof.

  Now, let $f \in \cH$ be homogeneous of degree $n \ge 0$ and suppose that $||f||_{\cH} \le 1$. A well-known characterization
  of the norm in a reproducing kernel Hilbert space implies that
  \begin{equation*}
    (z,w) \mapsto K(z,w) - f(z) \overline{f(w)}
  \end{equation*}
  is positive definite. Note that the degree $n$ homogeneous component of this kernel is $K_n(z,w) - f(z) \overline{f(w)}$,
  which is positive definite by Lemma \ref{lem:kernel_decomposition}. Using the Schur product theorem, we deduce that
  \begin{equation*}
     \sum_{k=0}^\infty K_k(z,w) (K_n(z,w) - f(z) \overline{f(w)})
  \end{equation*}
  is positive definite. Since $K_n K_k \le K_{n+k}$, this implies that
  \begin{equation*}
     0 \le \sum_{k=0}^\infty K_{n+k}(z,w) - \sum_{k=0}^\infty (K_{k} (z,w) f(z) \overline{f(w)})
     \le K(z,w) ( 1 - f(z) \overline{f(w)}),
  \end{equation*}
  so that $f$ is a contractive multiplier on $\cH$.
\end{proof}

\begin{rem}
  For Drury-Arveson space, the above proof can be somewhat simplified. In this case,
  $K_n(z,w) = \langle z,w \rangle^n$, hence
  $K_{n+k} = K_n K_k$,
  so that the first step is trivial.
\end{rem}

As a consequence, we obtain a simple necessary condition for the complete Nevanlinna-Pick property
of a unitarily invariant space.
\begin{cor}
  \label{cor:NP_nec}
  Let $d \in \bN \cup \{\infty\}$ and let
  $\cH$ be an irreducible unitarily invariant reproducing kernel Hilbert space on $\bB_d$ with reproducing kernel
  \begin{equation*}
    K(z,w) = \sum_{n=0}^\infty a_n \langle z,w \rangle^n
  \end{equation*}
  such that $a_0 = 1$. If $\cH$ is a complete Nevanlinna-Pick space, then
  \begin{equation*}
    a_n a_k \le a_{n + k}
  \end{equation*}
  for all $n,k \in \bN$.
\end{cor}

\begin{proof}
  The proof of Proposition \ref{prop:norm_hom} shows that if $\cH$ is a complete Nevanlinna-Pick space,
  then
  \begin{equation*}
    K_{n+k} - K_n K_k
  \end{equation*}
  is positive definite for every $k,n \in \bN$. But
  \begin{equation*}
    K_n(z,w) = a_n \langle z,w \rangle^n
  \end{equation*}
  for $z,w \in \bB_d$, hence the result follows.
\end{proof}

\begin{exa}
  In the setting of the last lemma, let $a_n = (n+1)^s$ for $n \in \bN$. If $s > 0$, then
  \begin{equation*}
    a_1^2 = 4^s > 3^s = a_2,
  \end{equation*}
  so $\cH$ is not a complete Nevanlinna-Pick space. Observe that if $d= 1$ and $s=1$, we obtain
  the well-known fact that the Bergman space on $\bD$ is not a complete Nevanlinna-Pick space.
\end{exa}

\begin{exa}
  Let us observe that the necessary condition in Corollary \ref{cor:NP_nec} is not sufficient.
  Let $d = 1$ and define
  $a_0 = 1, a_1 = \frac{1}{2}, a_n =1$ for $n \ge 2$, that is, $\cH$ is the space on $\bD$ with reproducing kernel
  \begin{equation*}
    K(z,w) = \sum_{n=0}^\infty a_n (z \overline w)^n = \frac{1}{1 - z \overline{w}} - \frac{1}{2} z \overline{w}.
  \end{equation*}
  Then $a_k a_n \le a_{n + k}$ for $n,k \in \bN$.
  However,
  \begin{equation*}
    1 - \frac{1}{K(z,w)} = \frac{1}{2} z \overline{w} + \frac{3}{4} (z \overline w)
    + \frac{1}{8} (z \overline w)^3 - \frac{5}{16} (z \overline {w})^4 + \text{h.o.t.}.
  \end{equation*}
  Hence, $\cH$ is not a complete Nevanlinna-Pick space by \cite[Theorem 7.33]{AM02}.
\end{exa}

If $\cH$ is a \graded\ complete Nevanlinna-Pick space, we let $A(\cH)$ denote the norm closed
linear span of the homogeneous elements in $\Mult(\cH)$. For example, $A(H^2)$ is the disc algebra.

For \graded\ complete Nevanlinna-Pick spaces, there is a bounded version of Corollary \ref{cor:multiplier_equal}.

\begin{prop}
  Let $X$ be a set equipped with an action of $\bT$.
  Let $\cH_1$ and $\cH_2$ be two irreducible complete Nevanlinna-Pick spaces on
  $X$ with reproducing kernels $K_1$ and $K_2$, respectively.
  Assume that $\cH_1$ and $\cH_2$
  are \graded\ with respect to the action of $\bT$ on $X$. Then the following assertions are equivalent:
  \begin{enumerate}[label=\normalfont{(\roman*)}]
    \item $\cH_1 = \cH_2$ as vector spaces.
    \item $\Mult(\cH_1) = \Mult(\cH_2)$ as algebras.
    \item $A(\cH_1) = A(\cH_2)$ as algebras.
    \item There exist $c_1,c_2 > 0$ such that $c_1^2 K_2 - K_1$ and $c_2^2 K_1 - K_2$ are positive definite.
    \item The identity map $\cH_1 \to \cH_2$ is a bounded isomorphism which induces similarities $\Mult(\cH_1) = \Mult(\cH_2)$
      and $A(\cH_1) = A(\cH_2)$.
  \end{enumerate}
\end{prop}

\begin{proof}
  The equivalence of (i) and (iv) is well known. (v) implies (i), (ii) and (iii) is trivial, and (i) implies (v)
  follows from the closed graph theorem.
  It remains to see that (ii) or (iii) implies (i). In both cases, $\cH_1$ and $\cH_2$ have
  the same homogeneous elements by Proposition \ref{prop:norm_hom}.
  Moreover, since all algebras in question are semi-simple, there are
  constants $C_1, C_2 > 0$ such that
  \begin{equation*}
    \frac{1}{C_2} ||f||_{\Mult(\cH_2)} \le ||f||_{\Mult(\cH_1)} \le C_1 ||f||_{\Mult(\cH_2)}
  \end{equation*}
  for every homogeneous element $f$ (see \cite[Proposition 4.2]{Dales78}). Since homogeneous elements of different
  degree are orthogonal in $\cH_1$ and $\cH_2$, we deduce from Proposition \ref{prop:norm_hom} that there is a bounded isomorphism
  $\cH_1 \to \cH_2$ which acts as the identity on homogeneous elements, and hence everywhere. Thus, (i) holds.
\end{proof}

\section{Restrictions of unitarily invariant spaces}
\label{sec:unit_inv}

For the remainder of this article, we will consider restrictions of unitarily invariant
spaces on $\bB_d$, and from now on, we will always assume that $d < \infty$.

Suppose that $\cH$ is a unitarily invariant space on $\bB_d$ with reproducing kernel
\begin{equation}
  \label{eqn:kernel_power_series}
  K(z,w) = \sum_{n=0}^\infty a_n \langle z,w \rangle^n,
\end{equation}
where $a_0 = 1$ and $a_n \ge 0$ for all $n \in \bN$. We will assume that $\cH$ has the following properties:
\begin{enumerate}[label=\normalfont{(\alph*)}]
  \item $\cH$ contains the coordinate functions.
  \item $\cH$ is algebraically consistent on $\bB_d$.
  \item $\cH$ is an irreducible complete Nevanlinna-Pick space.
\end{enumerate}
For simplicity, we will call a space
which satisfies these conditions a \emph{unitarily invariant complete NP-space on $\bB_d$}.

The conditions above can also be expressed in terms of the reproducing kernel. If the kernel is given as in
\eqref{eqn:kernel_power_series}, then (a) is equivalent to demanding that $a_1 > 0$ (see, for example
\cite[Section 4]{GRS02} or \cite[Proposition 4.1]{GHX04}). 
Lemma \ref{lem:alg_cons_examples} shows that Condition (b)
holds if and only if the radius of convergence of the series $\sum_{n=0}^\infty a_n t^n$ is $1$ (so that $\cH$ is defined on $\bB_d$) and $\sum_{n=0}^\infty a_n = \infty$.
In the presence of (a), $\cH$ is
an irreducible complete Nevanlinna-Pick space
if and only if the sequence $(b_n)_{n=1}^\infty$ defined by
\begin{equation}
  \label{eqn:complete_NP_cond}
  \sum_{n=1}^\infty b_n t^n = 1 - \frac{1}{\sum_{n=0}^\infty a_n t^n}
\end{equation}
for $t$ in a neighbourhood of $0$ is a sequence of non-negative real numbers, see Lemma \ref{lem:NP_a_b}.

We will also consider spaces on $\ol{\bB_d}$. The only difference to the above setting is that here,
the functions in $\cH$ are assumed to be analytic on $\bB_d$ and continuous on $\ol{\bB_d}$.
Moreover, $\cH$ is assumed to be algebraically consistent on $\ol{\bB_d}$. In terms of the reproducing kernel $K$, this means
that $\sum_{n=0}^\infty a_n < \infty$ but the power series $\sum_{n=0}^\infty a_n t^n$ has radius of convergence $1$ (see
Lemma \ref{lem:alg_cons_examples}).
We call such a space a \emph{unitarily invariant complete NP-space on $\overline{\bB_d}$}.
We say that $\cH$ is a \emph{unitarily invariant complete NP-space} to mean that $\cH$ is either
a unitarily invariant complete NP-space on $\bB_d$ or on $\ol{\bB_d}$.

\begin{rem} Let $\cH$ be a unitarily invariant complete NP-space as above.
  \label{rem:unitarily_invariant}
  \begin{enumerate}[label=\normalfont{(\alph*)},wide]
    \item $\cH$ is \graded\ in the sense of Section \ref{sec:graded_spaces}.
    \item The condition that the sequence $(b_n)$ in Equation \eqref{eqn:complete_NP_cond} in non-negative
      is often difficult to check in practice. A sufficient condition for this to hold is that
  the sequence $(a_n)_n$ is strictly positive and log-convex,
  i.e.
  \begin{equation*}
    \frac{a_n}{a_{n+1}} \le \frac{a_{n-1}}{a_n}  \quad (n \ge 1),
  \end{equation*}
  see, for example, \cite[Lemma 7.38]{AM02}.
  \item Since $\cH$ contains the coordinate functions, it follows from Proposition \ref{prop:norm_hom}
   that the coordinate functions are multipliers. Thus, all polynomials are multipliers. In particular,
   $\cH$ contains all polynomials, so that $a_n > 0$ for all $n \in \bN$ (see also \cite[Section 4]{GRS02}).
  \item The monomials $z^\alpha$, where $\alpha$ runs through all multi-indices of non-negative integers of length $d$,
    form an orthogonal basis for $\cH$. Moreover,
    \begin{equation*}
      ||z^\alpha||_{\cH}^2 = \frac{\alpha!}{|\alpha|! a_{|\alpha|}}
    \end{equation*}
    for every multi-index $\alpha$ (see, for example, \cite[Section 4]{GRS02} or \cite[Proposition 4.1]{GHX04}).
    It follows from unitary invariance that
    \begin{equation*}
      ||\langle \cdot,w \rangle^n||^2_{\cH_I} = \frac{||w||^{2 n}}{a_n}
    \end{equation*}
    for all $w \in \bC^d$ and all $n \in \bN$.
  \end{enumerate}
\end{rem}

\begin{exa}
  \label{exa:unit_invariant}
  For $- 1 \le s \le 0$, let $\cH_s(\bB_d)$ be the reproducing kernel Hilbert space
  on $\bB_d$ with kernel
  \begin{equation*}
    K_s(z,w) = \sum_{n=0}^\infty (n+1)^{-s} \langle z,w \rangle.
  \end{equation*}
  Using part (b) of Remark \ref{rem:unitarily_invariant}, it is easy to see that
  $\cH_s(\bB_d)$ is a unitarily invariant complete NP-space on $\bB_d$.

  If $s < -1$, the series in the definition of $K_s$ converges on $\overline{\bB_d} \times \overline{\bB_d}$.
  Let $\cH_s(\overline{\bB_d})$ be the reproducing kernel Hilbert space on $\overline{\bB_d}$
  with this kernel. As above, it is not hard to see that this space is a unitarily invariant complete NP-space on $\ol{\bB_d}$.

  Closely related to the spaces $\cH_s(\bB_d)$ for $s \in (-1,0]$ are the spaces from Corollary \ref{cor:multiplier_auto_inv}.
  If $\alpha \in (0,1]$, the space $\cK_\alpha$ with reproducing kernel
  \begin{equation*}
    K(z,w) = \frac{1}{(1- \langle z,w \rangle)^\alpha} \quad (z,w \in \bB_d)
  \end{equation*}
  is a unitarily invariant complete NP-space on $\bB_d$.
  Expressing the reproducing kernel as a binomial series and using part (d) of Remark \ref{rem:unitarily_invariant},
  it is straightforward to see that $\cK_\alpha$ and $\cH_{\alpha -1}$ agree as vector spaces,
  and have equivalent norms.
\end{exa}

\begin{rem}
  While we assume that our spaces on $\bB_d$ are invariant under unitary maps, we specifically
  do not assume that they are invariant under other conformal automorphisms of the unit ball.
  Such an assumption would simplify some arguments, but the condition of automorphism invariance
  is often difficult to check in practice. Indeed, even for spaces on $\bD$, there does
  not seem to be a simple criterion for automorphism invariance. We refer the reader to \cite[Section 3.1]{CM95}
  and \cite[Section 8]{CM98}.
\end{rem}

We now turn to restrictions of unitarily invariant complete NP-spaces.
Let $I \subsetneq \bC[z_1,\ldots,z_d]$ be a homogeneous ideal. Following \cite{DRS11}, we define
\begin{equation*}
  Z^0(I) = V(I) \cap \bB_d
\end{equation*}
and
\begin{equation*}
  Z(I) = V(I) \cap \overline{\bB_d},
\end{equation*}
where
\begin{equation*}
  V(I) = \{z \in \bC^d: f(z) = 0 \text{ for all } f \in I \}
\end{equation*}
denotes the vanishing locus of $I$.
Observe that since $I$ is a proper ideal, $Z^0(I)$ always contains the origin.

If $\cH$ is a unitarily invariant complete NP-space on $\bB_d$, we define $\cH_I = \cH \big|_{Z^0(I)}$.
If $\cH$ is a unitarily invariant complete NP-space on $\ol{\bB_d}$, we define $\cH_I = \cH \big|_{Z(I)}$.
Recall from Section \ref{sec:prelim} that the norm on $\cH_I$ is defined in such a way that the
restriction map from $\cH$ onto $\cH_I$ is a co-isometry.
Lemma \ref{lem:alg_cons_restr} shows that $\cH_I$ is algebraically consistent in both cases.
Observe that the circle group acts on $Z^0(I)$ and on $Z(I)$ by scalar multiplication, which
gives the spaces $\cH_I$ a grading in the sense of Section \ref{sec:graded_spaces}. Moreover, the
restriction map from $\cH$ onto $\cH_I$ respects the grading. Thus, a function
in $\cH_I$ is homogeneous of degree $n$ if and only if it is
the restriction of a homogeneous polynomial of degree $n$.

Since an ideal and its radical have the same vanishing locus, there is no loss of generality
in restricting our attention to radical homogeneous ideals. If $I \subsetneq \bC[z_1,\ldots,z_d]$
is a radical homogeneous ideal, then the ring of polynomial functions on $V(I)$ is canonically
isomorphic to the quotient $\bC[z_1,\ldots,z_d]/I$ by Hilbert's Nullstellensatz. The following
lemma, which gives a different description of the space $\cH_I$,
can be thought of as a Hilbert function space analogue of this fact. Results of this type are
certainly well known (cf. \cite[Section 6]{DRS11}), but we do not have a convenient reference for the
precise statement.

\begin{lem}
  \label{lem:nullstellensatz}
  Let $\cH$ be a unitarily invariant space on $\bB_d$ or on $\ol{\bB_d}$ with reproducing kernel $K$,
  and let $I \subsetneq \bC[z_1,\ldots,z_d]$ be a radical homogeneous ideal. Then the closure of $I$ in
  $\cH$ is given by
  \begin{equation*}
    \overline{I} = \{f \in \cH: f \big|_{Z^0(I)} = 0 \}.
  \end{equation*}
  Hence
  the map
  \begin{equation*}
    \cH \ominus I \to \cH_I,
  \end{equation*}
  given by restriction, is a unitary operator. Moreover, $\cH \ominus I$ is the closed linear span of
  the kernel functions $K(\cdot,w)$ for $w \in Z^0(I)$.
\end{lem}

\begin{proof}
  Let
  \begin{equation*}
    R: \cH \to \cH_I
  \end{equation*}
  be the restriction map. Then $R$ is a co-isometry by definition of $\cH_I$,
  thus it suffices to show that $\ker R = \overline{I}$.
  It is clear that $\overline{I} \subset \ker R$. Conversely, let $f \in \ker R$ and let
  $f = \sum_{n=0}^\infty f_n$ be the homogeneous decomposition of $f$.
  Then
  \begin{equation*}
    0 = f(t z) = \sum_{n=0}^\infty t^n f_n(z)
  \end{equation*}
  for all $t \in \bD$ and all $z \in Z^0(I)$, hence each $f_n$ vanishes on $V(I)$.
  Consequently, $f_n \in I$ for all $n \in \bN$ by Hilbert's Nullstellensatz, thus $f \in \overline{I}$.

  Since the restriction map from $\cH$ onto $\cH_I$ is a co-isometry, it follows that this map
  is a unitary operator from $\cH \ominus I$ onto $\cH_I$. Moreover, given $f \in \cH$, we see
  that $f$ is orthogonal to the kernel functions $K(\cdot,w)$ for $w \in Z^0(I)$ if and only
  if $f$ vanishes on $Z^0(I)$, which happens if and only if $f \in \overline{I}$ by the first part.
\end{proof}

Thus, instead of thinking of $\cH_I$ as a space of functions on $Z^0(I)$ or on $Z(I)$,
we may also regard it as a subspace of $\cH$. The following lemma shows
how composition operators act in this second picture of $\cH_I$.
It is a straightforward generalization of a well-known result about
composition operators on reproducing kernel Hilbert spaces (see, for example, \cite[Theorem 1.4]{CM98}).
\begin{lem}
  \label{lem:comp_operators}
  Let $\cH$ be a reproducing kernel Hilbert space on a set $X$ with reproducing kernel $K_{\cH}$, and let
  $\cK$ be a reproducing kernel Hilbert space on a set $Y$ with reproducing kernel $K_{\cK}$. Suppose that
  $Z \subset X$ and $W \subset Y$, and define
  \begin{equation*}
    I(Z) = \{f \in \cH: f \big|_{Z} = 0 \}
  \end{equation*}
  and
  \begin{equation*}
    I(W) = \{f \in \cK: f \big|_{W} = 0 \}.
  \end{equation*}
  Then for a function $\varphi: W \to Z$, the following are equivalent:
  \begin{enumerate}[label=\normalfont{(\roman*)}]
    \item There exists a bounded composition operator $C_\varphi: \cH \big|_{Z} \to \cK \big|_{W}$ such that
      $C_\varphi(f) = f \circ \varphi$ for all $f \in \cH \big|_{Z}$.
    \item There exists a bounded operator $T_\varphi: \cK \ominus I(W) \to \cH \ominus I(Z)$ with
      $T( K_{\cK} (\cdot,w)) = K_{\cH} (\cdot,\varphi(w))$ for all $w \in W$.
  \end{enumerate}
  In this case,
  \begin{equation*}
    T_\varphi = R_Z^{-1} (C_\varphi)^* R_W,
  \end{equation*}
  where
  \begin{equation*}
    R_Z : \cH \ominus I(Z) \to \cH \big|_{Z}, \quad f \mapsto f \big|_{Z},
  \end{equation*}
  and
  \begin{equation*}
    R_W : \cK \ominus I(W) \to \cK \big|_{W}, \quad f \mapsto f \big|_{W},
  \end{equation*}
  denote the unitary restriction maps.
\end{lem}

\begin{proof}
  Suppose that (i) holds. For $w \in W$ and $f \in \cH \ominus I(Z)$, we have
  \begin{align*}
    \langle f,  R_Z^{*} (C_\varphi)^* R_W K_{\cK} (\cdot,w) \rangle_{\cH}
    &= \langle (f \big|_{Z}) \circ \varphi, K_{\cK} (\cdot,w) \big|_{W} \rangle_{\cK |_{W}} \\
    &= f(\varphi(w)) \\
    &= \langle f, K_{\cH}(\cdot,\varphi(w)) \rangle_{\cH}.
  \end{align*}
  Since $K_{\cH}(\cdot,\varphi(w)) \in \cH \ominus I(W)$, we conclude that (ii) holds with
  $T_\varphi = R_Z^* (C_\varphi)^* R_W$, which also proves the additional assertion.

  Conversely, if (ii) holds, let $f \in \cH \ominus I(Z)$ and let $w \in W$. Clearly,
  $R_Z^* (f \big|_Z) = f$ and $R_W^* (K_{\cK}(\cdot,w) \big|_{W}) = K_{\cK}(\cdot,w)$, hence
  \begin{align*}
    (R_W T_\varphi^* R_Z^* f \big|_Z)(w) &= \langle R_W T_\varphi^* f, K_\cK(\cdot,w) \big|_{W} \rangle_{\cK |_W} \\
    &= \langle f, T_\varphi K_{\cK}(\cdot,w) \rangle_{\cH} \\
    &= \langle f, K_{\cH}(\cdot, \varphi(w)) \rangle_{\cH} \\
    &= (f \circ \varphi)(w).
  \end{align*}
  Consequently, (i) holds with $C_\varphi = R_W T_\varphi^* R_Z^*$.
\end{proof}

It may seem restrictive that we only consider restrictions to varieties defined by homogeneous polynomials.
Indeed, if $\cH$ is a unitarily invariant complete NP-space on $\bB_d$ and $X \subset \bB_d$ has circular symmetry, i.e.
$\bT X = X$, then $\cH \big|_{X}$ is \graded\ in the sense of Section \ref{sec:graded_spaces}.
It turns out, however, that algebraic consistency forces $X$ to be a homogeneous variety. More generally, we obtain
the following result.

\begin{lem}
  Let $\cH$ be a normalized irreducible Hilbert function space of analytic
  functions on $\bB_d$ (respectively of continuous functions on $\ol{\bB_d}$ which are analytic
  on $\bB_d$). Let $X \subset \bB_d$ (respectively $X \subset \ol{\bB_d}$) be a non-empty set which
  satisfies $\bT X \subset X$. If $\cH \big|_X$ is algebraically consistent,
  then $X = Z^0(I)$ (respectively $X = Z(I)$) for some
  radical homogeneous ideal $I \subsetneq \bC[z_1,\ldots,z_d]$.
\end{lem}

\begin{proof}
  We first consider the case where $\cH$ is a space of analytic functions on $\bB_d$.
  Let $I$ be the ideal of all polynomials that
  vanish on $X$. Suppose that $f \in \cO(\bB_d)$ vanishes on $X$,
  and let $f = \sum_{n=0}^\infty f_n$ be the homogeneous expansion of $f$. Given $x \in X$,
  the function 
  \begin{equation*}
    \overline{\bD} \to \bC, \quad \lambda \mapsto f(\lambda x),
  \end{equation*}
  is contained in the disc algebra and vanishes on $\bT$, hence it vanishes identically.
  Using the homogeneous expansion of $f$, we see that $f_n(x) = 0$ for all $n \in \bN$.
  Thus, every $f_n$ and hence $f$ vanishes
  on $Z^0(I)$. This argument also shows that $I$ is a homogeneous
  ideal. Moreover, every function in $\cH \big|_X$ extends uniquely to a function in $\cO(\bB_d) \big|_{Z^0(I)}$.

  Clearly, $X \subset Z^0(I)$. To establish equality, denote for $f \in \cH$ the unique extension
  of $f$ to a function in $\cO(\bB_d) \big|_{Z^0(I)}$ by $\widehat f$.
  Observe that $\widehat f$ in fact belongs to $\cH \big|_{Z^0(I)}$ and has the same norm as $f$.
  Since $\cO(\bB_d) \big|_{Z^0(I)}$ is
  an algebra, we see that $\widehat{\varphi f} = \widehat{\varphi} \widehat{f}$ for all
  $\varphi \in \Mult(\cH \big|_X)$ and $f \in \cH \big|_X$. Assume for a contradiction that there exists $x \in Z^0(I) \setminus X$.
  Then $f \mapsto \widehat{f} (x)$ defines a bounded functional on $\cH \big|_X$ which is
  partially multiplicative.
  Since $\cH$ is irreducible, this functional is not given by evaluation at a point in $X$. 
  This contradicts algebraic consistency of $\cH \big|_X$, hence $X = Z^0(I)$.

  Finally, if $\cH$ is a space of continuous functions on $\ol{\bB_d}$ which are analytic on $\bB_d$,
  then the proof above applies to this setting as well once we replace
  $Z^0(I)$ with $Z(I)$ and $\cO(\bB_d)$ with the algebra of all continuous functions on $\ol{\bB_d}$
  which are analytic on $\bB_d$.
\end{proof}

\section{The maximal ideal space}
\label{sec:maximal_ideal_space}

To classify the multiplier algebras of the spaces $\cH_I$ introduced in the last section,
we follow the same route as \cite{DRS11}. To this end, we first study the character spaces
of these multiplier algebras.
We begin with an easier object. Recall that if $\cH$ is a \graded\ complete Nevanlinna-Pick space,
$A(\cH)$ denotes the norm closure of the span of all homogeneous elements in $\Mult(\cH)$.
If $\cA$ is a unital Banach algebra, we let $\cM(\cA)$ denote its maximal ideal space.

\begin{lem}
  \label{lem:gelfand_A_H}
  Let $\cH$ be a unitarily invariant complete NP-space on $\bB_d$ or on $\ol{\bB_d}$,
  and let $I \subsetneq \bC[z_1,\ldots,z_d]$ be a radical homogeneous ideal.
  Then
  \begin{equation*}
    \cM(A(\cH_I)) \to Z(I), \quad \rho \mapsto (\rho(z_1),\ldots,\rho(z_d)),
  \end{equation*}
  is a homeomorphism.
\end{lem}

\begin{proof}
  Let $\rho \in \cM(A(\cH_I))$. We first show that
  $\lambda = (\rho(z_1),\ldots,\rho(z_d)) \in \overline{\bB_d}$. Suppose otherwise,
  and let $0 < r < 1$ be such that $||r \lambda|| > 1$. If $p$ is a polynomial
  with homogeneous decomposition $p = \sum_{n=0}^N p_n$, then
  \begin{equation*}
    |p(r \lambda)|^2 \le \Big( \sum_{n=0}^N r^n | \rho(p_n)| \Big)^2 \le \sum_{n=0}^N r^{2 n}
    \sum_{n=0}^N ||p_n||^2_{\Mult(\cH_I)}.
  \end{equation*}
  By Proposition \ref{prop:norm_hom},
  $||p_n||_{\Mult(\cH_I)} = ||p_n||_{\cH_I}$, hence this quantity is dominated by
  \begin{equation*}
    \frac{1}{1 - r^2} \sum_{n=0}^N ||p_n||^2_{\cH_I}  = \frac{1}{1 - r^2} ||p||^2_{\cH_I}.
  \end{equation*}
  Consequently, $p \mapsto p(r \lambda)$ extends to a well-defined bounded functional $\widetilde \rho$ on $\cH_I$. It is easy to see
  that $\widetilde \rho$ is partially multiplicative, but
  \begin{equation*}
    (\widetilde \rho(z_1),\ldots,\widetilde \rho(z_d)) = r \lambda \notin \overline{\bB_d}.  
  \end{equation*}
  This contradicts the fact that $\cH_I$ is algebraically consistent.
  Clearly, $\lambda \in V(I)$. Thus, if $\Phi$ denotes the map from the statement of the lemma,
  it follows that $\Phi(\rho) \in Z(I)$.
  It is clear that $\Phi$ is continuous, and since the polynomials are dense in
  $A(\cH_I)$ by definition, it is also  injective.
  
  Since $\cM(A(\cH_I))$ is compact, we may finish the proof by showing that $\Phi$ is surjective.
  If $\cH$ is a space on $\bB_d$, then the elements of $A(\cH_I)$ extend to continuous functions
  on $Z(I)$,
  as the multiplier norm dominates the supremum norm. If $\cH$ is a space on $\ol{\bB_d}$, they
  are already defined on $Z(I)$, so in both cases,
  every $\lambda \in Z(I)$ gives rise to a character
  $\delta_\lambda$ given by point evaluation at $\lambda$, and this character
  satisfies $\Phi(\delta_\lambda) = \lambda$.
\end{proof}

The character space of the whole multiplier algebra is often much more complicated.
Indeed, if $\cH$ is the Hardy space $H^2(\bD)$, then $\Mult(\cH) = H^\infty$,
an algebra whose
character space is known to be very complicated (see, for example, \cite[Chapter V]{Garnett07}).

Since every character on $\Mult(\cH_I)$ restricts to a character on $A(\cH_I)$, we obtain in the setting of the last lemma
a continuous map
\begin{equation*}
  \pi: \cM(\Mult(\cH_I)) \to Z(I), \quad \rho \mapsto \rho(z_1,\ldots,z_d).
\end{equation*}
This map is surjective, as evaluation at a point in $Z^0(I)$ is a character and the character space is compact.

If $\cH$ is a space on $\ol{\bB_d}$, then
$\Mult(\cH_I)$ consists of continuous functions on the compact set $Z(I)$. The weak* continuous characters are precisely
the point evaluations at points in $Z(I)$ by Proposition \ref{prop:alg_consistent_equiv}
and thus form a compact subset of $\cM(\Mult(\cH_I))$. The question
whether every character is a point evaluation in this setting remains open.

If $\cH$ is a space on $\bB_d$, then the weak* continuous characters are point evaluations at points in $Z^0(I)$,
again by Proposition \ref{prop:alg_consistent_equiv},
thus they form a proper subset of the maximal ideal space. The next lemma shows that in this case, multipliers
can oscillate wildly near the boundary of $\bB_d$, and hence the character space of $\Mult(\cH_I)$ is rather complicated.

\begin{lem}
  \label{lem:interpolating_sequences}
  Let $\cH$ be a unitarily invariant complete NP-space on $\bB_d$ and let $I \subsetneq \bC[z_1,\ldots,z_d]$ be a
  radical homogeneous ideal.
  Let $(\lambda_n)$ be a sequence in $Z^0(I)$
  with $\lim_{n \to \infty} ||\lambda_n||=1$. Then $(\lambda_n)$ contains
  a subsequence which is interpolating for $\Mult(\cH_I)$. In particular, $\pi^{-1} (\lambda)$
  contains a copy of $\beta \bN \setminus \bN$ for every $\lambda \in V(I) \cap \partial \bB_d$. 
\end{lem}

\begin{proof}
  The proof of Proposition 9.1 in \cite{DHS15}
  shows that it suffices to show that $K_I(\lambda_n, \lambda_n)$ converges to $\infty$, where
  $K_I$ denotes the reproducing kernel of $\cH_I$. However, if
  \begin{equation*}
    K(z,w) = \sum_{n=0}^\infty a_n \langle z,w \rangle^n \quad (z,w \in \bB_d)
  \end{equation*}
  denotes the reproducing kernel of $\cH$, then $K_I$ is simply the restriction of $K$ to $Z^0(I) \times Z^0(I)$.
  Moreover, since $\cH$ is algebraically consistent on $\bB_d$, we have
  $\sum_{n=0}^\infty a_n = \infty$ by Lemma \ref{lem:alg_cons_examples},
  thus $K_I(\lambda_n,\lambda_n)$ tends to $\infty$, as asserted.

  For the proof of the additional assertion, we note that for every $\lambda \in V(I) \cap \partial \bB_d$,
  there is an interpolating sequence $(\lambda_n)$ which converges to $\lambda$ by the first part.
  Hence, the algebra homomorphism
  \begin{equation*}
    \Mult(\cH_I) \to \ell^\infty, \quad \varphi \mapsto (\varphi(\lambda_n)),
  \end{equation*}
  is surjective, and its adjoint is a topological embedding of $\beta \bN \setminus \bN$ into
  $\pi^{-1}(\lambda)$.
\end{proof}

We now turn to the fibres of $\pi$ over points in the open ball. Let $\cH$ be a unitarily invariant
complete NP-space on $\bB_d$ or on $\ol{\bB_d}$, and let
\begin{equation*}
  \pi: \cM(\Mult(\cH)) \to \ol{\bB_d}, \quad \rho \mapsto (\rho(z_1),\ldots,\rho(z_d)),
\end{equation*}
be the map from above. For $\lambda \in \bB_d$, the fibre $\pi^{-1} (\lambda)$
always contains the character of evaluation at $\lambda$. If one allows
the case $d = \infty$, then these fibres can be much larger, see \cite{DHS15a}.
We say that $\cH$ is \emph{tame} if the fibres of $\pi$ over $\bB_d$ are singletons.
Equivalently, if $\rho$ is a character
on $\Mult(\cH)$ such that $\lambda = \pi(\rho) \in \bB_d$, then $\rho$ is the character of evaluation at $\lambda$.
Note that even if $\cH$ is a space on $\ol{\bB_d}$,
we do not impose a condition on fibres over the boundary.
It remains open whether there are non-tame spaces if $d < \infty$. We also mention that for spaces on $\bD$,
the question of when the fibres of $\pi$ are singletons already appears in \cite{Shields74} (see Question 3
on page 78).

\begin{exa}
  \label{exa:H_infty_tame}
  Perhaps the easiest example of a tame space is the Hardy space $H^2(\bD)$ on the unit disc.
  Let us briefly recall the well-known argument, which we will generalize below. Suppose $\rho$ is a character
  on $H^\infty(\bD) = \Mult(H^2(\bD))$ such that $\lambda = \rho(z) \in \bD$.
  If $\varphi \in H^\infty(\bD)$, then
  \begin{equation*}
    \varphi_\lambda = \frac{\varphi - \varphi(\lambda)}{z - \lambda} \in H^\infty(\bD)
  \end{equation*}
  by the maximum modulus principle,
  and
  \begin{equation*}
    \varphi = \varphi(\lambda) + (z - \lambda) \varphi_\lambda.
  \end{equation*}
  Since $\rho$ is a character and since $\rho(z - \lambda) = 0$, it follows that
  \begin{equation*}
    \rho(\varphi) = \varphi(\lambda),
  \end{equation*}
  thus $\rho$ is the character of evaluation at $\lambda$.
\end{exa}

\begin{rem}
  \label{rem:tame_restriction}
  If $\cH$ is tame and if $I \subsetneq \bC[z_1,\ldots,z_d]$ is a radical homogeneous ideal, then $\cH_I$ is similarly
  well-behaved. More precisely, if $\rho$ is a character on $\Mult(\cH_I)$ such that $\pi(\rho) \in \bB_d$ (and
  hence $\pi(\rho) \in Z^0(I)$), then $\rho$ is the character of evaluation at $\lambda$. Indeed, this follows
  from tameness of $\cH$ and from the fact that the restriction map from $\Mult(\cH)$ to $\Mult(\cH_I)$
  is surjective, since $\cH$ is a Nevanlinna-Pick space.
\end{rem}

Proposition 3.2 in \cite{DRS15} shows that $H^2_d$ is tame (for $d < \infty$).  The argument in \cite{DRS15}
uses a result about characters on the non-commutative free semigroup algebra $\cL_d$ from
\cite{DP98a}, and the fact that $\Mult(H^2_d)$ is a quotient of $\cL_d$ \cite{DP98}.
Since this does not apply to unitarily invariant complete NP-spaces besides $H^2_d$, we will use a different
argument similar to the one in Example \ref{exa:H_infty_tame}.
The underlying principle, however, is always the same, namely a factorization result for elements
in the Banach algebra in question.
In the following proposition, we record some sufficient conditions for tameness in decreasing order of generality.

\begin{prop}
  \label{prop:tame}
  Let $\cH$ be a unitarily invariant complete NP-space
  with reproducing kernel $K(z,w) = \sum_{n=0}^\infty a_n \langle z,w \rangle^n$.
  Consider the following conditions:
  \begin{enumerate}[label=\normalfont{(\alph*)}]
    \item Gleason's problem can be solved in $\Mult(\cH)$. That is, given $\lambda \in \bB_d$ and
      $\varphi \in \Mult(\cH)$, there are $\varphi_1,\ldots,\varphi_d \in \Mult(\cH)$ such that
      \begin{equation*}
        \varphi - \varphi(\lambda) = \sum_{i=1}^d (z_i - \lambda_i) \varphi_i.
      \end{equation*}
    \item For every $\lambda \in \bB_d$, the space
      \begin{equation*}
        \sum_{i=1}^d (z_i - \lambda_i) \cH \subset \cH
      \end{equation*}
      is closed in $\cH$.
    \item The condition
      \begin{equation*}
        \lim_{n \to \infty} \frac{a_n}{a_{n+1}} = 1
      \end{equation*}
      holds.
  \end{enumerate}
  Then $(c) \Rightarrow (b) \Rightarrow (a)$, and each of $(a),(b),(c)$ implies that $\cH$ is tame.
\end{prop}

\begin{proof}
  We first show that $(a)$ implies that $\cH$ is tame.
  Let $\rho$ be a character on $\Mult(\cH)$ with $\pi(\rho) = \lambda \in \bB_d$ and let
  $\varphi \in \Mult(\cH)$. By assumption,
  there are $\varphi_1,\ldots,\varphi_d \in \Mult(\cH)$ such that
  \begin{equation*}
    \varphi - \varphi(\lambda) = \sum_{i=1}^d (z_i - \lambda_i) \varphi_i.
  \end{equation*}
  Since the right-hand side is contained in the kernel of the multiplicative
  linear functional $\rho$, it follows that
  \begin{equation*}
    \rho(\varphi) = \varphi(\lambda),
  \end{equation*}
  hence $\rho$ is the character of point evaluation at $\lambda$.

  (b) $\Rightarrow$ (a) We use a factorization theorem for multipliers on complete Nevanlinna-Pick spaces to show that
  (a) is satisfied (cf. Section 4 of \cite{GRS05}). We first claim that
  \begin{equation*}
    \sum_{i=1}^d (z_i - \lambda_i) \cH = \{f \in \cH: f(\lambda) = 0 \}.
  \end{equation*}
  Indeed, to see the nontrivial inclusion, suppose that $f \in \cH$ vanishes at $\lambda$.
  Since the polynomials form a dense subset of $\cH$, there is a sequence $(p_n)$ of polynomials
  which converges to $f$ in $\cH$. Then $(p_n - p_n(\lambda))$ is a sequence of polynomials
  vanishing at $\lambda$ which converges to $f$, as evaluation at $\lambda$ is continuous. Observe that the space on the left-hand side contains
  all polynomials vanishing at
  $\lambda$ and is closed by assumption. Thus, $f$ belongs to the space on the left-hand side, as asserted.

  Hence, if $\varphi \in \Mult(\cH)$ with $\varphi(\lambda) = 0$, then $\ran(M_\varphi)$ is contained in the range
  of the row multiplication operator
  \begin{equation*}
  (M_{z_1 - \lambda_1}, \ldots, M_{z_d - \lambda_d}).
  \end{equation*}
  Let $z - \lambda$ denote the $\cB(\bC^d,\bC)$-valued multiplier $(z_1 - \lambda_1,\ldots,z_d - \lambda_d)$.
  Then by the Douglas lemma, there exists $c > 0$ such that
  \begin{equation*}
    M_\varphi M_\varphi^* \le c^2 M_{z-\lambda} M_{z - \lambda}^*.
  \end{equation*}
  In this situation, a factorization theorem valid for multiplier algebras
  of complete Nevanlinna-Pick spaces (see, for example, Theorem 8.57 in \cite{AM02}) implies the existence
  of a $\cB(\bC,\bC^d)$-valued multiplier $\Psi$ such that
  \begin{equation*}
    c (z- \lambda) \Psi = \varphi.
  \end{equation*}
  Writing
  \begin{equation*}
    \Psi =
    \begin{pmatrix}
      \psi_1 \\ \vdots \\ \psi_d
    \end{pmatrix},
  \end{equation*}
  we see that
  \begin{equation*}
    \varphi = \sum_{i=1}^d (z_i - \lambda_i) (c \psi_i).
  \end{equation*}
  Consequently, Gleason's problem can be solved in $\Mult(\cH)$, so (a) holds.

  (c) $\Rightarrow$ (b)
  The proof uses the notion of essential Taylor spectrum (see, for example, Section 33 and 34 in \cite{Muller07},
  Section 2.6 in \cite{EP96},
  or \cite{Curto81}).
  By Theorem 4.5 (2) in \cite{GHX04}, the assumption that $a_n/a_{n+1}$ converges to $1$ implies that
  the essential Taylor spectrum of $M_z = (M_{z_1},\ldots,M_{z_d})$
  equals $\partial \bB_d$,
  hence the $d$-tuple $(M_{z_1} - \lambda_1, \ldots,M_{z_d} - \lambda_d)$ is a Fredholm tuple for all
  $\lambda = (\lambda_1,\ldots,\lambda_d) \in \bB_d$.
  In particular, the last coboundary map in the Koszul complex has closed range, thus
  the row operator
  \begin{equation*}
    (M_{z_1} - \lambda_1, \ldots, M_{z_d} - \lambda_d)
  \end{equation*}
  has closed range for all $\lambda \in \bB_d$.
  Consequently, (b) holds.
\end{proof}

\begin{exa}
  The spaces $\cH_s(\bB_d)$, $\cH_s(\ol{\bB_d})$ and $\cK_\alpha$ in Example \ref{exa:unit_invariant}
  all satisfy condition (c) of the preceding proposition and are hence tame.
\end{exa}

\begin{rem}
  (a) The regularity condition $\lim_{n \to \infty} \frac{a_n}{a_{n+1}} = 1$ is not uncommon in the study of unitarily
  invariant kernels, see for example Section 4 in \cite{GRS02}. Proposition 4.5 in \cite{GRS02} shows that this condition
  automatically holds if $\sum_{n=0}^\infty a_n = \infty$ and $(a_n)$ is eventually decreasing.

  (b) If $(a_n)_n$ is log-convex (see part (b) of Remark \ref{rem:unitarily_invariant}),
  then $\lim_{n \to \infty} \frac{a_n}{a_{n+1}}$ always exists in $[0,\infty]$. Since $\cH$ is assumed to be algebraically
  consistent on $\bB_d$ or on $\ol{\bB_d}$, the power series $\sum_{n=0}^\infty a_n z^n$
  has radius of convergence $1$ (see Lemma \ref{lem:alg_cons_examples}),
  hence $\lim_{n \to \infty} \frac{a_n}{a_{n+1}} = 1$ is automatic in this case.

  (c) The construction of a complete Nevanlinna-Pick space $\cH$ on $\bD$ which violates condition (a)
  at $\lambda = 0$, and
  hence all conditions of the preceding proposition, is outlined in \cite[p. 326]{DHS15a}.
  It is not known if this space $\cH$ is tame.

  (d) The idea to use the factorization theorem to solve Gleason's
  problem in $\Mult(\cH)$ already appears in \cite{GRS05}, where this was done for the multiplier
  algebra of the Drury-Arveson space. The main difference between the two arguments is that in \cite{GRS05},
  it was shown that Gleason's problem can be solved in $\Mult(H^2_d)$ for $\lambda = 0$, and automorphism
  invariance of $\Mult(H^2_d)$ was used to deduce the general case. The argument here does not require automorphism
  invariance.
\end{rem}

We finish this section by observing that
tameness is also implied by the presence of a Corona theorem. In practice,
this result is of very limited use,
since establishing tameness is usually much easier than establishing a Corona theorem.
Indeed, it is very easy to see that $H^2(\bD)$ is tame (see Example \ref{exa:H_infty_tame}), whereas the Corona theorem for $H^\infty(\bD)
= \Mult(H^2(\bD))$ is hard.
Nevertheless, since there are no known examples of complete Nevanlinna-Pick spaces on $\bB_d$ for which the Corona
theorem fails, the next result explains the lack of examples of spaces which are not tame.

\begin{prop}
  Let $\cH$ be a unitarily invariant complete NP-space on $\bB_d$. If the set
  of all point evaluations at points in $\bB_d$ is weak* dense in the maximal ideal space of $\Mult(\cH)$,
  then $\cH$ is tame.
\end{prop}

\begin{proof}
  Let $\rho$ be a character on $\Mult(\cH)$ such that $\pi(\rho) = \lambda \in \bB_d$. By assumption,
  there is a net of points $(\lambda_\alpha)$ in $\bB_d$ such that $\delta_{\lambda_\alpha}$ converges
  to $\rho$ in the weak* topology. Hence, $\lambda_\alpha = \pi(\delta_{\lambda_\alpha})$
  converges to $\lambda = \pi(\rho)$. Since the multipliers are continuous on $\bB_d$,
  it follows that $\delta_{\lambda_\alpha}$ converges to $\delta_\lambda$ in the weak* topology,
  whence $\rho = \delta_\lambda$.
\end{proof}

\section{Holomorphic maps on homogeneous varieties}
\label{sec:hol_maps_on_varieties}

In the last section, we saw that the maximal ideal space of an algebra of the type
$A(\cH_I)$ or $\Mult(\cH_I)$ contains a copy of the homogeneous variety $Z^0(I)$. We will see in the next
section that under suitable conditions,
algebra homomorphisms between our algebras induce holomorphic maps between the varieties.
Thus, we will require some results about holomorphic maps on homogeneous varieties.

Throughout this section, let $I,J \subsetneq \bC[z_1,\ldots,z_d]$ be radical homogeneous ideals.
We say that a map $F: Z^0(I) \to \bC^{d'}$,
where $d' \in \bN$, is holomorphic if for every $z \in Z^0(I)$, there exists an open neighbourhood $U$
of $z$ and a holomorphic function $G$ on $U$ which agrees with $F$ on $U \cap Z^0(I)$.

We require the following variant of the maximum modulus principle.

\begin{lem}
  \label{lem:maximum_modulus}
  Let $F: Z^0(I) \to \overline{\bB_d}$ be a holomorphic map. If $F$ is not constant,
  then $F(Z^0(I)) \subset \bB_d$.
\end{lem}

\begin{proof}
  We may assume that $\{0\} \subsetneq Z^0(I)$. Suppose that there exists $w \in Z^0(I)$
  such that $||F(w)||=1$ and choose $\widetilde w \in Z^0(I)$ satisfying $w \in \bD \widetilde w$.
  The ordinary maximum modulus principle shows that the holomorphic function
  \begin{equation*}
    \bD \to \ol{\bD}, \quad t \mapsto \langle F(t \widetilde w), F(w) \rangle,
  \end{equation*}
  is the constant function $1$. Consequently, $F(t \widetilde w)= F(w)$ for all $t \in \ol{\bD}$,
  and in particular $F(0) = F(w) \in \partial \bB_d$. Now, if $z \in Z^0(I)$ is arbitrary, another
  application of
  the maximum modulus principle shows that the function
  \begin{equation*}
    \bD \to \ol{\bD}, \quad t \mapsto \langle F(t z), F(0) \rangle ,
  \end{equation*}
  is the constant function $1$, hence $F(z) = F(0)$. Thus, $F$ is constant.
\end{proof}

The next goal is to show that every biholomorphism between $Z^0(I)$ and $Z^0(J)$ which fixes the origin
is the restriction
of an invertible linear map. This result is Theorem 7.4 in \cite{DRS11}, where it was established by
adjusting the proof of Cartan's uniqueness theorem from \cite[Theorem 2.1.3]{Rudin08}. We provide
a simpler proof, which only uses the Schwarz lemma from ordinary complex analysis. This proof already appeared
in the author's Master's thesis \cite{Hartz12a}.
We begin with
the following variant of the Schwarz lemma.

\begin{lem}
  \label{lem:schwarz}
  Let $d' \in \bN$ and let $F: Z^0(I) \to \bB_{d'}$ be a holomorphic map such that $F(0) = 0$. Then $||F(z)|| \le ||z||$ for all
  $z \in Z^0(I)$. If equality holds for some $z \in Z^0(I) \setminus \{0\}$, then there exists $w_0 \in \partial \bB_{d'}$
  such that
  \begin{equation}
    \label{eqn:schwarz}
    F \Big( t \frac{z}{||z||} \Big) = t w_0
  \end{equation}
  for all $t \in \bD$. In particular, $F$ maps the disc $\bC z \cap \bB_d$ biholomorphically onto
  the disc $\bC F(z) \cap \bB_{d'}$ in this case.
\end{lem}

\begin{proof}
  We may assume that $\{0\} \subsetneq Z^0(I)$. Let $z \in Z^0(I) \setminus \{0\}$, suppose that $F(z) \neq 0$
  and define $w_0 = F(z) / ||F(z)||$. By the classical Schwarz lemma, the function
  \begin{equation*}
    f: \bD \to \bD, \quad t \mapsto \Big \langle F \Big(t \frac{z}{||z||} \Big), w_0  \Big \rangle,
  \end{equation*}
  satisfies $|f(t)| \le |t|$ for all $t \in \bD$. The first statement now follows by choosing $t = ||z||$.

  If $||F(z)|| = ||z||$, then $f(||z||) = ||z||$, thus $f$ is the identity by the Schwarz lemma. Since
  $||F(t \frac{z}{||z||}) || \le |t|$ for all $t \in \bD$ by the first part, Equation \eqref{eqn:schwarz} holds. The last assertion
  is now obvious.
\end{proof}

The desired result about biholomorphisms which fix the origin follows as an application of the last lemma.
\begin{prop}[{\cite[Theorem 7.4]{DRS11}}]
  \label{prop:bihol_0_linear}
  Let $F: Z^0(I) \to Z^0(J)$ be a biholomorphism such that $F(0) = 0$. Then there exists an invertible linear map $A$
  on $\bC^d$ which maps $V(I)$ isometrically onto $V(J)$ such that $A \big|_{Z^0(I)} = F$.
\end{prop}

\begin{proof}
  We may again assume that $\{0\} \subsetneq Z^0(I)$.
  Let $G$ be a holomorphic map which is defined on a neighbourhood $U$ of $0$ and which coincides with $F$ on $U \cap Z^0(I)$.
  Let $A_0$ be the derivative of $G$ at $0$. Lemma \ref{lem:schwarz}, applied to $F$ and its inverse, shows that $||F(z)|| = ||z||$
  for all $z \in Z^0(I)$, so the second part of the same lemma applies. Taking the derivative with respect to $t$
  in Equation \eqref{eqn:schwarz}
  for fixed $z \in Z^0(I) \setminus \{0\}$,
  we see that $w_0$ necessarily satisfies $w_0 ||z|| = A_0 z$, hence
  \begin{equation*}
    F(z) = ||z|| w_0 = A_0 z.
  \end{equation*}
  Thus, $A_0 \big|_{Z^0(I)} = F$, and $A_0$ is isometric on $Z^0(I)$ since $F$ is. Linearity of
  $A_0$ implies that $A_0$ maps $V(I)$ isometrically onto $V(J)$.

  Finally, the same argument, applied to $F^{-1}$ in place of $F$, shows that there exists a linear map $B_0$ on $\bC^d$
  such that $B_0 \big|_{Z^0(J)} = F^{-1}$. From this, we deduce that $A_0$ restricts to a linear isomorphism
  from the linear span of $Z^0(I)$ onto the linear span of $Z^0(J)$. Thus, if we let $A$ be an invertible
  extension of $A_0 \big|_{Z^0(I)}$ to $\bC^d$, then $A$ satisfies all the requirements of the proposition.
\end{proof}

We also crucially require a result from \cite{DRS11}, which, loosely speaking, allows us
to repair biholomorphisms which do not fix the origin. This result is contained in the proof of Proposition 4.7
in \cite{DRS11}.
The proof in \cite{DRS11} proceeds in two steps. In a first step, tools from algebraic
geometry and knowledge about the structure of conformal automorphisms of $\bB_d$ are used
to reduce the statement about arbitrary homogeneous varieties to the case of discs. The second
step, which deals with the case of discs, is an argument from plane conformal geometry.

It turns out that the first step, namely the reduction to discs, also follows immediately from Lemma \ref{lem:schwarz}.

\begin{lem}
  \label{lem:bihol_disc}
  Let $F: Z^0(I) \to Z^0(J)$ be a biholomorphism with $F(0) \neq 0$. Let $b = F(0)$ and let $a = F^{-1} (0)$.
  Then $||a|| = ||b||$ and $F$ maps the disc $D_1 = \bC a \cap \bB_d$ biholomorphically onto the disc $D_2 = \bC b \cap \bB_d$.
\end{lem}

\begin{proof}
  Let
  \begin{equation*}
    f: \bD \to Z^0(J), \quad t \mapsto F \Big( t \frac{a}{||a||} \Big),
  \end{equation*}
  and let $\varphi$ be an automorphism of $\bD$ which maps $0$ to $||a||$ and vice versa. Then $h = f \circ \varphi$ satisfies
  the assumptions of Lemma \ref{lem:schwarz}, hence
  \begin{equation*}
    ||b|| = ||h(||a||)|| \le ||a||.
  \end{equation*}
  By symmetry, $||a|| \le ||b||$, so $||a|| = ||b||$. It now follows from the second part of Lemma \ref{lem:schwarz}
  that $h$ maps $\bD$ biholomorphically onto the disc $D_2$. The result follows.
\end{proof}

The second step is essentially the following lemma.
For $\lambda \in \bT$, let $U_\lambda$ denote the unitary map on $\bC^d$ defined by
\begin{equation*}
  U_\lambda(z) = \lambda z
\end{equation*}
for $z \in \bC^d$.

\begin{lem}[Davidson-Ramsey-Shalit]
  \label{lem:auto_disc_turn}
  Let $\varphi$ be a conformal automorphism of $\bD$. The set
  \begin{equation*}
    \{ (U_\lambda \circ \varphi^{-1} \circ U_\mu \circ \varphi)(0): \lambda,\mu \in \bT\} \subset \bD
  \end{equation*}
  is a closed disc around $0$ which contains the point $\varphi^{-1}(0)$.
\end{lem}

\begin{proof} We repeat the relevant part of the proof of Theorem 7.4 in \cite{DRS11}.
  The assertion is trivial if $\varphi$ fixes the origin, so we may assume that $\varphi(0) \neq 0$. Then
  \begin{equation*}
    C = \{ (U_\mu \circ \varphi) (0): \mu \in \bT\}
  \end{equation*}
  is the circle around $0$ with radius $|\varphi(0)|$. 
  Since automorphisms of $\bD$ map circles to circles,
  it follows that the set $\varphi^{-1}(C)$ is a circle which obviously passes through $0$.
  Moreover, $\varphi^{-1}(0)$ is
  contained in the interior of the circle $\varphi^{-1}(C)$ as $0$ is contained in the interior of $C$.
  Thus
  \begin{equation*}
    \{ (U_\lambda \circ \varphi^{-1} \circ U_\mu \circ \varphi)(0): \lambda,\mu \in \bT\}
    = \{U_\lambda (\varphi^{-1}(C)): \lambda \in \bT \}
  \end{equation*}
  is a closed disc around $0$ which contains $\varphi^{-1}(0)$.
\end{proof}

Observe that if $I \subsetneq \bC[z_1,\ldots,z_d]$ is a radical homogeneous ideal, then
$U_\lambda$ leaves $Z^0(I)$ and $Z(I)$ invariant for each $\lambda \in \bT$.
Combining Lemmata \ref{lem:bihol_disc} and \ref{lem:auto_disc_turn}, we obtain the result
from \cite{DRS11} which allows us to repair biholomorphisms which do not fix the origin.

\begin{lem}[Davidson-Ramsey-Shalit]
  \label{lem:bihol_turn}
  Let $I,J \subsetneq \bC[z_1,\ldots,z_d]$ be radical homogeneous ideals and suppose that $F: Z^0(I) \to Z^0(J)$ is a biholomorphism.
  Then there are $\lambda,\mu \in \bT$ such that the biholomorphism
  \begin{equation*}
    F \circ U_\lambda \circ F^{-1} \circ U_\mu \circ F: Z^0(I) \to Z^0(J)
  \end{equation*}
  fixes the origin.
\end{lem}

\begin{proof}
  The assertion is trivial if $F(0) = 0$, so we may assume that $F(0) \neq 0$. It follows then from Lemma \ref{lem:bihol_disc}
  that it suffices to consider the case where $d=1$ and where $Z^0(I) = Z^0(J) = \bD$, the unit disc.
  An application of Lemma \ref{lem:auto_disc_turn} shows that there are $\lambda,\mu \in \bT$ such that
  \begin{equation*}
    F^{-1}(0) =( U_\lambda \circ F^{-1} \circ U_\mu \circ F) (0),
  \end{equation*}
  hence $F \circ U_\lambda \circ F^{-1} \circ U_\mu \circ F$ fixes the origin.
\end{proof}

We finish this section by giving another application of the crucial Lemma \ref{lem:auto_disc_turn}
of Davidson, Ramsey and Shalit.
We will show that
the group of unitaries is a maximal subgroup of $\Aut(\bB_d)$, the group of conformal
automorphisms of $\bB_d$. Since the group $\Aut(\bB_d)$ is well studied,
it is likely that this has been observed before. Nevertheless, even when $d=1$, the only result in this direction
that seems to be widely known is the fact that the group of unitaries is a maximal compact subgroup of $\Aut(\bB_d)$.

Recall that for $a \in \bB_d$, there exists an automorphism $\varphi_a$ of $\bB_d$ defined by
\begin{equation*}
  \varphi_a(z) = \frac{a-P_a z - s_a Q_a z}{1 - \langle z,a \rangle} \quad (z \in \bB_d),
\end{equation*}
where $P_a$ is the orthogonal projection of $\bC^d$ onto the subspace spanned by $a$, $Q_a = I - P_a$
and $s_a = (1 - |a|^2)^{1/2}$. Then $\varphi_a$ is an involution
which interchanges $0$ and $a$ (see, for example, \cite[Theorem 2.2.2]{Rudin08}).
Moreover, every $\varphi \in \Aut(\bB_d)$ is of the form $\varphi = U \circ \varphi_a$, where
$U$ is unitary and $a = \varphi^{-1}(0)$ \cite[Theorem 2.2.5]{Rudin08}. We begin with a preliminary lemma.
\begin{lem}
  \label{lem:maximal_subgroup_prelim}
  Let $G \subset \Aut(\bB_d)$ be a subsemigroup which contains all unitary maps and let
  $\cO$ denote the orbit of $0$ under $G$. Then the following assertions hold:
  \begin{enumerate}[label=\normalfont{(\alph*)}]
    \item $G$ is a subgroup of $\Aut(\bB_d)$.
    \item A point $a \in \bB_d$ belongs to $\cO$ if and only if $\varphi_a \in G$.
    \item $G= \Aut(\bB_d)$ if and only if $\cO = \bB_d$.
  \end{enumerate}
\end{lem}

\begin{proof}
  (a) If $\varphi \in G$, then $\varphi = U \varphi_a$ for some unitary map $U$ and $a \in \bB_d$. Then $\varphi_a \in G$.
  Since $\varphi_a$ is an involution, it follows that $\varphi^{-1} = \varphi_a U^{-1} \in G$. Hence, $G$
  is a group.

  (b)
  For the proof of the non-trivial implication, suppose that $a \in \cO$
  and let $\varphi \in G$ with
  $a = \varphi(0)$. Then $\varphi^{-1} \in G$ by part (a) and $(\varphi^{-1})^{-1}(0) = a$, hence
  \begin{equation*}
    \varphi^{-1} = U \circ \varphi_a
  \end{equation*}
  for some unitary map $U$. Since $U \in G$, it follows that $\varphi_a \in G$, as asserted.

  (c) This follows immediately from (b) and the description of the automorphisms of $\bB_d$ in terms of unitary maps
  and the involutions $\varphi_a$.
\end{proof}

We now show that the group of rotations is a maximal subgroup of the group $\Aut(\bD)$. We will then deduce
the higher-dimensional analogue from this result.
\begin{lem}
  \label{lem:rot_maximal_subgroup}
  The group of rotations is a maximal subgroup of the group $\Aut(\bD)$.
\end{lem}

\begin{proof}
  Let $G$ be a subgroup of $\Aut(\bD)$ which properly contains the group of rotations.
  Let $\cO$ be the orbit of $0$ under $G$. We wish to show that $\cO = \bD$, which is equivalent
  to the assertion by part (c) of Lemma \ref{lem:maximal_subgroup_prelim}.
  
  We first claim that $\ol{\bD} \cO \subset \cO$.
  To this end, let $a \in \cO$. Then $\varphi_a \in G$ by part (b) of Lemma \ref{lem:maximal_subgroup_prelim}.
  An application of Lemma \ref{lem:auto_disc_turn} now shows that
  $\cO$ contains the closed disc of radius $|a|$ around $0$, which proves the claim.

  We finish the proof by showing that $\cO$ contains points of modulus arbitrarily close to $1$.
  Since $G$ contains a non-rotation automorphism, $\cO \neq \{0\}$. Clearly, $\cO$ is rotationally
  invariant, hence there exists $r > 0$ such that $r \in \cO$ and therefore $\varphi_r \in G$ by part (b) of Lemma \ref{lem:maximal_subgroup_prelim}. Consider the hyperbolic automorphism $f$ defined by
  \begin{equation*}
    f(z) = \varphi_r(-z) = \frac{r + z}{1 + r z}
  \end{equation*}
  for $z \in \bD$. Then $f \in G$. Moreover, it is well known and easy to see that
  \begin{equation*}
    \lim_{n \to \infty} f^n (0) = 1,
  \end{equation*}
  where $f^n$ denotes the $n$-fold iteration of $f$. Thus, the proof is complete.
\end{proof}

We are now ready to prove a multivariate analogue of the last lemma.
\begin{prop}
  \label{prop:unitary_max_subsemi}
  The group of unitary maps on $\bC^d$ is a maximal subsemigroup of $\Aut(\bB_d)$.
\end{prop}

\begin{proof}
  Suppose that $G$ is a subsemigroup of $\Aut(\bB_d)$ which properly contains the group of unitary maps.
  Then $G$ is a subgroup by part (a) of Lemma \ref{lem:maximal_subgroup_prelim}.
  Let $\cO$ denote the orbit of $0$ under $G$.
  Our goal is to show that $\cO = \bB_d$ (see part (c) of Lemma \ref{lem:maximal_subgroup_prelim}).
  Since $G$ contains
  all unitaries, it suffices to show that $\bD e_1 \subset \cO$, where $e_1$
  denotes the first standard basis vector of $\bC^d$.

  To this end, let
  \begin{equation*}
    H = \{\varphi \in G: \varphi(\bD e_1) = \bD e_1 \}.
  \end{equation*}
  Identifying $\bD e_1$ with $\bD$ we obtain a subgroup
  \begin{equation*}
    \widetilde H = \{ \varphi \big|_{\bD} : \varphi \in H \} 
  \end{equation*}
  of $\Aut(\bD)$. Clearly, $\widetilde H$ contains all rotations $U_\lambda$ for $\lambda \in \bT$.
  Since $G$ contains a non-unitary automorphism, $\{0\} \neq \mathcal \cO$. Moreover, $\cO$ is invariant
  under unitary maps, hence there exists $r > 0$ such that $r e_1 \in \cO$ and thus $\varphi_{r e_1} \in G$
  by part (b) of Lemma \ref{lem:maximal_subgroup_prelim}. Observe that $\varphi_{r e_1} \in H$, so $\widetilde H$
  contains the non-rotation automorphism $\varphi_r$. It now follows from Lemma \ref{lem:rot_maximal_subgroup}
  that $\widetilde H = \Aut(\bD)$. Since $\Aut(\bD)$ acts transitively on $\bD$, the definition
  of $\widetilde H$ implies that $\bD e_1 \subset \cO$, which completes the proof.
\end{proof}

There is an immediate consequence for collections of functions on $\bB_d$ which are unitarily invariant.

\begin{cor}
  Let $S \neq \emptyset$ be a collection of functions on $\bB_d$ and define
  \begin{equation*}
    G = \{\varphi \in \Aut(\bB_d): f \circ \varphi \in S \text{ for all }f \in S\}.
  \end{equation*}
  Assume that $G$ contains $\cU$, the group of unitary maps on $\bC^d$. Then either $G = \cU$ or $G = \Aut(\bB_d)$.
\end{cor}

\begin{proof}
  It is clear that $G$ is a subsemigroup of $\Aut(\bB_d)$, so the result follows from Proposition
  \ref{prop:unitary_max_subsemi}.
\end{proof}

The last result applies in particular to reproducing kernel Hilbert spaces $\cH$ on $\bB_d$ with a kernel
of the form
\begin{equation*}
  K(z,w) = \sum_{n=0}^\infty a_n \langle z,w \rangle^n \quad (z,w \in \bB_d).
\end{equation*}
In this case, by the closed graph theorem, $G$ is also the set of all automorphisms
of $\bB_d$ which induce a bounded composition operator on $\cH$. Moreover, $G$ contains all unitaries.
Thus, the last result says that such a space $\cH$ is either invariant under all automorphisms of $\bB_d$, or
under unitaries only.

\section{Existence of graded isomorphisms}
\label{sec:existence_graded}

The question of when two algebras of the type $\Mult(\cH_I)$ are isomorphic
is more difficult than the question about equality of multiplier algebras studied
in Section \ref{sec:graded_spaces}.
The chief reason is that isomorphisms do not necessarily respect the grading. Thus, our
goal is to establish the existence of graded isomorphisms. As in \cite{DRS11}, this will follow
from an application of Lemma \ref{lem:bihol_turn}.

Throughout this section, let $\cH$ and $\cK$ be unitarily invariant complete NP-spaces
on $\bB_d$ or on $\ol{\bB_d}$ and let $I,J \subsetneq \bC[z_1,\ldots,z_d]$ be radical homogeneous
ideals. We allow the case where $\cH$ is a space on $\bB_d$, and $\cK$ is a space on $\ol{\bB_d}$, or vice versa.
We will consider the multiplier algebras $\Mult(\cH_I)$ and $\Mult(\cK_J)$, as well
as their norm closed versions $A(\cH_I)$ and $A(\cK_J)$.
To cover both cases, we first study homomorphisms from $A(\cH_I)$ into $\Mult(\cK_J)$.
We identify the maximal ideal space of $A(\cH_I)$ with $Z(I)$ by
Lemma \ref{lem:gelfand_A_H}. Similarly, we identify $Z^0(J)$ with a subset
of the maximal ideal space of $\Mult(\cK_J)$ via point evaluations.
The following lemma should be compared to Proposition 7.1 and Lemma 11.5 in \cite{DRS11}.

\begin{lem}
  \label{lem:phi_star_hol_homo}
  Let $\cH$ and $\cK$ as well as $I,J \subsetneq \bC[z_1,\ldots,z_d]$ be as above.
  \begin{enumerate}[label=\normalfont{(\alph*)}]
    \item If $\Phi: A(\cH_I) \to \Mult(\cK_J)$ is an injective unital homomorphism,
      then $\Phi^*$ maps $Z^0(J)$ holomorphically into $Z^0(I)$.
    \item If $\Phi: \Mult(\cH_I) \to \Mult(\cK_J)$ is an injective unital
      homomorphism and weak*-weak* continuous, then $\Phi^*$ maps $Z^0(J)$
      holomorphically into $Z^0(I)$.
    \item If $\Phi: \Mult(\cH_I) \to \Mult(\cK_J)$ is an injective unital
      homomorphism, and if $\cH$ is tame, then $\Phi^*$ maps $Z^0(J)$ holomorphically
      into $Z^0(I)$, and $\Phi$ is weak*-weak* continuous.
  \end{enumerate}
\end{lem}

\begin{proof}
  (a) Clearly, $\Phi^*$ maps $Z^0(J)$ into $Z(I)$, and the $j$-th coordinate
  of $\Phi^*$ is given by $\Phi(z_j) \in A(\cH_I)$, hence $F = \Phi^* \big|_{Z^0(J)}$ is holomorphic.
  Lemma \ref{lem:maximum_modulus} shows that
  the range of $F$ contains points in $\partial \bB_d$
  only if $F$ is constant. In this case, $\Phi(z_j) = \lambda_j$,
  where $(\lambda_1,\ldots,\lambda_d) \in \partial \bB_d$.
  Since $\Phi$ is unital and injective, it follows that $\lambda_j = z_j$ on $Z^0(I)$,
  which is absurd. Thus, the range of $F$ is contained in $Z^0(I)$.

  (b) By definition of the map
  $\pi: \cM(\Mult(\cH_I)) \to Z(I)$, part (a) implies that
  $\pi \circ \Phi^*$ is holomorphic and maps $Z^0(J)$ into $Z^0(I)$. Since $\Phi$
  is weak*-weak* continuous,
  $\Phi^*(Z^0(J))$ consists of point evaluations by Lemma \ref{lem:hom_weak_star}, so the assertion follows.

  (c) Again by part (a),
  $\pi \circ \Phi^*$ is holomorphic and maps $Z^0(J)$ into $Z^0(I)$. Since
  $\cH$ is tame, we conclude that $\Phi^*$ maps $Z^0(J)$ into the set
  (of point evaluations at points in) $Z^0(I)$ (see Remark \ref{rem:tame_restriction}). If $\cK$ is a space
  on $\bB_d$,
  Lemma \ref{lem:hom_weak_star} therefore implies that $\Phi$ is weak*-weak* continuous.
  Now, assume that $\cK$ is a space on $\ol{\bB_d}$. If $\cH$ is a space on $\ol{\bB_d}$ as well,
  then $\Phi^*(Z(J)) \subset Z(I)$ by continuity of $\Phi^*$, thus $\Phi$ is again weak*-weak* continuous
  by Lemma \ref{lem:hom_weak_star}.
  
  It remains to consider the case where $\cH$ is a space on $\bB_d$
  and $\cK$ is space on $\ol{\bB_d}$. We claim that $\Phi^*(Z^0(J))$
  is contained in a ball of radius $r < 1$. This will finish the proof,
  as $\Phi^*(Z(J)) \subset r Z(I) \subset Z^0(I)$ by continuity, so once again, the assertion
  follows from Lemma \ref{lem:hom_weak_star}. Suppose that $\Phi^*(Z^0(J))$
  contains a sequence $(\Phi^*(\lambda_n))$ with $||\Phi^*(\lambda_n)|| \to 1$. By passing
  to a subsequence, we may assume that $(\lambda_n)$ converges
  to a point $\lambda \in Z(J)$.
  Lemma \ref{lem:interpolating_sequences}
  shows that there is a multiplier $\varphi \in \Mult(\cH_I)$ such that
  $(\varphi(\Phi^*(\lambda_n)))$ does not converge. However,
  \begin{equation*}
    \varphi(\Phi^*(\lambda_n)) = (\Phi(\varphi))(\lambda_n),
  \end{equation*}
  and $\Phi(\varphi) \in \Mult(\cK_J)$ is a continuous function on $Z(J)$.
  This is a contradiction, and the proof is complete.
\end{proof}

For isomorphisms, we obtain the following consequence.
\begin{cor}
  \label{cor:phi_star_hol_iso}
  Let $\cH$ and $\cK$ as well as $I,J \subsetneq \bC[z_1,\ldots,z_d]$ be as above.
  \begin{enumerate}[label=\normalfont{(\alph*)}]
    \item If $\Phi: A(\cH_I) \to A(\cH_J)$ is an isomorphism, then $\Phi^*$ maps $Z^0(J)$
      biholomorphically onto $Z^0(I)$.
    \item Let $\Phi: \Mult(\cH_I) \to \Mult(\cH_J)$ be an isomorphism, and assume that
      $\cH$ is tame or that $\Phi$ is weak*-weak* continuous. Then $\Phi^*$ maps $Z^0(J)$
      biholomorphically onto $Z^0(I)$, and $\Phi$ is a weak*-weak* homeomorphism.
  \end{enumerate}
\end{cor}

\begin{proof}
  (a) immediately follows from part (a) of the preceding lemma.

  (b) By part (c) of the last lemma, $\Phi$ is weak*-weak* continuous in both cases. Since it is also a homeomorphism
  in the norm topologies, the Krein-Smulian theorem combined with weak* compactness of the unit balls shows
  that $\Phi^{-1}$ is weak*-weak* continuous as well (see, for example, the argument at the end of \cite{DHS15a}).
  Thus, part (b) of the last lemma also applies to $\Phi^{-1}$,
  so that $\Phi^*$ is a biholomorphism between $Z^0(J)$ and $Z^0(I)$.
\end{proof}

For $n \in \bN$, let $(\cH_I)_n$ denote the space of all homogeneous elements of $\cH_I$ of degree $n$.
Recall that $(\cH_I)_n \subset A(\cH_I)$ for all $n \in \bN$.
We say that a homomorphism $\Phi: A(\cH_I) \to \Mult(\cK_J)$ is graded
if
\begin{equation*}
  \Phi( (\cH_I)_n ) \subset (\cK_J)_n
\end{equation*}
for all $n \in \bN$. Graded isomorphisms admit a particularly simple description in terms
of their adjoint.

\begin{lem}
  \label{lem:graded_char}
  Let $\cH$ and $\cK$ as well as $I,J \subsetneq \bC[z_1,\ldots,z_d]$ be as above,
  and suppose that $\Phi: A(\cH_I) \to A(\cK_J)$ is an isomorphism (respectively that
  $\Phi: \Mult(\cH_I) \to \Mult(\cK_J)$ is a weak*-weak* continuous isomorphism). Then the following are equivalent:
  \begin{enumerate}[label=\normalfont{(\roman*)}]
    \item $\Phi$ is graded.
    \item $\Phi^*(0) = 0$.
    \item There exists an invertible linear map $A$ on $\bC^d$ which maps $V(J)$ isometrically onto $V(I)$
      such that $\Phi$ is given
      by composition with $A$, that is,
      \begin{equation*}
        \Phi(\varphi) = \varphi \circ A
      \end{equation*}
     for all $\varphi \in A(\cH_I)$ (respectively $\varphi \in \Mult(\cH_I)$).
  \end{enumerate}
\end{lem}

\begin{proof}
  (iii) $\Rightarrow$ (i) is obvious.

  (i) $\Rightarrow$ (ii)
  Let $\lambda = \Phi^*(0) \in Z(I)$. If $\lambda \neq 0$,
  then there is a homogeneous element $\varphi \in A(\cH_I)$ of degree $1$ such that $\varphi(\lambda) \neq 0$.
  Corollary \ref{cor:phi_star_hol_iso} implies that $\Phi^*(0) \in Z^0(I)$, hence
  \begin{equation*}
    \Phi(\varphi) (0) = \varphi( \Phi^*(0)) = \varphi(\lambda) \neq 0.
  \end{equation*}
  In particular, $\Phi(\varphi)$ is not homogeneous of degree $1$, hence $\Phi$ is not graded.

  (ii) $\Rightarrow$ (iii) By Corollary \ref{cor:phi_star_hol_iso}, $\Phi^*$ maps $Z^0(J)$ biholomorphically
  onto $Z^0(I)$. Since $\Phi^*(0) = 0$, Proposition \ref{prop:bihol_0_linear} therefore yields an invertible
  linear map $A$ which maps $V(J)$ isometrically onto $V(I)$ such that $\Phi^*$ coincides
  with $A$ on $Z^0(J)$. It follows that
  \begin{equation*}
    \Phi(\varphi) = \varphi \circ A
  \end{equation*}
  on $Z^0(J)$ for all $\varphi \in A(\cH_I)$ (respectively $\varphi \in \Mult(\cH_I)$). Moreover,
  if $\Phi$ is a map from $A(\cH_I)$ onto $A(\cK_J)$, then this identity holds on $Z(J)$ by continuity.

  Assume now that $\Phi$ is a map from $\Mult(\cH_I)$ onto $\Mult(\cK_J)$. If $\cK$ is a space on $\bB_d$, we are done.
  If $\cK$ and $\cH$ are spaces on $\ol{\bB_d}$, then $\Phi(\varphi) = \varphi \circ A$ again holds
  on $Z(J)$ by continuity. We finish the proof by showing that the remaining case where $\cK$ is a space
  on $\ol{\bB_d}$, $\cH$ is a space on $\bB_d$ and $V(J) \neq \{0\}$ does not occur. Indeed, in this 
  case, $V(I) \neq \{0\}$ and
  $\Phi^*$ would map $Z(J)$ onto a necessarily compact subset of $Z^0(I)$ by Lemma \ref{lem:hom_weak_star}.
  This contradicts the fact that $\Phi^*$ maps $Z^0(J)$ onto $Z^0(I)$.
\end{proof}

We mention that in the case where $\cH = \cK = H^2_d$, the Drury-Arveson space, isomorphisms
as above are called vacuum-preserving in \cite{DRS11}.

The desired consequence about the existence of graded isomorphisms is the following result.
\begin{prop}
  \label{prop:graded_isomorphism}
  Let $\cH$ and $\cK$ as well as $I,J \subsetneq \bC[z_1,\ldots,z_d]$ be as above.
  \begin{enumerate}[label=\normalfont{(\alph*)}]
    \item If $A(\cH_I)$ and $A(\cK_J)$ are algebraically (respectively isometrically) isomorphic, then
      there exists a graded algebraic (respectively isometric) isomorphism from $A(\cH_I)$ onto $A(\cK_J)$.
    \item If $\Mult(\cH_I)$ and $\Mult(\cK_J)$ are algebraically (respectively isometrically) isomorphic via
      a weak*-weak* continuous isomorphism, then
      there exists a graded weak*-weak* continuous algebraic (respectively isometric) isomorphism
      from $\Mult(\cH_I)$ onto $\Mult(\cK_J)$.
  \end{enumerate}
\end{prop}

\begin{proof}
  By Lemma \ref{lem:graded_char}, it suffices to show in each case that there exists an isomorphism
  whose adjoint fixes the origin. We will achieve this by applying Corollary \ref{cor:phi_star_hol_iso} and
  Lemma \ref{lem:bihol_turn}. To this end, observe that
  for $\lambda \in \bT$, the unitary map $U_\lambda$ on $\bC^d$ given by multiplication with $\lambda$
  induces a unitary composition operator $C_{U_\lambda}$ on $\cH_I$. If $\varphi \in \Mult(\cH_I)$, then
  \begin{equation*}
    C_{U_\lambda} M_\varphi C_{U_\lambda}^* = M_{\varphi \circ U_\lambda},
  \end{equation*}
  hence $\Phi_\lambda^I(M_{\varphi}) = C_{U_\lambda} M_\varphi C_{U_\lambda}^*$
  defines an isometric, weak*-weak* continuous automorphism of $\Mult(\cH_I)$ which maps
  $A(\cH_I)$ onto $A(\cH_I)$. Clearly, the adjoint of this automorphism, restricted to $Z^0(I)$, is given
  by multiplication with $U_\lambda$. The same result holds for $\cK_J$ in place of $\cH_I$.
  
  Suppose now that $\Phi$ is an isomorphism between $A(\cH_I)$ and $A(\cK_J)$ (respectively a weak*-weak* continuous
  isomorphism between $\Mult(\cH_I)$ and $\Mult(\cK_J)$). By Corollary \ref{cor:phi_star_hol_iso}, the adjoint
  $\Phi^*$ maps $Z^0(J)$ biholomorphically onto $Z^0(I)$. From Lemma \ref{lem:bihol_turn}, we infer that
  there exist $\lambda,\mu \in \bT$ such that the map
  \begin{equation*}
    \Phi^* \circ U_\lambda \circ (\Phi^*)^{-1} \circ U_\mu \circ \Phi^*
  \end{equation*}
  fixes the origin. This map is the adjoint of
  \begin{equation*}
    \Phi \circ \Phi^I_\mu \circ \Phi^{-1} \circ \Phi^J_\lambda \circ \Phi,
  \end{equation*}
  which is an isomorphism between $A(\cH_I)$ and $A(\cK_J)$ (respectively a weak*-weak* continuous isomorphism between $\Mult(\cH_I)$
  and $\Mult(\cK_J)$). Moreover, it is isometric if $\Phi$ is isometric, which finishes the proof.
\end{proof}

\section{Isomorphism results}
\label{sec:iso_results}

We are now ready to establish the main results about isomorphism of multiplier algebras of spaces
of the type $\cH_I$.
We will usually make an assumption
which guarantees that the Hilbert function spaces have dimension at least $2$. In projective algebraic geometry,
the maximal ideal of $\bC[z_1,\ldots,z_d]$ which is generated by the coordinate functions $z_1,\ldots,z_d$
is called the \emph{irrelevant ideal} (see \cite[Chapter VII]{ZS75}).
This is because the vanishing locus of this ideal in $\bC^d$ is just the origin, hence the projective
vanishing locus in $\bP^{d-1}(\bC)$ is empty. We will say that a radical homogeneous ideal of $\bC[z_1,\ldots,z_d]$
is \emph{relevant} if it is proper and not equal to the irrelevant ideal. By the projective Nullstellensatz,
the projective vanishing locus of every such ideal $I$ is not empty, thus $Z^0(I) \subset \bC^d$ always contains
a disc.

\begin{prop}
  \label{prop:structure_graded_algebraic}
  Let $\cH$ and $\cK$ be unitarily invariant complete NP-spaces, and let $I$ and $J$
  be relevant radical homogeneous ideals in $\bC[z_1,\ldots,z_d]$.
  Let $\Phi: A(\cH_I) \to A(\cK_J)$ be a graded algebraic isomorphism (respectively $\Phi: \Mult(\cH_I) \to \Mult(\cK_J)$
  a graded weak*-weak* continuous isomorphism).
  
  Then $\cH = \cK$ as vector spaces, and there exists an invertible linear
  map $A$ which maps $V(J)$ isometrically onto $V(I)$ such that $\Phi$ is given by composition with $A$.
  Moreover, $A$ induces a bounded invertible composition operator
  \begin{equation*}
    C_A: \cH_I \to \cK_J, \quad f \mapsto f \circ A,
  \end{equation*}
  such that
  \begin{equation*}
    \Phi(M_\varphi) = C_A M_\varphi (C_A)^{-1}
  \end{equation*}
  for all $\varphi \in A(\cH_I)$ (respectively $\varphi \in \Mult(\cH_I)$). In particular, $\Phi$ is given by a similarity.
\end{prop}

\begin{proof}
  By Lemma \ref{lem:graded_char}, there exists an invertible linear map $A$ which maps $V(J)$ isometrically
  onto $V(I)$ and such that $\Phi$ is given by composition with $A$. Since all Banach algebras under consideration
  are semi-simple,
  $\Phi$ and its inverse are (norm) continuous (see \cite[Proposition 4.2]{Dales78}).
  Thus, if $f \in \cH_I$ is homogeneous, then Proposition \ref{prop:norm_hom} shows that
  \begin{equation*}
    || f \circ A||_{\cK_J} = || f \circ A||_{\Mult(\cK_J)} \le ||\Phi|| \, ||f||_{\Mult(\cH_I)} = ||\Phi|| \,
    ||f||_{\cH_I},
  \end{equation*}
  so there exists a bounded operator $C_A: \cH_I \to \cK_J$ such that
  \begin{equation*}
    C_A f = f \circ A
  \end{equation*}
  holds for every polynomial $f$, and hence for all $f \in \cH_I$.
  Consideration of $\Phi^{-1}$ shows that $C_A$ is invertible.
  Moreover, for $\varphi \in \Mult(\cH_I)$ and $f \in \cK_J$, we have
  \begin{equation*}
    C_A M_\varphi (C_A)^{-1} f = (\varphi \circ A) f,
  \end{equation*}
  hence $\Phi$ is given by conjugation with $C_A$.

  We finish the proof by showing that $\cH$ and $\cK$ coincide as vector spaces. To this end, let
  \begin{equation*}
    K_{\cH} (z,w) = \sum_{n=0}^\infty a_n \langle z,w \rangle^n
  \end{equation*}
  and
  \begin{equation*}
    K_{\cK} (z,w) = \sum_{n=0}^\infty a_n' \langle z,w \rangle^n
  \end{equation*}
  denote the reproducing kernels of $\cH$ and $\cK$, respectively.
  Since $I$ and $J$ are radical, Lemma \ref{lem:nullstellensatz} implies that
  the restriction maps $R_I: \cH \ominus I \to \cH_I$ and $R_J: \cK \ominus J \to \cK_J$ are unitary.
  Let
  \begin{equation*}
    T_A = R_I^{-1} (C_A)^* R_J \in \cB(\cK \ominus J, \cH \ominus I).
  \end{equation*}
  Then $T_A$ is bounded and invertible, and Lemma \ref{lem:comp_operators} combined with
  Lemma \ref{lem:nullstellensatz} implies that
  \begin{equation*}
    T_A K_{\cK}(\cdot,w) = K_{\cH} (\cdot, A w)
  \end{equation*}
  for all $w \in Z^0(J)$. Using the homogeneity of $J$, it is easy to deduce from
  $K_{\cK}(\cdot,w) \in \cK \ominus J$ for $w \in Z^0(J)$ that $\langle \cdot,w \rangle^n \in \cK \ominus J$
  for all $w \in Z^0(J)$ and all $n \in \bN$. Similarly, $\langle \cdot,z \rangle^n \in \cH \ominus I$
  for all $z \in Z^0(I)$ and all $n \in \bN$.
  Moreover, $C_A$ and hence $T_A$ respects the degree of homogeneous polynomials.
  Consequently,
  \begin{equation}
    \label{eq:T_A}
    T_A a_n' \langle \cdot,w \rangle^n = a_n \langle \cdot, A w \rangle^n
  \end{equation}
  for all $n \in \bN$ and all $w \in V(J)$.
  Using part (d) of Remark \ref{rem:unitarily_invariant} and the fact that $||A w|| = ||w||$, we see that
  \begin{equation*}
    ||a_n' \langle \cdot,w \rangle^n||^2_{\cK_J} = a_n' ||w||^{2 n}
  \end{equation*}
  and that
  \begin{equation*}
    ||a_n \langle \cdot, A w \rangle^n||^2_{\cH_I} = a_n ||w||^{2 n}.
  \end{equation*}
  Since $J$ is relevant, $V(J)$ contains a non-zero vector $w$, hence
  \begin{equation*}
    ||(C_A^*)^{-1}||^2 \le \frac{a_n}{a_n'} \le ||C_A^*||^2
  \end{equation*}
  for all $n \in \bN$, from which it immediately follows that $\cH = \cK$ as vector spaces (see part (d)
  of Remark \ref{rem:unitarily_invariant}).
\end{proof}

Using the same methods as in the last proof, we obtain a version of Proposition \ref{prop:structure_graded_algebraic}
for isometric isomorphisms.
\begin{prop}
  \label{prop:structure_graded_isometric}
  Let $\cH$ and $\cK$ be unitarily invariant complete NP-spaces, and let $I$ and $J$
  be relevant radical homogeneous ideals in $\bC[z_1,\ldots,z_d]$.
  Let $\Phi: A(\cH_I) \to A(\cK_J)$ be a graded isometric isomorphism (respectively $\Phi: \Mult(\cH_I) \to \Mult(\cK_J)$
  a graded weak*-weak* continuous isometric isomorphism).
  
  Then $\cH = \cK$ as Hilbert spaces, and there exists a unitary
  map $U$ which maps $V(J)$ onto $V(I)$ such that $\Phi$ is given by composition with $U$.
  Moreover, $U$ induces a unitary composition operator
  \begin{equation*}
    C_U: \cH_I \to \cK_J, \quad f \mapsto f \circ U,
  \end{equation*}
  such that
  \begin{equation*}
    \Phi(M_\varphi) = C_U M_\varphi (C_U)^{-1}
  \end{equation*}
  for all $\varphi \in A(\cH_I)$ (respectively $\varphi \in \Mult(\cH_I)$). In particular, $\Phi$ is unitarily implemented.
\end{prop}

\begin{proof}
  Proposition \ref{prop:structure_graded_algebraic} and its proof show that there exists an invertible linear
  map $U$ which maps $V(J)$ isometrically onto $V(I)$ such that $U$ induces a unitary composition operator
  \begin{equation*}
    C_U: \cK_J \to \cH_I, \quad f \mapsto f \circ U,
  \end{equation*}
  and such that $\Phi$ is given by conjugation with $C_U$. Since $C_U$ is a unitary operator, the last part of the proof
  of Proposition \ref{prop:structure_graded_algebraic} shows that $a_n = a_n'$ for all $n \in \bN$ in the notation
  of the proof, and hence
  $\cH = \cK$ as Hilbert spaces.

  Finally, setting $n=1$ in Equation \eqref{eq:T_A}, we see that
  \begin{equation*}
    T_U \langle \cdot,w \rangle = \langle \cdot, U w \rangle
  \end{equation*}
  for all $w \in V(J)$, and hence for all $w$ in the linear span of $V(J)$.
  Since $T_U$ is a unitary operator, part (d) of Remark \ref{rem:unitarily_invariant} implies that
  $U$ is isometric on the linear span of $V(J)$. Hence, $U$ is a unitary map from the linear
  span of $V(J)$ onto the linear span of $V(I)$. Changing $U$ on the orthogonal complement of $\spa(V(J))$ if necessary,
  we can therefore achieve that $U$ is a unitary map on $\bC^d$.
\end{proof}

The last result, combined with Proposition \ref{prop:graded_isomorphism}, provides a necessary
condition for the existence of an isometric isomorphism between two algebras of the form $A(\cH_I)$,
namely condition (iii) in the next theorem. This condition turns out to be sufficient as well.
We thus obtain our main result regarding the isometric isomorphism problem. It generalizes \cite[Theorem 8.2]{DRS11}.
For a bounded invertible operator $S$ between two Hilbert spaces $\cH$ and $\cK$, let
\begin{equation*}
  \Ad(S): \cB(\cH) \to \cB(\cK), \quad T \mapsto S T S^{-1},
\end{equation*}
be the induced isomorphism between $\cB(\cH)$ and $\cB(\cK)$.

\begin{thm}
  \label{thm:main_result_unitary}
  Let $\cH$ and $\cK$ be unitarily invariant complete NP-spaces, and let $I$ and $J$
  be relevant radical homogeneous ideals in $\bC[z_1,\ldots,z_d]$.
  Then the following are equivalent:
  \begin{enumerate}[label=\normalfont{(\roman*)}]
    \item $A(\cH_I)$ and $A(\cK_J)$ are isometrically isomorphic.
    \item $\Mult(\cH_I)$ and $\Mult(\cK_J)$ are isometrically isomorphic via a weak*-weak* continuous isomorphism.
    \item $\cH = \cK$ as Hilbert spaces and there is a unitary map $U$ on $\bC^d$ which maps $V(J)$ onto $V(I)$.
  \end{enumerate}
  If $\cH$ or $\cK$ is tame, then this is equivalent to
  \begin{enumerate}[label=\normalfont{(\roman*)},resume]
    \item $\Mult(\cH_I)$ and $\Mult(\cK_J)$ are isometrically isomorphic.
  \end{enumerate}
  If $U$ is a unitary map on $\bC^d$ as in (iii), then
  $U$ induces a unitary composition operator
  \begin{equation*}
    C_U: \cH_I \to \cK_J, \quad f \mapsto f \circ U,
  \end{equation*}
  and $\Ad(C_U)$ maps $A(\cH_I)$ onto $A(\cK_J)$ and $\Mult(\cH_I)$ onto $\Mult(\cK_J)$.
\end{thm}

\begin{proof}
  It follows from Proposition \ref{prop:graded_isomorphism} and Proposition \ref{prop:structure_graded_isometric} that (i) or
  (ii) implies (iii).
  Moreover, if one of the spaces is tame, then Corollary \ref{cor:phi_star_hol_iso} (b)
  shows the equivalence of (ii) and (iv).

  Conversely, suppose that (iii) holds. Since $\cH = \cK$ is unitarily invariant, $U$ induces a unitary
  composition operator $\widehat C_U \in \cB(\cH)$. If $K$ denotes the reproducing kernel of $\cH$, then
  \begin{equation*}
    (\widehat C_U)^* K(\cdot,w) = K(\cdot, U w)
  \end{equation*}
  for all $w \in Z^0(J)$ (or $w \in Z(J)$ if $\cH$ is a space on $\ol{\bB_d}$). Since $\cH \ominus I$ and
  $\cH \ominus J$ are spanned by kernel functions (see Lemma \ref{lem:nullstellensatz}),
  the implication (ii) $\Rightarrow$ (i)
  in Lemma \ref{lem:comp_operators} shows that $U$ induces a unitary composition operator $C_U: \cH_I \to \cH_J$.
  
  Then for $\varphi \in \Mult(\cH_I)$ and $f \in \cK_J$,
  \begin{equation*}
    C_U M_\varphi (C_U)^{-1} f = (\varphi \circ U) \cdot f,
  \end{equation*}
  hence $\Ad(C_U)$ maps $\Mult(\cH_I)$ into $\Mult(\cH_J)$ and $A(\cH_I)$ into $A(\cH_J)$.
  If we consider $\Ad(C_{U^{-1}})$, we see that $\Ad(C_U)$ is an isomorphism
  from $\Mult(\cH_I)$ onto $\Mult(\cH_J)$ and from $A(\cH_I)$ onto $A(\cH_J)$.
  Hence, (i) and (ii) hold, and the additional assertion is proven.
\end{proof}

For algebraic isomorphisms, the situation is more difficult. Proposition \ref{prop:graded_isomorphism} and
Proposition \ref{prop:structure_graded_algebraic} show that if $A(\cH_I)$ and $A(\cK_J)$ are algebraically
isomorphic, then $\cH = \cK$ as vector spaces and there exists an invertible linear map $A$
on $\bC^d$ which maps $V(J)$ isometrically onto $V(I)$. Note that here, $A$ will in general only be
isometric on $V(J)$ and not on all of $\bC^d$. In this case,
it is no longer obvious that $A$ induces an algebraic isomorphism between $A(\cH_I)$ and $A(\cK_J)$.
The reason why the proof of Theorem \ref{thm:main_result_unitary} does not carry over is that now, $A$ does not
induce a composition operator on all of $\cH$. In the case of $\cH = H^2_d$,
this problem already appeared in \cite{DRS11}, where
it was solved under additional assumptions on the geometry of $V(J)$. The general case
was settled in \cite{Hartz12}. Fortunately, we can use a crucial reduction from \cite{DRS11} and the main result of \cite{Hartz12} in our setting as well.

\begin{lem}
  \label{lem:composition_bounded}
  Let $\cH$ be a reproducing kernel Hilbert space on $\bB_d$ (or on $\ol{\bB_d}$) with a reproducing kernel
  of the form
  \begin{equation*}
    K(z,w) = \sum_{n=0}^\infty a_n \langle z,w \rangle^n,
  \end{equation*}
  where $a_n > 0$ for all $n \in \bN$.
  Suppose that $I \subsetneq \bC[z_1,\ldots,z_d]$ is a radical homogeneous ideal.
  If $A$ is a linear map on $\bC^d$ which is isometric on $V(I)$, then there exists a bounded operator
  \begin{equation*}
    T_A: \cH \ominus I \to \cH \quad \text{such that} \quad T_A K(\cdot,w) = K(\cdot, Aw)
  \end{equation*}
  for all $w \in Z^0(I)$ (respectively
  $w \in Z(I)$ if $\cH$ is a space on $\ol{\bB_d}$).
\end{lem}

\begin{proof}
  The first part of the proof is a straightforward adaptation of the proof of \cite[Proposition 2.5]{Hartz12}.
  Let
  \begin{equation*}
    V(I) = V_1 \cup \ldots \cup V_r
  \end{equation*}
  be the decomposition of $V(I)$ into irreducible homogeneous varieties and let $\widehat I$ be the vanishing ideal
  of $\spa{V_1} \cup \ldots \cup \spa{V_r}$. Then $\widehat I \subset I$ by Hilbert's Nullstellensatz.
  By \cite[Proposition 7.6]{DRS11}, the linear map $A$ is isometric on $V(\widehat I)$, so we may assume without
  loss of generality that
  \begin{equation*}
    V(I) = V_1 \cup \ldots \cup V_r
  \end{equation*}
  is a union of subspaces, so
  \begin{equation*}
    I = I_1 \cap \ldots \cap I_r,  
  \end{equation*}
  where $I_j$ is the vanishing ideal of $V_j$. Then by a variant of \cite[Lemma 2.3]{Hartz12},
  \begin{equation*}
    \cH \ominus I = \ol{ \cH \ominus I_1 + \ldots + \cH \ominus I_r}.
  \end{equation*}
  We define $T_A$ on the dense subspace of $\cH \ominus I$ consisting of polynomials by
  \begin{equation*}
    T_A p = p \circ A^*.
  \end{equation*}
  Then $T_A \langle \cdot,w \rangle^n = \langle \cdot, A w \rangle^n$ for all $w \in V(I)$.
  Using the fact that $\cH$ is unitarily invariant and that $A$ is isometric on each $V_j$, it is not
  hard to see that
  the map $T_A$ is isometric on $\cH \ominus I_j$ for every $j$ (cf. \cite[Lemma 2.2]{Hartz12}).
  As in the proof of \cite[Proposition 2.5]{Hartz12}, we may therefore
  finish the proof by showing that the algebraic sum
  \begin{equation*}
    \cH \ominus I_1 + \ldots + \cH \ominus I_r
  \end{equation*}
  is closed.
  
  If $\cH = H^2_d$, this is the main result of \cite{Hartz12}. More generally, in the present setting,
  there exists a unique unitary operator
  \begin{equation*}
    U: H^2_d \to \cH \quad \text{ with } \quad U(p) = \sqrt{a_n} p
  \end{equation*}
  for every homogeneous polynomial $p$ of degree $n$ (see, for example \cite[Proposition 4.1]{GHX04}).
  Since each $I_j$ is a homogeneous ideal, $U(I_j) = I_j$
  and hence $U(H^2_d \ominus I_j) = \cH \ominus I_j$ for $1 \le j \le r$. Consequently,
  closedness of the algebraic sum
  \begin{equation*}
    \cH \ominus I_1 + \ldots + \cH \ominus I_r
  \end{equation*}
  follows from the special case where $\cH = H^2_d$.
\end{proof}

With the help of Lemma \ref{lem:composition_bounded}, we can now prove the main result regarding
the algebraic isomorphism problem. It generalizes \cite[Theorem 8.5]{DRS11} and \cite[Theorem 5.9]{Hartz12}.
Observe that since the algebras $A(\cH_I)$ and $\Mult(\cH_I)$ are semi-simple, algebraic isomorphisms are
automatically norm continuous.

\begin{thm}
  \label{thm:main_result_algebraic}
  Let $\cH$ and $\cK$ be unitarily invariant complete NP-spaces, and let $I$ and $J$
  be relevant radical homogeneous ideals in $\bC[z_1,\ldots,z_d]$.
  Then the following are equivalent:
  \begin{enumerate}[label=\normalfont{(\roman*)}]
    \item $A(\cH_I)$ and $A(\cK_J)$ are algebraically isomorphic.
    \item $\Mult(\cH_I)$ and $\Mult(\cK_J)$ are isomorphic via a weak*-weak* continuous isomorphism.
    \item $\cH = \cK$ as vector spaces and there is an invertible linear map $A$ on $\bC^d$ which maps $V(J)$ isometrically onto $V(I)$.
  \end{enumerate}
  If $\cH$ or $\cK$ is tame, then this is equivalent to
  \begin{enumerate}[label=\normalfont{(\roman*)},resume]
    \item $\Mult(\cH_I)$ and $\Mult(\cK_J)$ are algebraically isomorphic.
  \end{enumerate}
  If $A$ is an invertible linear map on $\bC^d$ as in (iii), then $A$ induces a bounded invertible
  composition operator
  \begin{equation*}
    C_A: \cH_I \to \cK_J, \quad f \mapsto f \circ A,
  \end{equation*}
  and $\Ad(C_A)$ maps $A(\cH_I)$ onto $A(\cK_J)$ and $\Mult(\cH_I)$ onto $\Mult(\cK_J)$.
\end{thm}

\begin{proof}
  It follows from Proposition \ref{prop:graded_isomorphism} and Proposition \ref{prop:structure_graded_algebraic} that (i) or
  (ii) implies (iii).
  Moreover, if one of the spaces is tame, then Corollary \ref{cor:phi_star_hol_iso} (b) once again
  shows the equivalence of (ii) and (iv).

  Assume that (iii) holds. Since $\cH = \cK$ as vector spaces, the formal identity
  \begin{equation*}
    E: \cH \to \cK, \quad f \mapsto f,
  \end{equation*}
  is bounded and bounded below by the closed graph theorem. By Lemma \ref{lem:composition_bounded}, there exists
  a bounded operator
  \begin{equation*}
    T: \cK \ominus J \to \cK \quad \text{ such that } \quad
    T K_{\cK}(\cdot,w) = K_{\cK}(\cdot, A w)
  \end{equation*}
  for all $w \in Z^0(J)$ (respectively $w \in Z(J)$). Let $T_A = E^* T$.
  Then
  \begin{equation*}
    T_A (K_{\cK}(\cdot,w)) = E^* K_{\cK}(\cdot,A w) =  K_{\cH} (\cdot, A w)
  \end{equation*}
  for all $w$, from which we deduce with the help of Lemma \ref{lem:nullstellensatz}
  that $T_A$ maps $\cK \ominus J$ into $\cH \ominus I$.
  Replacing $A$ with $A^{-1}$, we see that $T_A \in \cB(\cK \ominus J, \cH \ominus I)$ is invertible.
  It now follows from Lemma \ref{lem:comp_operators} that $A$ induces a bounded invertible composition
  operator $C_A: \cH_I \to \cK_J$.
  As in the proof of Theorem \ref{thm:main_result_unitary}, we see that $\Ad(C_A)$ is the desired isomorphism.
\end{proof}

Just as in \cite{DRS11}, we obtain from the geometric rigidity result \cite[Proposition 7.6]{DRS11} a rigidity
result for our algebras. It generalizes \cite[Theorem 8.7]{DRS11}. The author is grateful to the anonymous referee
for pointing out this corollary.

\begin{cor}
  Let $\cH$ be a unitarily invariant complete NP-space and let $I$ and $J$ be relevant radical homogeneous
  ideals in $\bC[z_1,\ldots,z_d]$. Suppose that $V(I)$ or $V(J)$ is irreducible.
  \begin{enumerate}[label=\normalfont{(\alph*)}]
    \item If $A(\cH_I)$ and $A(\cH_J)$ are algebraically isomorphic, then $A(\cH_I)$ and $A(\cH_J)$
      are unitarily equivalent.
    \item If $\Mult(\cH_I)$ and $\Mult(\cH_J)$ are isomorphic via a weak*-weak* continuous
      isomorphism, then $\Mult(\cH_I)$ and $\Mult(\cH_J)$ are unitarily equivalent.
    \item If $\cH$ or $\cK$ is tame and $\Mult(\cH_I)$ and $\Mult(\cH_J)$ are algebraically
      isomorphic, then $\Mult(\cH_I)$ and $\Mult(\cH_J)$ are unitarily equivalent.
  \end{enumerate}
\end{cor}

\begin{proof}
  In each case, Theorem \ref{thm:main_result_algebraic} shows that there exists an invertible
  linear map $A$ on $\bC^d$ which maps $V(J)$ isometrically onto $V(I)$. In particular,
  $V(I)$ and $V(J)$ are both irreducible. Proposition 7.6 in \cite{DRS11} implies
  that $A$ is isometric on the linear span of $V(J)$, and hence can be chosen to be unitary.
  All assertions now follow from Theorem \ref{thm:main_result_unitary}.
\end{proof}

Let us apply Theorems \ref{thm:main_result_unitary} and \ref{thm:main_result_algebraic} in the setting
where $\cH$ and $\cK$ are given by log-convex sequences (see part (b) of Remark \ref{rem:unitarily_invariant}).
This includes in particular the spaces $\cH_s(\bB_d)$ and $\cH_s(\ol{\bB_d})$ of Example \ref{exa:unit_invariant}.
If $\boldsymbol{a} = (a_n)_n$ is a sequence of positive real numbers such that the series
\begin{equation*}
  \sum_{n=0}^\infty a_n z^n
\end{equation*}
has radius of convergence $1$, we write $\cH(\boldsymbol{a})$ for the reproducing kernel
Hilbert space with reproducing kernel
\begin{equation*}
  K(z,w) = \sum_{n=0}^\infty a_n \langle z,w \rangle^n.
\end{equation*}
If $\sum_{n=0}^\infty a_n = \infty$, this is a reproducing kernel Hilbert space on $\bB_d$,
and if $\sum_{n=0}^\infty a_n < \infty$, this a space on $\ol{\bB_d}$.

\begin{cor}
  \label{cor:log_convex_spaces}
  Let $\boldsymbol{a} = (a_n)$ and $\boldsymbol{a'}= (a_n')$ be two log-convex sequences of positive real numbers such that
  \begin{equation*}
    a_0 = 1 = a_0'
  \end{equation*}
  and
  \begin{equation*}
    \lim_{n \to \infty} \frac{a_n}{a_{n+1}} = 1
    = \lim_{n \to \infty} \frac{a'_n}{a'_{n+1}}.
  \end{equation*}
  Let $\cH = \cH(\boldsymbol{a})$ and $\cK = \cH(\boldsymbol{a'})$. Let $I,J \subset \bC[z_1,\ldots,z_d]$ be two
  relevant radical homogeneous ideals of polynomials. Then the following are equivalent:
  \begin{enumerate}[label=\normalfont{(\roman*)}]
    \item $A(\cH_I)$ and $A(\cK_J)$ are isometrically isomorphic.
    \item $\Mult(\cH_I)$ and $\Mult(\cK_J)$ are isometrically isomorphic.
    \item $a_n = a_n'$ for all $n \in \bN$ and there exists a unitary map $U$ on $\bC^d$ which
      maps $V(J)$ onto $V(I)$.
  \end{enumerate}
  Moreover, the following assertions are equivalent as well:
  \begin{enumerate}[label=\normalfont{(\roman*)}]
    \item $A(\cH_I)$ and $A(\cK_J)$ are algebraically isomorphic.
    \item $\Mult(\cH_I)$ and $\Mult(\cK_J)$ are algebraically isomorphic.
    \item There exist constants $C_1,C_2 > 0$ such that
      \begin{equation*}
        C_1 \le \frac{a_n}{a_n'} \le C_2
      \end{equation*}
      for all $n \in \bN$ and there exists an invertible linear map $A$ on $\bC^d$
      which maps $V(J)$ isometrically onto $V(I)$.
  \end{enumerate}
\end{cor}

\begin{proof}
  The assumptions on $\boldsymbol{a}$ and $\boldsymbol{a'}$ imply that $\cH$ and $\cK$ are unitarily
  invariant complete NP-spaces on $\bB_d$ or on $\ol{\bB_d}$ (see the beginning of Section \ref{sec:unit_inv}).
  Proposition \ref{prop:tame} (c) shows that $\cH$ and $\cK$ are tame. The first set of equivalences
  is now an immediate consequence of Theorem \ref{thm:main_result_unitary}. To prove the second
  set of equivalences, in light of Theorem \ref{thm:main_result_algebraic},
  it suffices to show that $\cH = \cK$ as vector spaces if and only if there exist constants
  $C_1 , C_2 > 0$ such that
  \begin{equation*}
    C_1 \le \frac{a_n}{a_n'} \le C_2
  \end{equation*}
  for all $n \in \bN$. To this end, observe that if $\cH = \cK$ as vector spaces,
  then the formal identity $E: \cH \to \cK, f \mapsto f$, is bounded and invertible
  by the closed graph theorem, hence the existence of the constants follows from
  the description of the norm in part (d) of Remark \ref{rem:unitarily_invariant}.
  The other implication follows from part (d) of Remark \ref{rem:unitarily_invariant} as well.
\end{proof}

We finish this article by considering the last result about algebraic isomorphism from
the point of view adopted in \cite{DRS15} and \cite{DHS15}.
That is, we will identify a multiplier algebra $\Mult(\cH_I)$
with an algebra of the form $\cM_V = \Mult(H^2_\infty \big|_V)$ for
a suitable variety $V \subset \bB_\infty$.

We first show that most of our examples of unitarily invariant complete NP-spaces
cannot be embedded into a finite dimensional ball.
More generally, let $\cH$ be an irreducible complete Nevanlinna-Pick space on $\bB_d$
with reproducing kernel of the form
\begin{equation*}
  K(z,w) = \sum_{n=0}^\infty a_n \langle z,w \rangle^n,
\end{equation*} 
where $a_0 = 1$.
Recall that an embedding for $\cH$ is an injective function $j: \bB_d \to \bB_m$ for some $m \in \bN \cup \{\infty\}$
such that
\begin{equation*}
  \langle j(z), j(w) \rangle  =
  1 - \frac{1}{
    \sum_{n=0}^\infty a_n \langle z,w \rangle^n}
\end{equation*}
for all $z,w \in \bB_d$. By Lemma \ref{lem:NP_a_b},
there is a sequence $(c_n)$ of non-negative real numbers such that
\begin{equation*}
  1 - \frac{1}{\sum_{n=0}^\infty a_n \langle z,w \rangle^n} = \sum_{n=1}^{\infty} c_n \langle z,w \rangle^n
\end{equation*}
for all $z,w \in \bB_d$.
Since
\begin{equation*}
  \langle z,w \rangle^n = \sum_{|\alpha|=n} \binom{n}{\alpha} z^{\alpha} \ol{w}^\alpha
  = \langle \psi_n(z), \psi_n(w) \rangle,
\end{equation*}
where
\begin{equation*}
  \psi_n: \bC^d \to \bC^{\binom{n+d-1}{n}}, \quad z \mapsto \Bigg(\sqrt{\binom{n}{\alpha}} z^\alpha \Bigg)_{|\alpha| = n},
\end{equation*}
an embedding $j$ for $\cH$ can be explicitly constructed by setting
\begin{equation*}
  j(z) = (\sqrt{c_1} \psi_1(z), \sqrt{c_2} \psi_2(z), \sqrt{c_3} \psi_3(z), \ldots).
\end{equation*}
Using the fact that $\sum_{n=1}^\infty c_n \le 1$, it is not hard to see that $j$ is an analytic
map from $\bB_d$ into $\bB_m$ which extends to a norm continuous map from $\ol{\bB_d}$ to $\ol{\bB_m}$.
If $d=1$, these embeddings are simply the embeddings considered in Section 7 and 8 of \cite{DHS15}.

In particular, we see that if only finitely many of the $c_n$ are non-zero, then $\cH$ admits an embedding
into a finite dimensional ball, that is, there exists $m < \infty$ and an embedding
$j: \bB_d \to \bB_m$ for $\cH$. In fact, this property characterizes unitarily invariant spaces which admit
an embedding into a finite dimensional ball.

\begin{prop}
  \label{prop:finite_embedding}
  Let $\cH$ be an irreducible complete Nevanlinna-Pick space on $\bB_d$ 
  with reproducing kernel of the form
  \begin{equation*}
    K(z,w) = \sum_{n=0}^\infty a_n \langle z,w \rangle^n,
  \end{equation*}
  where $a_0 = 1$.
  Then $\cH$ admits an embedding into a finite dimensional ball if and only if the analytic
  function $f$ on $\bD$ defined by
  \begin{equation*}
    f(t) = \frac{1}{\sum_{n=0}^\infty a_n t^n}
  \end{equation*}
  is a polynomial.
\end{prop}

\begin{proof}
  With notation as in the discussion preceding the proposition, observe that
  \begin{equation*}
    1 - f(t) = \sum_{n=1}^\infty c_n t^n
  \end{equation*}
  for all $t \in \bD$. Hence, $f$ is indeed an analytic function on $\bD$, and $f$ is a polynomial
  if and only if all but finitely many $c_n$ are zero.

  For the proof of the remaining implication, suppose that $\cH$ admits an embedding $j$ into $\bB_m$
  for some $m < \infty$. From
  \begin{equation*}
    1 - \frac{1}{K(z,w)} = \langle j(z), j(w) \rangle_{\bC^m}, 
  \end{equation*}
  we deduce that the rank of the kernel $L = 1 - 1/K$ is at most $m$ in the sense that
  for any finite collection of points $\{z_1,\ldots,z_n\}$, the matrix
  \begin{equation*}
    \Big( L(z_i,z_j) \Big)_{i,j=1}^n
  \end{equation*}
  has rank at most $m$. Let $\cK$ denote the reproducing kernel Hilbert space on $\bB_d$ with reproducing kernel $L$.
  Since
  \begin{equation*}
    \langle L(\cdot,w), L(\cdot,z) \rangle_{\cK} = L(z,w),
  \end{equation*}
  and since $\cK$ is spanned by the kernel functions $L(\cdot,z)$ for $z \in \bB_d$, it follows that
  the dimension of $\cK$ is at most $m$. However, $L$ also admits the representation
  \begin{equation*}
    L(z,w) = \sum_{n=1}^\infty c_n \langle z,w \rangle^n,
  \end{equation*}
  hence for every $n \in \bN$ with $c_n \neq 0$, the space $\cK$ contains the monomial $z_1^n$, and different
  monomials are orthogonal. Consequently, $c_n = 0$ for all but finitely many $n$, so $f$ is a polynomial.
\end{proof}

As a consequence, we see that all spaces in Example \ref{exa:unit_invariant} besides
the Drury-Arveson space
do not admit an embedding into a finite dimensional ball.

\begin{cor}
  Let $\cH$ be a unitarily invariant complete NP-space
  with reproducing kernel of the form
  \begin{equation*}
    K(z,w) = \sum_{n=0}^\infty a_n \langle z,w \rangle^n.
  \end{equation*}
  If $\cH$ admits an embedding into a finite dimensional ball, then
  the sequence $(a_n)$ converges to a positive real number, and hence $\cH = H^2_d$ as vector spaces.
  In particular, the space $\cH_s(\bB_d)$ for $-1 \le s < 0$,
  the space $\cH_s(\ol{\bB_d})$ for $s < -1$, and the space $\cK_\alpha$ for $0 < \alpha < 1$ do not
  admit an embedding into a finite dimensional ball.
\end{cor}

\begin{proof}
  Assume that $\cH$ admits an embedding into a finite dimensional ball.
  Proposition \ref{prop:finite_embedding} implies that there exists $N \in \bN$ and non-negative
  real numbers $c_1,\ldots,c_N$ such that
  \begin{equation*}
    \sum_{n=0}^\infty a_n t^n = \frac{1}{1 - \sum_{n=1}^N c_n t^n}.
  \end{equation*}
  Observe that $c_1 > 0$ as $a_1 > 0$.
  Since the power series on the left-hand side has radius of convergence $1$, this rational
  function in $t$ has a pole on $\partial \bD$. Because $a_n \ge 0$ for all $n \in \bN$, this is only
  possible if $\sum_{n=0}^\infty a_n = \infty$, from which we deduce that $\sum_{n=1}^N c_n = 1$.
  Let $r = \sum_{n=1}^N n c_n$. In this setting,
  the Erd\H{o}s-Feller-Pollard theorem (see \cite[Chapter XIII, Section 11]{Feller68})
  implies that $\lim_{n \to \infty} a_n = 1/ r > 0$. The remaining assertions are now obvious.
\end{proof}

Suppose now that $\boldsymbol{a}$ is a sequence as in Corollary \ref{cor:log_convex_spaces},
and assume first
that we are in the case where $\sum_{n=0}^{\infty} a_n = \infty$.
Let $j_{\boldsymbol{a}}: \bB_d \to \bB_\infty$ denote the embedding
for $\cH(\boldsymbol{a})$ which was constructed above.
Note that $c_1 \neq 0$ as $a_1 \neq 0$, that the coordinates
$j_{\boldsymbol{a}}$ are polynomials (in fact monomials), and that
the first $d$ coordinates are given by $(\sqrt{c_1} z_1,\ldots, \sqrt{c_1} z_d)$.
In particular, $j_{\boldsymbol{a}}: \bB_d \to V_{\boldsymbol{a}}$ is invertible, where $V_{\boldsymbol{a}}$ denotes
the range of $j_{\boldsymbol{a}}$. An inverse of $j_{\boldsymbol{a}}$ is
given by
\begin{equation}
  \label{eqn:phi_a_inverse}
  j_{\boldsymbol{a}}^{-1} (z) = \frac{1}{\sqrt{c_1}} (z_1,\ldots,z_d)  
\end{equation}
for $z \in V_{\boldsymbol{a}}$.

Suppose now that $I \subset \bC[z_1,\ldots,z_d]$ is a relevant radical homogeneous ideal. Then
the restriction of $j_{\boldsymbol{a}}$ to $Z^0(I)$ is an embedding for $\cH(\boldsymbol{a})_I$.
Since $\cH(\boldsymbol{a})_I$ is algebraically consistent, the image
\begin{equation*}
  V_{\boldsymbol{a},I} = j_{\boldsymbol{a}}(Z^0(I)) \subset \bB_\infty
\end{equation*}
is a variety by Proposition \ref{prop:alg_consistent_equiv}. Moreover, $j_a$ maps $Z^0(I)$ biholomorphically
onto $V_{\boldsymbol{a},I}$.
This discussion also applies to the case where $\sum_{n=0}^{\infty} a_n < \infty$ by simply replacing $\bB_d$ with $\ol{\bB_d}$
and $Z^0(I)$ with $Z(I)$ above.
In this case, $\sum_{n=1}^\infty c_n < 1$ and $j_{\boldsymbol{a}}$ maps $Z(I)$ homeomorphically
onto $V_{\boldsymbol{a},I}$ and $Z^0(I)$ biholomorphically onto its image.

Let $m \in \bN \cup \{\infty\}$. For a variety $V \subset \bB_m$, let $\cM_V = \Mult(H^2_m \big|_{V})$.
Following \cite{DHS15}, we say that
two varieties $V,W \subset \bB_m$ are \emph{multiplier biholomorphic} if there
exists a homeomorphism $F: V \to W$ such that every coordinate of $F$ is in $\cM_V$ and
every coordinate of $F^{-1}$ is in $\cM_W$. If $m < \infty$, then such a map is automatically
a biholomorphism in the usual sense. This definition
is motivated by \cite[Theorem 5.6]{DRS15},
which states that if $\cM_V$ and $\cM_W$ are algebraically isomorphic, then
$V$ and $W$ are multiplier biholomorphic, provided that $m < \infty$ and $V$ and $W$ satisfy
some mild geometric assumptions.
Moreover, there are examples of two discs in $\bB_2$ which are biholomorphic,
but not multiplier biholomorphic (see \cite[Section 5]{DHS15}).

However,
already the results in Section 7 and 8 of \cite{DHS15} show that there are multiplier biholomorphic
discs in $\bB_\infty$ whose multiplier algebras are not isomorphic.
It turns out that for the varieties $V_{\boldsymbol{a},I}$ constructed above, of which
the discs from Section 7 and 8 of \cite{DHS15} are a special case,
the multiplier biholomorphism classes
only depend on the ideal $I$ and on summability of the sequence $\boldsymbol{a}$. They do not detect
any other properties of the sequence $\boldsymbol{a}$.
In light of Corollary \ref{cor:log_convex_spaces}, this means that the relation of multiplier biholomorphism fails rather dramatically
at distinguishing the isomorphism classes of the algebras $\cM_{V_{\boldsymbol{a},I}}$.

\begin{prop}
  \label{prop:bihol_vs_mult_bihol}
  Let $\boldsymbol{a} = (a_n)$ and $\boldsymbol{a'} = (a_n')$ be sequences as in Corollary \ref{cor:log_convex_spaces},
  and let $I,J \subset \bC[z_1,\ldots,z_d]$ be relevant radical homogeneous ideals. Let
  $V_{\boldsymbol{a},I}$ and $V_{\boldsymbol{a'},J}$ be the varieties defined above. Then the following are equivalent:
  \begin{enumerate}[label=\normalfont{(\roman*)}]
    \item $V_{\boldsymbol{a},I}$ and $V_{\boldsymbol{a'},J}$ are multiplier biholomorphic.
    \item The sequences $\boldsymbol{a}$ and $\boldsymbol{a'}$ are either both summable or both not summable, 
      and there exists an invertible linear map $A$ on $\bC^d$ which maps $V(J)$ isometrically onto $V(I)$.
  \end{enumerate}
\end{prop}

\begin{proof}

  (i) $\Rightarrow$ (ii) Observe that $V_{\boldsymbol{a},I}$ is homeomorphic
  to $Z^0(I)$ if $\boldsymbol{a}$ is not
  summable and homeomorphic to $Z(I)$ if $\boldsymbol{a}$ is summable. Since $Z(I)$ is compact and $Z^0(I)$ is not,
  it follows that if (i) holds, then $\boldsymbol{a}$ and $\boldsymbol{a'}$ are either both summable or both not summable. In the non-summable
  case, $Z^0(I)$ and $Z^0(J)$ are biholomorphic. In the summable case, there is a homeomorphism
  $F:Z(I) \to Z(J)$ which is analytic on $Z^0(I)$ and whose inverse is analytic on $Z^0(J)$.
  Then Lemma \ref{lem:maximum_modulus} implies that $Z^0(I)$ and $Z^0(J)$ are biholomorphic. Finally, an
  application of Lemma \ref{lem:bihol_turn}
  and Proposition \ref{prop:bihol_0_linear} shows that there exists an invertible linear map $A$
  on $\bC^d$ which maps $V(J)$ isometrically onto $V(I)$.

  (ii) $\Rightarrow$ (i) Let us first assume that $\boldsymbol{a}$ and $\boldsymbol{a'}$ are both not summable, and let
  $j_{\boldsymbol{a}}: \bB_d \to V_{\boldsymbol{a}}$ and $j_{\boldsymbol{a'}}:\bB_d \to V_{\boldsymbol{a'}}$ be the embeddings constructed earlier.
  Then $F = j_{\boldsymbol{a}} \circ A \circ j_{\boldsymbol{a'}}^{-1}$ maps $V_{\boldsymbol{a'},J}$ homeomorphically
  onto $V_{\boldsymbol{a},I}$.
  From Equation \eqref{eqn:phi_a_inverse} and the fact that the coordinates of $j_{\boldsymbol{a}}$ are polynomials,
  we deduce that the coordinates of $F$ are polynomials in $z_1,\ldots,z_d$. Similarly, the coordinates
  of $F^{-1}$ are polynomials in $z_1,\ldots,z_d$, hence $F$ is a multiplier biholomorphism.

  After replacing $\bB_d$ with $\ol{\bB_d}$, the same argument applies in the situation where $\boldsymbol{a}$ and $\boldsymbol{a'}$ are both summable. Hence, the proof is complete.
\end{proof}

\begin{ack}
  The author would like to thank his advisor, Ken Davidson, for his advice and support. Moreover, he is grateful
  to the anonymous referee for his or her careful reading and for his or her very valuable comments which improved this article.
\end{ack}

\bibliographystyle{amsplain}
\bibliography{literature.bib}

\end{document}